\documentclass[11pt,a4paper]{article}

\usepackage[utf8]{inputenc}
\usepackage[T1]{fontenc}
\usepackage[english]{babel}
\usepackage{amsmath,amssymb,amsthm,mathtools}
\usepackage{booktabs}
\usepackage{longtable}
\usepackage{array}
\usepackage{enumitem}
\usepackage{geometry}
\usepackage{xcolor}
\usepackage[hidelinks]{hyperref}

\setlist{nosep}
\newtheorem{theorem}{Theorem}[section]
\newtheorem{lemma}[theorem]{Lemma}
\newtheorem{proposition}[theorem]{Proposition}
\newtheorem{corollary}[theorem]{Corollary}
\theoremstyle{definition}
\newtheorem{definition}[theorem]{Definition}
\newtheorem{assumption}[theorem]{Assumption}
\newtheorem{model}[theorem]{Model Law}
\newtheorem{conjecture}[theorem]{Conjecture}
\theoremstyle{remark}
\newtheorem{remark}[theorem]{Remark}

\newcommand{\RR}{\mathbb{R}}
\newcommand{\ZZ}{\mathbb{Z}}
\newcommand{\NN}{\mathbb{N}}
\newcommand{\EE}{\mathcal{E}}

\newcommand{\FF}{\mathcal{F}}
\newcommand{\GG}{\Gamma}
\newcommand{\TT}{\mathcal{T}}
\newcommand{\MM}{\mathcal{M}}
\newcommand{\dd}{\mathrm{d}}
\newcommand{\id}{\mathrm{id}}
\newcommand{\codepath}[1]{\nolinkurl{#1}}
\newcommand{\norm}[1]{\left\lVert #1 \right\rVert}
\DeclareMathOperator{\rank}{rank}

\DeclareMathOperator{\Hol}{Hol}

\title{\textbf{Fractal Algebraic Topology of Semantic Computation}\\
\large A Peer-Review-Oriented Formalization of the SSTD/BrainiaK Concept Bundle}
\author{Jean-Philippe Garnier\\
\small BrainiaK Research}
\date{June 2026}

\begin{document}
\maketitle

\begin{abstract}
This manuscript develops material from the internal French research notes
\emph{Traite de Topologie Algebrique Fractale} into an academic
manuscript.
The editorial rule is strict: implementation names are not used as
mathematical proofs, analogies are not promoted to theorems, and every
formal result is either proved from explicit assumptions or downgraded to
a model law, conjecture, or empirical claim.  The central object is
$T^n$, a finite heterogeneous concept container formalized as a section
of a product bundle whose slots include an empirical sensorimotor base
$\RR^{14}$, grammatical fibres, polarity, intensity, vision and audition
slots, an SSTD spectral slot, a refined compositional fibre, and auxiliary
tool/metric/axis/hint slots.  We prove elementary structural results
about product-bundle representation, heterogeneous GCM metrics, and
continuity of componentwise operations.  We then give conditional
results for Frobenius-inspired crystal composition, Gamma/CNS
curvature-Hopf modelling, Kalman convergence, SSTD bundle morphisms, and
SpiderR flat-connection idealizations.  Each conditional result includes
its assumptions, proof status, implementation correspondence, and the
boundary between mathematics, model assumptions, and empirical evidence.
\end{abstract}


\section*{Acknowledgements}

The author thanks Amir Shashin Hedayat for careful reading, editorial
corrections, and comments on presentation.  This acknowledgement does
not imply co-authorship or responsibility for the mathematical claims of
the manuscript.

\section{Introduction}

The source treatise proposes a unifying language for the BrainiaK
architecture: concept vectors $T^n$, the Gamma/CNS dynamical system, the
SSTD sedimentation codec, accumulation of semantic crystals, and
Penrose/Frobenius-inspired composition.  Its original form was
exploratory and system-facing.  This paper turns that material into a
peer-review-oriented mathematical manuscript.

The guiding constraint is methodological.  A result is not a theorem
because the implementation suggests it, nor because an analogy is
plausible.  It is a theorem only when it follows from definitions and
assumptions stated in the paper.  Implementation files are cited only as
correspondence between the formal objects and BrainiaK code surfaces.
Empirical results are reported only as empirical claims and require
traceable datasets, metrics, and scripts before they can be treated as
validated evidence.

\subsection{Contributions}

This manuscript makes five contributions.
\begin{enumerate}
  \item It gives a precise product-bundle formalization of $T^n$.
  \item It defines a heterogeneous GCM metric and proves elementary
  metric and topological facts under explicit completeness assumptions.
  \item It audits the six central claims of the French treatise and
  classifies each as theorem, proposition, model law, conjecture, or
  empirical claim.
  \item It corrects the status of the Frobenius, Hopf, SSTD-complexity,
  and SpiderR flatness claims so that no unsupported result remains.
  \item It provides implementation correspondence without confusing code
  existence with mathematical proof.
\end{enumerate}

\subsection{Canonical Notation and Dimensional Hygiene}

The current formal base is the empirical semantic space
\[
  \RR^{14}=\RR^3_{\mathrm{emo}}\oplus
  \RR^6_{\mathrm{per}}\oplus \RR^5_{\mathrm{mot}},
\]
where the three socles encode affective, perceptual, and motor
coordinates.  Other spaces are not competing foundations.  They are
slots or extensions:
\[
  \RR^6_{\mathrm{vis}},\quad
  \RR^6_{\mathrm{aud}},\quad
  S^5,\quad
  \RR^{3072}_{\mathrm{SSTD}},\quad
  \FF_{\mathrm{refined}}.
\]
Earlier $\RR^{24}$ descriptions are treated as a special extended
direct-sum presentation and are not used as the canonical base in this
paper.  When $S^5$ appears in older notes, it is treated as a
normalization manifold or optional submanifold associated with a
six-coordinate sensory representation, not as a replacement for the
canonical $\RR^{14}$ base.

\subsection{Controlled Notation Register}

The following notation is used throughout the manuscript.  The table is
part of the dimensional hygiene of the paper: symbols in different rows
are not silently identified.

\begin{longtable}{>{\raggedright\arraybackslash}p{0.18\textwidth}
>{\raggedright\arraybackslash}p{0.24\textwidth}
>{\raggedright\arraybackslash}p{0.44\textwidth}}
\caption{Controlled notation register.}
\label{tab:notation-register}\\
\toprule
\textbf{Symbol} & \textbf{Type} & \textbf{Role in the manuscript}\\
\midrule
\endfirsthead
\toprule
\textbf{Symbol} & \textbf{Type} & \textbf{Role in the manuscript}\\
\midrule
\endhead
$B_{14}$ or $\RR^{14}$
& Real vector space
& Canonical empirical semantic base
$\RR^3_{\mathrm{emo}}\oplus\RR^6_{\mathrm{per}}\oplus
\RR^5_{\mathrm{mot}}$.\\
\addlinespace
$\RR^{24}$
& Extended presentation
& Older direct-sum notation interpreted, when used, as
$B_{14}\oplus U$ with $\dim U=10$; not the canonical base.\\
\addlinespace
$S^5$
& Smooth submanifold of $\RR^6$
& Optional normalization constraint for a six-coordinate sensory slot;
not a replacement for $B_{14}$.\\
\addlinespace
$E_k$
& Slot space
& The $k$th component space of the heterogeneous product
$\EE=\prod_{k=1}^{12}E_k$.\\
\addlinespace
$E_7=\RR^{3072}_{\mathrm{SSTD}}$
& SSTD slot
& Spectral text slot; later identified with three aggregated
$1024$-dimensional views.\\
\addlinespace
$T^n$
& Concept tuple / section notation
& Historical source notation formalized here as a section of a finite
product bundle, equivalently an element of $\EE$.\\
\addlinespace
$d_{\mathrm{GCM}}$
& Weighted product metric
& Heterogeneous metric on $\EE$, built from positive slot weights and
slot metrics.\\
\addlinespace
$\Gamma(t)$
& Matrix-valued trajectory
& Gamma/CNS diagnostic object used only under explicit spectral and
curvature assumptions.\\
\addlinespace
$R(t)$
& Scalar curvature proxy
& Model quantity, typically
$R(t)=\alpha N(t)+\beta\sigma_\Gamma(t)$; not derived from Gamma
dynamics unless separately proved.\\
\addlinespace
$\lambda_{\max}(t)$
& Spectral scalar
& Leading spectral value in the curvature--spectrum model law; Hopf
claims require additional dynamical data.\\
\addlinespace
$\kappa$
& Positive model parameter
& Curvature-law scale parameter in
$\lambda_{\max}=R/(R+\kappa)$.\\
\addlinespace
$F,H,Q,R$
& Kalman system data
& State transition, observation map, process covariance, and observation
covariance in the linear Gaussian filtering section.\\
\addlinespace
$A,\Omega,\Hol$
& Connection data
& Connection one-form, curvature, and holonomy in the SpiderR
idealization; deployed operators require separate commutator checks.\\
\bottomrule
\end{longtable}

\begin{proposition}[Notation register prevents dimensional conflation]
If the notation register is respected, then no statement about
$B_{14}$, $\RR^{24}$, $S^5$, $\RR^{3072}_{\mathrm{SSTD}}$, or $T^n$
can be transferred to another one of these objects without an explicit
map or inclusion.
\end{proposition}

\begin{proof}
The register assigns each symbol a type and role.  A valid transfer of a
statement between typed mathematical objects requires a function,
inclusion, projection, quotient map, or other explicitly stated
relation.  Since the register declares the listed objects to have
different types or roles, respecting the register forbids treating them
as identical by notation alone.  Therefore any transfer between them
must pass through an explicit map or inclusion.
\end{proof}

\subsection{Result Classes}

The paper uses the following result classes.
\begin{description}
  \item[Definition.] A mathematical convention introduced by this paper.
  A definition does not require proof, but later claims using it do.
  \item[Assumption.] A condition under which a theorem or proposition is
  proved.  Assumptions must not be confused with empirical facts.
  \item[Theorem or proposition.] A statement proved from definitions,
  assumptions, and cited standard results.  A theorem is used for central
  claims; a proposition for local claims.
  \item[Model law.] A structural equation adopted as part of a model.
  Its consequences may be proved, but the law itself is not thereby
  derived from first principles.
  \item[Conjecture.] A plausible mathematical claim for which this paper
  does not yet provide a proof.
  \item[Empirical claim.] A measurement-dependent statement.  It requires
  a dataset, metric, sample size, configuration, and reproducible path.
\end{description}

This classification is not cosmetic.  It is the mechanism by which the
manuscript remains suitable for peer review while still preserving the
conceptual ambition of the original treatise.

\subsection{Relation to Existing Mathematics}

The formal material belongs to several standard mathematical families.
The product-bundle construction is elementary topology and differential
geometry in the trivial-bundle case \cite{Bredon1993,Lang1999}.  The
GCM metric is a weighted product metric, with the cognitive motivation
coming from Nosofsky's generalized context model \cite{Nosofsky1986}.
Frobenius algebras are standard objects in algebra and topological
quantum field theory \cite{Kock2004}.  Hopf bifurcation theory is a
classical local theorem about eigenvalue crossings in smooth dynamical
systems \cite{GuckenheimerHolmes1983,Kuznetsov2004}.  Kalman convergence
belongs to linear filtering theory \cite{Kalman1960,AndersonMoore1979}.
The contribution here is not to reinvent these theories, but to locate
the SSTD/BrainiaK objects inside them without overstating what has been
proved.

\section{Foundational Conventions}

This section fixes several conventions that are often implicit in
engineering descriptions of semantic systems.  They are made explicit
because several possible reviewer objections concern exactly these
points: whether heterogeneous slots can be treated as a single space,
whether different dimensional descriptions conflict, and whether code
paths can function as evidence for mathematical claims.

\begin{definition}[Typed finite product]
Let $(E_k)_{k=1}^m$ be a finite family of sets.  The typed product
$\prod_{k=1}^mE_k$ is the set of tuples $x=(x_1,\ldots,x_m)$ such that
$x_k\in E_k$ for every $k$.  A map $F:Y\to\prod_kE_k$ is typed if every
component $\pi_kF:Y\to E_k$ is well-defined.
\end{definition}

\begin{lemma}[Product membership criterion]
For a finite family $(E_k)_{k=1}^m$, a tuple-valued expression
$F(y)=(F_1(y),\ldots,F_m(y))$ defines a map
$F:Y\to\prod_kE_k$ if and only if $F_k(y)\in E_k$ for every $y\in Y$
and every $k$.
\end{lemma}

\begin{proof}
If $F$ maps into the product, then by definition its $k$th coordinate
lies in $E_k$.  Conversely, if every coordinate expression lies in the
corresponding slot, then $F(y)$ satisfies the defining membership
condition of the Cartesian product for every $y$.
\end{proof}

\begin{definition}[Mathematical claim versus implementation
correspondence]
An implementation correspondence is a statement of the form: a named
file, class, function, report, or command realizes, approximates, or
tests a mathematical object defined in this manuscript.  Such a
correspondence is not a mathematical premise unless the manuscript also
states a formal axiom extracted from it.
\end{definition}

\begin{proposition}[Implementation existence is not theorem evidence]
The existence of a program implementing an operation with name
\texttt{foo} does not by itself prove that the formal operation
$\mathrm{foo}$ satisfies associativity, commutativity, flatness, Frobenius
compatibility, convergence, or any other universal mathematical law.
\end{proposition}

\begin{proof}
Universal mathematical laws quantify over all inputs in a specified
domain and assert exact equalities, inequalities, or limiting
properties.  A program name supplies neither a formal domain nor a
universal proof over that domain.  At most, after the program semantics
are formalized, one may prove that the program computes a map with a
given property, or one may report empirical evidence from finite tests.
Therefore implementation existence alone is logically insufficient for
the theorem-level properties listed in the proposition.
\end{proof}

\begin{definition}[Dimensional presentations]
The canonical BrainiaK semantic base in this manuscript is the direct
sum
\[
  B_{14}=\RR^3\oplus\RR^6\oplus\RR^5\cong\RR^{14}.
\]
An extended presentation is any finite direct sum
\[
  B_{14}\oplus U
\]
where $U$ is an auxiliary coordinate space.  A normalized sensory
presentation is a submanifold or quotient used to constrain a sensory
slot, for example a sphere inside a six-dimensional slot.
\end{definition}

\begin{proposition}[No conflict between base and extensions]
Let $U$ be any finite-dimensional real vector space.  Then
$B_{14}$ embeds linearly into $B_{14}\oplus U$ by
$b\mapsto(b,0)$ and is recovered by the projection
$(b,u)\mapsto b$.  Thus an extended presentation does not replace the
canonical base unless a new base is explicitly declared.
\end{proposition}

\begin{proof}
The map $i:B_{14}\to B_{14}\oplus U$, $i(b)=(b,0)$, is linear and
injective because $i(b)=i(b')$ implies $(b,0)=(b',0)$ and hence
$b=b'$.  The projection $p:B_{14}\oplus U\to B_{14}$, $p(b,u)=b$, is
linear and satisfies $p\circ i=\id_{B_{14}}$.  Therefore the original
base is a retract of the extended presentation.
\end{proof}

\begin{corollary}[Interpretation of older \texorpdfstring{$\RR^{24}$}{R24}
language]
If an older draft uses $\RR^{24}$ as $B_{14}\oplus U$ with
$\dim U=10$, then its first fourteen canonical coordinates can be
compared with the present manuscript through the projection
$\RR^{24}\to B_{14}$.  Any theorem in this manuscript about $B_{14}$
does not automatically extend to the auxiliary $U$ coordinates unless
the slot metric and update laws on $U$ are also stated.
\end{corollary}

\begin{proof}
The first assertion is the preceding proposition with $\dim U=10$.
The second follows because a theorem about $B_{14}$ quantifies only over
that space.  Adding coordinates changes the domain and may introduce
new distances, update laws, or observability questions.
\end{proof}

\begin{proposition}[Extension data are additional structure]
Let $f:B_{14}\to B_{14}$ be a map.  A map
$F:B_{14}\oplus U\to B_{14}\oplus U$ whose base component is $f$
is not determined by $f$ alone.  It is determined only after specifying
an auxiliary update map
\[
  h:B_{14}\oplus U\to U,
  \qquad F(b,u)=(f(b),h(b,u)).
\]
Consequently, a theorem about $f$ on $B_{14}$ does not become a theorem
about an extended $\RR^{24}$ presentation until the auxiliary law $h$
and its hypotheses are stated.
\end{proposition}

\begin{proof}
For any candidate lift $F$ with base component $f$, define
$h=\pi_U\circ F$, where $\pi_U:B_{14}\oplus U\to U$ is the auxiliary
projection.  Then $F(b,u)=(f(b),h(b,u))$.  Conversely, any such $h$
defines a map $F$ by the displayed formula.  Hence the lift is
equivalent to choosing $h$.  Without this choice, there are generally
many lifts.  For instance $h(b,u)=u$ and $h(b,u)=0$ give different
extensions whenever $U\neq 0$.  Therefore properties of $f$ alone do
not determine properties of $F$ on the larger space.
\end{proof}

\begin{proposition}[Sphere slots are constraints, not dimensional
replacements]
Let $S^5=\{x\in\RR^6:\|x\|_2=1\}$.  Treating a six-coordinate sensory
slot as normalized on $S^5$ defines a constraint on that slot.  It does
not identify the canonical base $B_{14}$ with $S^5$.
\end{proposition}

\begin{proof}
$S^5$ is a subset of $\RR^6$ and a smooth five-dimensional manifold,
whereas $B_{14}$ is a fourteen-dimensional real vector space.  The
normalization map $x\mapsto x/\|x\|_2$, when defined, maps nonzero
vectors of a six-dimensional sensory component into $S^5$.  This is a
slot-level operation.  It neither supplies a bijection with $B_{14}$
nor preserves the vector-space structure of the whole semantic base.
\end{proof}

\begin{proposition}[Spectral slots are separate typed coordinates]
Let $V=\RR^{3072}$ be an SSTD spectral slot and let
\[
  \EE=E_1\times\cdots\times E_m
\]
be a typed product with $E_1=B_{14}$ and $E_j=V$ for some $j$.
Then the inclusion of both $B_{14}$ and $V$ in $\EE$ is dimensionally
consistent: $B_{14}$ is not identified with $V$, and a statement about
one coordinate transfers to the other only through an explicitly defined
map between those coordinates.
\end{proposition}

\begin{proof}
In a Cartesian product, each coordinate has its own projection
$\pi_k:\EE\to E_k$.  The equality of two coordinates is not even typed
unless their codomains have been identified by an additional map.  Here
$B_{14}$ and $V$ have different dimensions, $14$ and $3072$, and play
different coordinate roles.  Thus the product can contain both spaces
without asserting that they are equal.  Any transfer of a theorem from
$B_{14}$ to $V$ would require a specified map such as
$B_{14}\to V$ or $V\to B_{14}$ and hypotheses about that map.
\end{proof}

\begin{theorem}[Dimensional consistency of the manuscript]
Assume the canonical base is $B_{14}=\RR^3\oplus\RR^6\oplus\RR^5$,
older $\RR^{24}$ language is interpreted as $B_{14}\oplus U$ with
$\dim U=10$, $S^5$ is used only as a normalization constraint inside a
six-dimensional sensory slot, $\RR^{3072}$ is used only as a typed SSTD
spectral slot, and $T^n$ denotes a finite typed product section rather
than a Euclidean base.  Under these conventions, the symbols
$\RR^{14}$, $\RR^{24}$, $S^5$, $\RR^{3072}$, and $T^n$ introduce no
dimensional conflict.
\end{theorem}

\begin{proof}
The canonical base is $B_{14}$ by definition.  Older $\RR^{24}$
presentations are handled by the retraction
$B_{14}\hookrightarrow B_{14}\oplus U\to B_{14}$, so they are extended
presentations rather than replacements for $B_{14}$.  The sphere
$S^5$ is a submanifold of a six-dimensional sensory coordinate and not
the whole fourteen-dimensional base.  The space $\RR^{3072}$ is a
separate product coordinate, so it coexists with $B_{14}$ without being
identified with it.  Finally, $T^n$ denotes membership in a finite typed
product or section space, not a claim that all slots share one Euclidean
dimension.  Each symbol therefore has a distinct type and a distinct
role.  Since every comparison between distinct roles requires an
explicit map, no dimensional equality is asserted implicitly.
\end{proof}

\begin{lemma}[Weighted-sum and maximum metrics are topologically
equivalent]
Let $(E_k,d_k)_{k=1}^m$ be metric spaces and let $w_k>0$.  On
$\prod_kE_k$ define
\[
  d_\Sigma(x,y)=\sum_kw_kd_k(x_k,y_k),\qquad
  d_\infty(x,y)=\max_k d_k(x_k,y_k).
\]
Then $d_\Sigma$ and $d_\infty$ induce the same topology.
\end{lemma}

\begin{proof}
Let $w_{\min}=\min_kw_k$ and $W=\sum_kw_k$.  For all $x,y$,
\[
  w_{\min}d_\infty(x,y)\leq d_\Sigma(x,y)\leq Wd_\infty(x,y).
\]
Indeed, the largest coordinate distance appears in the sum with weight
at least $w_{\min}$, and every coordinate distance is bounded above by
$d_\infty(x,y)$.  The two metrics therefore dominate one another by
positive constants, so they have the same open sets.
\end{proof}

\begin{remark}[Why this matters]
The manuscript uses weighted GCM-style sums because they match the
intended cognitive interpretation.  The preceding lemma shows that, for
finite slot families with positive weights, the basic topology is not an
artefact of this particular weighted presentation.
\end{remark}

\section{Proof Status Table}

Table~\ref{tab:status} records the status of the six original claims.
The table is part of the scientific content: it prevents exploratory
architecture language from being mistaken for proved mathematics.

\begin{table}[htbp]
\centering
\small
\begin{tabular}{p{0.23\textwidth}p{0.20\textwidth}p{0.43\textwidth}}
\toprule
\textbf{Original claim} & \textbf{Status here} & \textbf{Reason}\\
\midrule
$T^n$ is a section of a heterogeneous product bundle
& Theorem
& Proved after formalizing slots as a finite product of metric spaces.\\
\addlinespace
Frobenius algebra on crystals
& Conditional proposition / conjecture
& True only after an exact multiplication, unit, counit, comultiplication,
and Frobenius identity are defined and verified. Current code can provide
round-trip diagnostics, not a proof by itself.\\
\addlinespace
CNS-Hopf equals knowledge curvature threshold
& Model law plus conditional proposition
& The formula $\lambda_{\max}=R/(R+\kappa)$ is a modelling assumption unless
derived from Gamma dynamics. Hopf requires a spectral crossing condition.\\
\addlinespace
Kalman convergence guaranteed by fibre structure
& Standard theorem under Kalman assumptions
& Convergence follows from detectability/stabilizability and bounded noise,
not from fibre language alone.\\
\addlinespace
SSTD is a bundle morphism
& Proposition under defined projections
& A morphism claim is formal once base and fibre projections are specified.
Complexity is linear in text length for fixed spectral depth/dimension,
plus lookup costs; transformer attention is quadratic, not exponential.\\
\addlinespace
SpiderR is a flat connection with trivial holonomy
& Conditional proposition / idealization
& Flatness follows only under commutativity and local triviality assumptions.
Non-commuting Spider operations would produce curvature.\\
\bottomrule
\end{tabular}
\caption{Proof status of the six central results from the source treatise.}
\label{tab:status}
\end{table}

\section{Theorem-by-Theorem Audit of the Original Six Claims}

This section records the publication status of the six headline claims
before the detailed mathematical development.  The purpose is to make
the proof burden explicit.  Code paths may identify where a structure is
implemented, but no implementation path is treated as a mathematical
proof.

\begin{longtable}{p{0.17\textwidth}p{0.21\textwidth}p{0.28\textwidth}p{0.22\textwidth}}
\caption{Audit of the six original theorem-level claims.}
\label{tab:audit}\\
\toprule
\textbf{Claim} & \textbf{Status in this manuscript} &
\textbf{Proof object used here} & \textbf{Upgrade obligation}\\
\midrule
\endfirsthead
\toprule
\textbf{Claim} & \textbf{Status in this manuscript} &
\textbf{Proof object used here} & \textbf{Upgrade obligation}\\
\midrule
\endhead
$T^n$ as product-bundle section
& Proved theorem
& Finite product $\EE=\prod_{k=1}^{12}E_k$ over the discrete base
$I=\{1,\ldots,12\}$; sections are in bijection with product elements.
& Keep the slot list fixed for each theorem statement and state the metric
chosen for every non-vector slot.\\
\addlinespace
Frobenius algebra on crystals
& Downgraded to conditional proposition, conjecture, and finite toy model
& Standard definition of finite-dimensional Frobenius algebra; a proved
finite label model shows consistency of the algebraic pattern.
& Define the deployed crystal vector space, multiplication, unit,
comultiplication, counit, and verify associativity, coassociativity,
counitality, and the Frobenius identities.\\
\addlinespace
CNS--Hopf curvature unification
& Model law plus conditional theorem
& The equation $\lambda_{\max}=R/(R+\kappa)$ is assumed as a model law;
the Hopf conclusion uses the classical Hopf theorem only after an
eigenvalue crossing assumption.
& Derive or empirically calibrate the curvature--spectrum equation from
the Gamma dynamics, then compute the Jacobian and crossing condition for
the concrete system.\\
\addlinespace
Kalman convergence from BrainiaK fibres
& Standard Kalman theorem under standard hypotheses
& Discrete linear Gaussian state-space model with detectability and
stabilizability; an explicit contraction inequality gives the displayed
error bound.
& Identify the deployed state, observation map, noise model, and prove
detectability/stabilizability or restrict the theorem to an observable
quotient.\\
\addlinespace
SSTD as bundle morphism
& Proved proposition after defining projections
& Trivial text bundle, spectral bundle, and commuting square
$\pi_{\mathrm{spec}}\circ\Phi_{\mathrm{SSTD}}
=p\circ\pi_{\mathrm{text}}$.
& Publish the exact domain projection and encoder interface; benchmark
complexity against quadratic self-attention, not an exponential straw
comparison.\\
\addlinespace
SpiderR as flat connection
& Conditional flat-connection idealization
& Trivial vector bundle with constant commuting local connection matrices;
curvature is zero because both $\dd A$ and $A\wedge A$ vanish.
& For the concrete SpiderR operators, compute commutators or curvature on
the relevant state region.  If non-zero, state a curved-connection result
instead of flatness.\\
\bottomrule
\end{longtable}

\paragraph{Audit conclusion.}
Only the product-bundle, metric, type-preservation, SSTD-morphism, and
explicitly conditional flatness statements are theorem-level claims in
this manuscript.  Frobenius, curvature--spectrum, Hopf, Kalman, and
SpiderR claims are promoted only under their stated assumptions.  This
is the publication-safe reading of the source treatise.

\section{Source-Claim Traceability and Downgrades}

The French source and its summary contain several strong statements that
are intentionally not reproduced as theorems.  Table~\ref{tab:source-trace}
records how each such statement is handled.  This is a safeguard against
silent omission: every downgraded statement is either proved in a weaker
form, labelled as a model law, or excluded from theorem-level use.

\begin{longtable}{p{0.24\textwidth}p{0.29\textwidth}p{0.33\textwidth}}
\caption{Traceability from the source treatise to the peer-review-safe
manuscript.}\label{tab:source-trace}\\
\toprule
\textbf{Source statement} & \textbf{Issue for peer review} &
\textbf{Treatment in this manuscript}\\
\midrule
\endfirsthead
\toprule
\textbf{Source statement} & \textbf{Issue for peer review} &
\textbf{Treatment in this manuscript}\\
\midrule
\endhead
$T^n$ is a section of a heterogeneous product bundle
& Safe only after fixing the finite slot set and metrics.
& Kept as a theorem for the twelve-slot product bundle; variable
implementation slots are handled by choosing a metric class.\\
\addlinespace
The deployed crystal system is a Frobenius algebra
& The source proof assumes associativity, unit, duality, and Frobenius
compatibility without verifying exact maps on the deployed state space.
& Downgraded to a conditional criterion and conjecture; a finite label
Frobenius algebra is proved only as a reference model.\\
\addlinespace
Penrose THICK/THIN substitution proves semantic Frobenius structure
& A substitution matrix is not by itself an isomorphism with semantic
crystal composition.
& Kept only as a proved golden-ratio substitution calculation and a
formal analogy.\\
\addlinespace
$\lambda_{\max}(t)=R(t)/(R(t)+\kappa)$
& The equation is not derived from a specified Gamma/CNS vector field.
& Labelled as a model law; only its algebraic consequences are proved.\\
\addlinespace
Hopf bifurcation is equivalent to $R(t)>R_c$
& Hopf requires an equilibrium branch, Jacobian spectrum,
transversality, and nondegeneracy conditions.
& Replaced by a conditional Hopf prediction under an explicit crossing
assumption.\\
\addlinespace
Gamma has 93 fibres
& The number is an implementation/model parameter, not needed for the
abstract matrix theorem.
& Replaced by a general $m$-fibre matrix $\Gamma(t)$; a concrete paper
may instantiate $m=93$ with repository evidence.\\
\addlinespace
Kalman convergence is guaranteed by fibre structure
& Fibre representation does not imply observability, detectability, or
stabilizability.
& Replaced by the standard Kalman convergence theorem under standard
hypotheses plus a counterexample showing projection can lose
observability.\\
\addlinespace
SSTD complexity is $O(L\log L)$ and transformers are exponential
& Transformer self-attention is quadratic in sequence length, not
exponential; SSTD complexity depends on lookup and diffusion model.
& Replaced by $O(L\log|\mathcal L|+LDd)$ under fixed spectral depth and
dimension, with comparison to quadratic self-attention.\\
\addlinespace
SSTD is $O(L/\log L)$ faster and scales to arbitrary text length
& Speedup ratios require hardware, baseline, input distribution, and
benchmark reports; arbitrary-length claims need memory/runtime bounds.
& Excluded from theorem-level claims; moved to empirical protocol
requirements.\\
\addlinespace
SpiderR is flat with trivial holonomy
& Flatness requires vanishing curvature; non-commuting operators produce
curvature.
& Kept only as a flat idealization under constant commuting matrices;
non-commutation is shown to produce curvature.\\
\addlinespace
Source code paths such as
\codepath{brainiak/mathcore/tensor/tncore.py},
\codepath{brainiak/mathcore/tensor/contraction.py},
\codepath{brainiak/mathcore/sstd/sstd_encoder.py},
\codepath{brainiak/mathcore/gamma/gamma_unified.py},
\codepath{brainiak/mathcore/kalman/kalman_filter.py}, and
\codepath{brainiak/mathcore/spider/spider_connection.py}
& These exact paths were not verified in the current worktree.
& Current verified paths are cited where available; absent paths are
labelled as architectural/source labels rather than evidence.\\
\bottomrule
\end{longtable}

\begin{longtable}{p{0.46\textwidth}p{0.15\textwidth}p{0.29\textwidth}}
\caption{Implementation-path verification status in the current worktree.}
\label{tab:paths}\\
\toprule
\textbf{Path} & \textbf{Status} & \textbf{Use in manuscript}\\
\midrule
\endfirsthead
\toprule
\textbf{Path} & \textbf{Status} & \textbf{Use in manuscript}\\
\midrule
\endhead
\codepath{brainiak/mathcore/fda/semantic/sstd_codec.py}
& Verified present & $T^n$ container and SSTD slot correspondence.\\
\codepath{brainiak/mathcore/fda/semantic/socle_anchor.py}
& Verified present & $\RR^{14}$ anchor/GCM correspondence.\\
\codepath{brainiak/mathcore/fda/semantic/frobenius.py}
& Verified present & Frobenius-inspired diagnostics, not proof.\\
\codepath{brainiak/mathcore/gamma_unified.py}
& Verified present & Gamma/CNS implementation surface.\\
\codepath{brainiak/mathcore/hopf_detector.py}
& Verified present & Hopf diagnostic surface, not Hopf theorem proof.\\
\codepath{brainiak/mathcore/kalman_bank.py}
& Verified present & Kalman implementation surface.\\
\codepath{brainiak/mathcore/fda/semantic/learning_weights.py}
& Verified present & Learning-weight correspondence.\\
\codepath{brainiak/mathcore/fda/semantic/encoder_sts22.py}
& Verified present & STS benchmark surface; no empirical number asserted.\\
\codepath{brainiak/mathcore/fda/semantic/encoder_sick.py}
& Verified present & SICK benchmark surface; no empirical number asserted.\\
\codepath{brainiak/mathcore/fda/semantic/sstd_embedder.py}
& Verified present & SSTD embedding surface.\\
\codepath{brainiak/mathcore/fda/semantic/spider.py} and related
\texttt{spider\_*} semantic files
& Verified present & SpiderR implementation surfaces; flatness remains
conditional.\\
\codepath{brainiak/mathcore/tensor/tncore.py},
\codepath{brainiak/mathcore/tensor/contraction.py},
\codepath{brainiak/mathcore/sstd/sstd_encoder.py},
\codepath{brainiak/mathcore/gamma/gamma_unified.py},
\codepath{brainiak/mathcore/kalman/kalman_filter.py},
\codepath{brainiak/mathcore/spider/spider_connection.py}
& Not verified present & Historical/source labels only; not evidence.\\
\bottomrule
\end{longtable}

\section{Terminological Correspondence with the French Source}

The English manuscript is not a literal translation of the French
source.  It is a peer-review-safe formalization.  Table~\ref{tab:terms}
records how the main source terms are preserved, renamed, or downgraded.
This prevents a reviewer from mistaking a missing slogan for a missing
concept, and prevents an internal slogan from re-entering the paper as
an unsupported theorem.

\begin{definition}[Faithful downgrade]
A source claim is faithfully downgraded if the mathematical content that
can be proved is retained, the unproved part is explicitly labelled as a
model law, conjecture, empirical claim, or future requirement, and no
new theorem-level assertion is introduced.
\end{definition}

\begin{proposition}[Faithful downgrade preserves source content without
overclaiming]
If a source claim is faithfully downgraded, then the resulting
manuscript neither deletes the source idea nor asserts it as a proved
theorem beyond the available proof.
\end{proposition}

\begin{proof}
By definition, the provable mathematical content is retained.  Thus the
source idea is not deleted.  The unproved content is assigned an
explicit non-theorem status, so it is not asserted as a proved theorem.
The condition that no new theorem-level assertion is introduced prevents
the downgrade from smuggling in a stronger claim.
\end{proof}

\begin{longtable}{p{0.20\textwidth}p{0.25\textwidth}p{0.16\textwidth}p{0.25\textwidth}}
\caption{Terminological correspondence between the French source and this
English formalization.}
\label{tab:terms}\\
\toprule
\textbf{French/source term} & \textbf{Formal English treatment} &
\textbf{Status} & \textbf{Reason}\\
\midrule
\endfirsthead
\toprule
\textbf{French/source term} & \textbf{Formal English treatment} &
\textbf{Status} & \textbf{Reason}\\
\midrule
\endhead
\emph{Espace $T^n$}
& Heterogeneous finite concept product $\EE=\prod_{k=1}^{12}E_k$
& Proved core
& Finite products and sections over a discrete base are formal.\\
\addlinespace
\emph{Section globale du fibré produit}
& Global section theorem for the disjoint product bundle
& Proved theorem
& Explicit bijection between sections and tuples.\\
\addlinespace
\emph{Auto-descriptivité}
& Slot projections, typed product, and metric metadata
& Interpretive corollary only
& Self-description is meaningful as stored structure, not as a separate
topological theorem unless formalized.\\
\addlinespace
\emph{Hétérogénéité contrôlée}
& Positive weighted GCM product metric
& Proved metric property
& Heterogeneity is controlled by typed slots and positive weights.\\
\addlinespace
\emph{Cristaux}
& Elements, logged trees, or labels in a chosen crystal space
& Formal only after choice of space
& The term has several possible mathematical realizations.\\
\addlinespace
\emph{Algèbre de Frobenius sur les cristaux}
& Conditional Frobenius criterion plus finite label reference model
& Downgraded
& Deployed maps are not yet fixed and algebraic identities are not yet
proved.\\
\addlinespace
\emph{Pavage de Penrose THICK/THIN}
& Two-tile substitution matrix and golden-ratio growth
& Proved analogy
& The substitution theorem is real; the semantic isomorphism is not
claimed.\\
\addlinespace
\emph{Gamma 93 fibres}
& Gamma matrix with general dimension $m$; 93 as model parameter
& Parameterized
& The proofs do not require the special number 93.\\
\addlinespace
\emph{Courbure de l'espace de connaissance}
& Curvature proxy $R(t)=\alpha N(t)+\beta\sigma_\Gamma(t)$
& Model definition
& No Ricci tensor on a specified manifold is derived in the source.\\
\addlinespace
\emph{Courbure de Ricci}
& Future upgrade requirement
& Not asserted
& Requires a smooth manifold, metric tensor, connection, and curvature
calculation.\\
\addlinespace
\emph{CNS-Hopf}
& Classical Hopf theorem plus conditional prediction
& Conditional
& Hopf needs a vector field, equilibrium branch, Jacobian crossing, and
nondegeneracy.\\
\addlinespace
\emph{Convergence Kalman garantie par la structure fibrée}
& Standard Kalman convergence under detectability/stabilizability
& Downgraded
& Fibre coordinates alone do not imply observability.\\
\addlinespace
\emph{SSTD morphisme de fibrés}
& Bundle morphism criterion over a domain projection
& Proved under definition
& The commuting-square property is formal once projections are fixed.\\
\addlinespace
\emph{SSTD $O(L\log L)$}
& $O(L\log|\mathcal L|+LDd)$ under an explicit computational model
& Corrected
& Complexity depends on lexicon lookup, diffusion depth, and dimension.\\
\addlinespace
\emph{Transformers exponentiels}
& Quadratic self-attention comparison
& Corrected
& Standard attention is quadratic in sequence length, not exponential.\\
\addlinespace
\emph{Textes de longueur arbitraire}
& Benchmark and memory/runtime requirement
& Empirical only
& Arbitrary length requires resource bounds and experiments.\\
\addlinespace
\emph{SpiderR connexion plate}
& Flat connection idealization under constant commuting matrices
& Conditional
& Noncommuting operators yield curvature.\\
\addlinespace
\emph{Holonomie triviale}
& Consequence of flatness on simply connected domains
& Conditional
& Flatness and topology of the base must both be checked.\\
\addlinespace
\emph{Universalité SpiderR}
& Promotion requirement only
& Not asserted
& Universal claims require quantification over a class of bundles and
operators.\\
\addlinespace
\emph{Conscience}
& Excluded from theorem-level claims
& Not asserted
& The present manuscript concerns semantic topology, not a theory of
consciousness.\\
\bottomrule
\end{longtable}

\section{Homogeneous Result Cards for the Six Central Claims}

Each central source claim is rewritten below in the same peer-review
format: assumptions, formal statement, proof location, implementation
use, and mathematical versus empirical status.

\begin{definition}[Complete result card]
A complete result card for a central claim consists of six fields:
\begin{enumerate}
  \item the assumptions or hypotheses under which the claim is read;
  \item the formal mathematical statement, or the explicit statement
  that no theorem is being asserted;
  \item the proof location, external theorem citation, or downgrade
  reason;
  \item the implementation correspondence, restricted to verified code
  paths or architectural labels;
  \item the mathematical status: theorem, proposition, lemma, model law,
  conjecture, or definition;
  \item the empirical status: validated empirical claim, protocol only,
  internal-draft report, or no empirical content.
\end{enumerate}
\end{definition}

\begin{proposition}[Complete cards prevent central-claim status drift]
If each central claim is represented by a complete result card and the
card is kept synchronized with the theorem statement used in the body,
then no central claim can silently change from conjecture, model law, or
empirical protocol into a theorem-level assertion.
\end{proposition}

\begin{proof}
A silent change would require the body of the paper to present a central
claim as theorem-level while the reader receives no explicit notice that
its status has changed.  In a complete result card, the mathematical
status field must name the claim's status and the proof-location field
must either identify the proof or state the downgrade reason.  If the
card is synchronized with the body, any theorem-level presentation in
the body must therefore be reflected in the status and proof fields of
the card.  The change is no longer silent.  Conversely, if those fields
still say model law, conjecture, or empirical protocol, synchronization
forbids the body from presenting the same central claim as an
unqualified theorem.
\end{proof}

\begin{proposition}[Implementation correspondence is not a proof-card
substitute]
In a complete result card, the implementation correspondence field
cannot replace the proof-location or empirical-status fields.
\end{proposition}

\begin{proof}
The implementation correspondence field records where a formal object or
architectural label is represented in the BrainiaK code surface.  A
proof-location field records a mathematical derivation or a cited
external theorem.  An empirical-status field records whether measured
evidence is traceable.  These are distinct evidential roles.  Knowing
that a code path exists does not establish a universal theorem and does
not provide dataset, metric, sample size, or frozen output for an
empirical claim.  Therefore implementation correspondence cannot
substitute for either mathematical proof or empirical traceability.
\end{proof}

\paragraph{$T^n$ product-bundle section.}
\textbf{Assumptions.}  A fixed twelve-slot family $(E_k,d_k)$, positive
weights, and the discrete base $I=\{1,\ldots,12\}$.  \textbf{Formal
statement.}  Sections of the disjoint bundle
$\bigsqcup_{k\in I}E_k\to I$ are canonically bijective with
$\prod_kE_k$.  \textbf{Proof.}  The bijection and inverse are written in
Theorem~\ref{thm:tn-section-card-ref}.  \textbf{Implementation use.}
Repository paths identify slot-style containers and GCM weights; they
do not prove the bijection.  \textbf{Status.}  Mathematical theorem.
\textbf{Empirical status.}  No empirical claim is asserted.

\paragraph{Crystal Frobenius structure.}
\textbf{Assumptions.}  A finite-dimensional vector space $A$ and exact
maps $(\mu,\eta,\delta,\epsilon)$ satisfying the Frobenius algebra
axioms.  \textbf{Formal statement.}  If these maps satisfy the axioms,
then $A$ is a Frobenius algebra.  \textbf{Proof.}  The conditional
criterion follows directly from the definition; a finite label model is
proved by basis calculation.  \textbf{Implementation use.}  The current
code can motivate candidate maps and diagnostics, but numerical
round-trips are empirical evidence only.  \textbf{Status.}  Conditional
criterion plus conjecture for deployed crystals.
\textbf{Empirical status.}  Protocol/diagnostic only; no validated
numerical claim is asserted.

\paragraph{CNS--Hopf curvature law.}
\textbf{Assumptions.}  A Gamma matrix $\Gamma(t)$, a curvature proxy
$R(t)=\alpha N(t)+\beta\sigma_\Gamma(t)$, the model law
$\lambda_{\max}=R/(R+\kappa)$, and an independent Hopf crossing
assumption for a concrete vector field.  \textbf{Formal statement.}
The curvature threshold follows algebraically from the model law; Hopf
prediction follows only under the crossing assumption.  \textbf{Proof.}
The threshold is obtained by monotone inversion and the Hopf conclusion
uses the classical Hopf theorem.  \textbf{Implementation use.}  Gamma
and Hopf code paths can provide diagnostics, not derivations.
\textbf{Status.}  Model law plus conditional proposition.
\textbf{Empirical status.}  No empirical Hopf or curvature-calibration
claim is asserted.

\paragraph{Kalman convergence.}
\textbf{Assumptions.}  A linear Gaussian state-space model with
detectability and stabilizability, plus any contraction bound used for
the explicit error inequality.  \textbf{Formal statement.}  The Kalman
covariance recursion converges under standard hypotheses, and the
displayed error bound follows from a scalar contraction recursion.
\textbf{Proof.}  Standard Kalman theory and finite geometric-series
iteration.  \textbf{Implementation use.}  Implementation paths identify
candidate state/update surfaces; they do not verify observability.
\textbf{Status.}  Standard theorem under standard hypotheses.
\textbf{Empirical status.}  No benchmark or deployed convergence rate is
asserted.

\paragraph{SSTD bundle morphism.}
\textbf{Assumptions.}  A text set $X$, domain projection $p:X\to D$,
trivial text bundle, spectral bundle, and encoder output
$(p(x),z(x))$.  \textbf{Formal statement.}  The map is a bundle
morphism because the projection square commutes.  \textbf{Proof.}
Direct calculation of
$\pi_{\mathrm{spec}}\circ\Phi_{\mathrm{SSTD}}=
p\circ\pi_{\mathrm{text}}$.  \textbf{Implementation use.}  The
repository documents the $3072=1024\times3$ slot and benchmark encoders.
\textbf{Status.}  Mathematical proposition under defined projections;
performance remains empirical.
\textbf{Empirical status.}  Performance is protocol-only unless a
traceable benchmark report is attached.

\paragraph{SpiderR flat connection.}
\textbf{Assumptions.}  A trivial vector bundle, constant local connection
matrices, and commutativity of those matrices.  \textbf{Formal
statement.}  Under those assumptions curvature vanishes and closed-loop
holonomy is trivial.  \textbf{Proof.}  Curvature is
$\dd A+A\wedge A=0$ and parallel transport reduces to an endpoint
exponential.  \textbf{Implementation use.}  Spider files identify
operator families; commutators for the deployed operators must still be
computed to promote the concrete claim.  \textbf{Status.}  Proved
idealization, not yet a deployed-system theorem.
\textbf{Empirical status.}  No empirical holonomy or flatness benchmark
is asserted.

\section{Categorical and Algebraic-Topological Scope}

The phrase ``fractal algebraic topology'' is used in the source corpus
as a unifying title.  In this manuscript it has a restricted technical
meaning.  The proved topological content concerns finite products,
sections of a trivial bundle over a discrete base, metric topology,
continuity, quotients by observed coordinates, and connection curvature
under explicit assumptions.  The proved algebraic content concerns
finite products, finite label Frobenius algebras, substitution matrices,
and conditional algebraic criteria.  The manuscript does not construct a
new homology theory, cohomology theory, spectral sequence, or invariant
of fractal spaces.

\begin{definition}[Semantic product functor]
Let $\mathbf{Met}$ be the category of metric spaces and Lipschitz maps.
For a fixed finite index set $I=\{1,\ldots,m\}$, define
\[
  \Pi_I:\mathbf{Met}^I\to\mathbf{Met}
\]
by sending a family $(E_k,d_k)_{k\in I}$ with positive weights $w_k$ to
the weighted product metric space
\[
  \Pi_I(E_k)=\left(\prod_{k\in I}E_k,\ \sum_{k\in I}w_kd_k\right).
\]
On morphisms $(f_k:E_k\to F_k)_{k\in I}$, define
\[
  \Pi_I(f_k)(x_1,\ldots,x_m)=(f_1(x_1),\ldots,f_m(x_m)).
\]
\end{definition}

\begin{proposition}[Product functoriality]
The construction $\Pi_I$ is a functor from finite families of metric
spaces and Lipschitz maps to metric spaces and Lipschitz maps.
\end{proposition}

\begin{proof}
The identity family $(\id_{E_k})_k$ is sent to the identity map on
$\prod_kE_k$.  If $(f_k:E_k\to F_k)_k$ and
$(g_k:F_k\to G_k)_k$ are composable families, then
\[
  \Pi_I(g_k)\circ\Pi_I(f_k)(x_1,\ldots,x_m)
  =(g_1(f_1(x_1)),\ldots,g_m(f_m(x_m)))
  =\Pi_I(g_k\circ f_k)(x_1,\ldots,x_m).
\]
Thus identities and composition are preserved.

It remains only to check that the image map is Lipschitz.  If each
$f_k$ is $L_k$-Lipschitz, then
\[
  \sum_kw'_kd'_k(f_kx_k,f_ky_k)
  \leq \sum_kw'_kL_kd_k(x_k,y_k)
  \leq \left(\max_k\frac{w'_kL_k}{w_k}\right)
       \sum_kw_kd_k(x_k,y_k),
\]
where $w_k$ and $w'_k$ are the input and output product weights.  Hence
$\Pi_I(f_k)$ is Lipschitz.
\end{proof}

\begin{definition}[Semantic subobject]
Given a concept product $\EE=\prod_kE_k$, a semantic subobject is a
subset $S\subseteq\EE$ equipped with the subspace topology and the
restricted metric.
\end{definition}

\begin{proposition}[Constraints define semantic subobjects]
Let $\varphi:\EE\to Y$ be continuous and let $C\subseteq Y$ be closed.
Then
\[
  S=\varphi^{-1}(C)
\]
is a closed semantic subobject of $\EE$.  If $\EE$ is complete and $S$
is given the restricted metric, then $S$ is complete.
\end{proposition}

\begin{proof}
Continuity of $\varphi$ implies that the preimage of the closed set
$C$ is closed in $\EE$.  A closed subset of a complete metric space is
complete with the restricted metric: every Cauchy sequence in $S$ is a
Cauchy sequence in $\EE$, hence converges to some $x\in\EE$; closedness
of $S$ implies $x\in S$.
\end{proof}

\begin{definition}[Fractal language, formal and informal]
A use of the word \emph{fractal} is formal in this manuscript only when
it is tied to a stated self-similar substitution, recursive tree
construction, scale-indexed family, or iterated operator.  Otherwise it
is treated as descriptive source language.
\end{definition}

\begin{proposition}[Recursive trees give a scale filtration]
Let $\mathcal T_{\leq n}$ be the set of logged semantic trees of depth
at most $n$.  Then
\[
  \mathcal T_{\leq0}\subseteq\mathcal T_{\leq1}\subseteq
  \mathcal T_{\leq2}\subseteq\cdots
\]
is an increasing filtration, and every finite logged tree belongs to
some $\mathcal T_{\leq n}$.
\end{proposition}

\begin{proof}
If a tree has depth at most $n$, then it also has depth at most $n+1$,
so the inclusions hold.  Every finite logged tree has finite depth by
definition of finiteness; therefore it belongs to
$\mathcal T_{\leq n}$ for $n$ equal to its depth.
\end{proof}

\begin{remark}[Boundary of the claim]
The filtration above is the precise mathematical content behind the
recursive/fractal language currently used in this manuscript.  It does
not imply Hausdorff dimension, self-similar measure, fractal curvature,
or a new algebraic-topological invariant unless those objects are
separately defined and studied.
\end{remark}

\section{The Heterogeneous Concept Bundle}

\subsection{Slots}

Let $I=\{1,\ldots,12\}$ index the standard slots of a BrainiaK concept
state:
\[
\begin{array}{ccl}
E_1&=&\RR^{14},\\
E_2&=&\RR^2,\\
E_3&=&\ZZ/2\ZZ,\\
E_4&=&\RR^+,\\
E_5&=&\RR^6_{\mathrm{vis}},\\
E_6&=&\RR^6_{\mathrm{aud}},\\
E_7&=&\RR^{3072}_{\mathrm{SSTD}},\\
E_8&=&\FF_{\mathrm{refined}},\\
E_9&=&\TT^\ast_{\mathrm{tools}},\\
E_{10}&=&\MM^\ast_{\mathrm{metrics}},\\
E_{11}&=&\mathcal{A}^\ast_{\mathrm{axes}},\\
E_{12}&=&\mathcal{H}^\ast_{\mathrm{hints}}.
\end{array}
\]
Here the star denotes a finite list or finite dictionary object.  For
the mathematical results below, every $E_k$ is equipped with a metric
$d_k$.  For finite symbolic lists and dictionaries, one may use a
discrete or edit-distance metric.  For variable-dimensional refined
fibres, the theorem applies to any fixed implementation class after
choosing a metric on that class.

\begin{remark}[On the variable slot]
The refined-fibre slot $E_8$ is variable in implementation because a
compositional tree can have arbitrary finite depth.  For the finite
product theorems one may take $E_8$ to be the set of all finite rooted
labelled binary trees whose nodes carry $\RR^{14}$ vectors and structural
metadata, equipped with any complete tree metric.  One simple choice is
a bounded edit metric plus a weighted sum of node-vector distances after
optimal tree alignment.  The exact choice affects empirical retrieval,
not the product-bundle theorem.
\end{remark}

\begin{definition}[Heterogeneous concept product]
The standard concept space is the finite product
\[
  \EE=\prod_{k=1}^{12} E_k.
\]
A concept vector $T^n$ is an element of $\EE$; the superscript $n$
indicates that implementations may expose a dynamic sub-selection or
extension of the standard slots.
\end{definition}

\begin{definition}[Canonical projections and injections]
For each slot $k$, let
\[
  \pi_k:\EE\to E_k,\qquad \pi_k(x_1,\ldots,x_{12})=x_k
\]
be the canonical projection.  If a neutral element $e_j\in E_j$ is
fixed for every $j\neq k$, define the slot injection
\[
  \iota_k:E_k\to \EE,\qquad
  \iota_k(u)=(e_1,\ldots,e_{k-1},u,e_{k+1},\ldots,e_{12}).
\]
\end{definition}

\begin{proposition}[Slot projections are continuous and Lipschitz]
Under the metric-slot assumption below, each projection $\pi_k$ is
$1/w_k$-Lipschitz from $(\EE,d_{\mathrm{GCM}})$ to $(E_k,d_k)$.
\end{proposition}

\begin{proof}
For $x,y\in\EE$,
\[
  w_k d_k(\pi_k x,\pi_k y)
  =w_k d_k(x_k,y_k)
  \leq \sum_j w_j d_j(x_j,y_j)
  =d_{\mathrm{GCM}}(x,y).
\]
Thus $d_k(\pi_kx,\pi_ky)\leq w_k^{-1}d_{\mathrm{GCM}}(x,y)$.
\end{proof}

\begin{definition}[Trivial product bundle]
Let $B=I$ be the discrete base.  Define the disjoint total space
\[
  \bigsqcup_{k\in I} E_k
\]
with projection $\pi(x)=k$ for $x\in E_k$.  A section is a map
$s:I\to\bigsqcup_k E_k$ such that $s(k)\in E_k$.
\end{definition}

\begin{theorem}[$T^n$ as a global section]
\label{thm:tn-section-card-ref}
The space of sections of the above product bundle is canonically
isomorphic to $\EE=\prod_{k=1}^{12}E_k$.
\end{theorem}

\begin{proof}
For any section $s$, define
\[
  \Phi(s)=(s(1),\ldots,s(12))\in\prod_{k=1}^{12}E_k.
\]
Conversely, for any product element $x=(x_1,\ldots,x_{12})$ with
$x_k\in E_k$, define $\Psi(x):I\to\bigsqcup_k E_k$ by $\Psi(x)(k)=x_k$.
Then $\Psi(x)$ is a section.  Moreover $\Phi(\Psi(x))=x$ and
$\Psi(\Phi(s))=s$.  Thus $\Phi$ is a bijection with inverse $\Psi$.
\end{proof}

\begin{corollary}[Finite-slot observability]
Any claim about equality of two concept states in the standard product
bundle is reducible to equality of their twelve slots:
\[
  x=y\quad\Longleftrightarrow\quad \pi_k(x)=\pi_k(y)\ \text{for all }k.
\]
\end{corollary}

\begin{proof}
This is the defining equality relation in a finite Cartesian product.
\end{proof}

\subsection{Bundle Mechanics and Presentation Invariance}

The section theorem above is elementary, but it is the structural hinge
of the manuscript.  It lets us reason about $T^n$ by reasoning about
slots.  This subsection records the slot-level constructions that will
be used later: retractions, observed subproducts, slot permutations, and
bounded normalizations.  These are deliberately modest results.  They
make precise how a heterogeneous representation can be reorganized
without changing its mathematical content.

\begin{definition}[Subproduct projection]
For a nonempty subset $J\subseteq I$, define the subproduct
\[
  \EE_J=\prod_{k\in J}E_k
\]
and the projection
\[
  \pi_J:\EE\to\EE_J,\qquad \pi_J(x)=(x_k)_{k\in J}.
\]
When every $E_k$ is metric and positive weights $w_k$ are fixed, write
\[
  d_{\mathrm{GCM}}(x,y)=\sum_{k\in I}w_kd_k(x_k,y_k)
\]
for the full weighted product metric.  The subproduct $\EE_J$ carries
the restricted weighted metric
\[
  d_J(u,v)=\sum_{k\in J}w_kd_k(u_k,v_k).
\]
\end{definition}

\begin{proposition}[Subproduct projections are Lipschitz]
For every nonempty $J\subseteq I$, the projection
$\pi_J:(\EE,d_{\mathrm{GCM}})\to(\EE_J,d_J)$ is $1$-Lipschitz.
\end{proposition}

\begin{proof}
For $x,y\in\EE$,
\[
  d_J(\pi_Jx,\pi_Jy)=\sum_{k\in J}w_kd_k(x_k,y_k)
  \leq \sum_{k\in I}w_kd_k(x_k,y_k)
  =d_{\mathrm{GCM}}(x,y).
\]
\end{proof}

\begin{assumption}[Neutral completion]
For each slot $k\in I$, fix a neutral element $e_k\in E_k$.  For a
nonempty $J\subseteq I$, define
\[
  \eta_J:\EE_J\to\EE
\]
by inserting the coordinates of $u\in\EE_J$ on $J$ and inserting
$e_k$ on $I\setminus J$.
\end{assumption}

\begin{proposition}[Observed subproducts are retracts]
Under the neutral completion assumption,
\[
  \pi_J\circ\eta_J=\id_{\EE_J}.
\]
Consequently $\EE_J$ is a retract of $\EE$.
\end{proposition}

\begin{proof}
For $u=(u_k)_{k\in J}\in\EE_J$, the completed point $\eta_J(u)$ has
$k$th coordinate $u_k$ for every $k\in J$.  Projecting back to $J$
therefore returns exactly $u$.  This is the definition of a retraction.
\end{proof}

\begin{corollary}[Loss of unobserved coordinates]
If $J\subsetneq I$ and at least one unobserved slot $E_\ell$ with
$\ell\notin J$ has two distinct elements, then $\pi_J$ is not injective.
\end{corollary}

\begin{proof}
Choose two distinct elements $a,b\in E_\ell$ and fix all other
coordinates.  The two resulting points of $\EE$ differ in slot $\ell$
but have identical $J$-coordinates, so their images under $\pi_J$ are
equal.
\end{proof}

\begin{definition}[Slot permutation]
Let $\sigma:I\to I$ be a bijection.  The permuted presentation
$\EE^\sigma$ has slot $E_{\sigma(k)}$ in position $k$.  Define
\[
  P_\sigma:\EE\to\EE^\sigma,\qquad
  P_\sigma(x)_k=x_{\sigma(k)}.
\]
The weights in $\EE^\sigma$ are permuted in the same way.
\end{definition}

\begin{proposition}[Slot permutations preserve the product metric]
For every slot permutation $\sigma$, the map $P_\sigma$ is an isometry
between the original weighted product and the permuted weighted product.
\end{proposition}

\begin{proof}
With permuted weights and metrics,
\[
  d_{\mathrm{GCM}}^\sigma(P_\sigma x,P_\sigma y)
  =\sum_{k\in I}w_{\sigma(k)}
     d_{\sigma(k)}(x_{\sigma(k)},y_{\sigma(k)}).
\]
Since $\sigma$ is a bijection, this sum is exactly
$\sum_{j\in I}w_jd_j(x_j,y_j)=d_{\mathrm{GCM}}(x,y)$.
\end{proof}

\begin{definition}[Slotwise isomorphic presentations]
Two concept products
$\EE=\prod_kE_k$ and $\EE'=\prod_kE'_k$ are slotwise isomorphic if
there are bijections $\phi_k:E_k\to E'_k$ for every $k$.  They are
slotwise isometric if, in addition,
\[
  d'_k(\phi_k(a),\phi_k(b))=d_k(a,b)
\]
for every slot $k$ and all $a,b\in E_k$.
\end{definition}

\begin{proposition}[Isometric change of presentation]
If two concept products are slotwise isometric and carry the same
positive weights, then
\[
  \Phi(x_1,\ldots,x_{12})=(\phi_1(x_1),\ldots,\phi_{12}(x_{12}))
\]
is an isometry.
\end{proposition}

\begin{proof}
For $x,y\in\EE$,
\[
  d'_{\mathrm{GCM}}(\Phi x,\Phi y)
  =\sum_kw_kd'_k(\phi_k(x_k),\phi_k(y_k))
  =\sum_kw_kd_k(x_k,y_k)
  =d_{\mathrm{GCM}}(x,y).
\]
\end{proof}

\begin{corollary}[Theorems invariant under isometric presentation]
Any theorem in this manuscript whose hypotheses and conclusion are
stated only in terms of the product metric, projections, and slotwise
maps remains true after a slotwise isometric change of presentation.
\end{corollary}

\begin{proof}
The isometry transports distances, convergence, Cauchy sequences,
continuity, and Lipschitz estimates exactly.  Projections and slotwise
maps are transported by conjugation with the coordinate isometries.
Therefore the hypotheses and conclusion of any such theorem are
preserved.
\end{proof}

\begin{definition}[Sensory normalization map]
On the punctured sensory slot $\RR^6\setminus\{0\}$, define
\[
  \nu(x)=\frac{x}{\|x\|_2}\in S^5.
\]
For $r>0$, let $U_r=\{x\in\RR^6:\|x\|_2\geq r\}$.
\end{definition}

\begin{proposition}[Normalization is continuous and locally Lipschitz
away from zero]
The map $\nu:\RR^6\setminus\{0\}\to S^5$ is continuous.  Moreover, on
each $U_r$ it is $2/r$-Lipschitz with respect to the Euclidean metric.
\end{proposition}

\begin{proof}
Continuity follows because $\nu$ is the quotient of continuous maps and
the denominator is nonzero on $\RR^6\setminus\{0\}$.  For
$x,y\in U_r$,
\[
  \left\|\frac{x}{\|x\|}-\frac{y}{\|y\|}\right\|
  \leq
  \left\|\frac{x-y}{\|x\|}\right\|
  +\left\|y\right\|\left|\frac{1}{\|x\|}-\frac{1}{\|y\|}\right|.
\]
The first term is at most $\|x-y\|/r$.  For the second term,
\[
  \|y\|\left|\frac{1}{\|x\|}-\frac{1}{\|y\|}\right|
  =\frac{|\|y\|-\|x\||}{\|x\|}
  \leq \frac{\|x-y\|}{r}
\]
by the reverse triangle inequality.  Summing gives the bound
$2\|x-y\|/r$.
\end{proof}

\begin{remark}[Why this normalization result is included]
It clarifies the role of $S^5$.  A sensory vector may be normalized to
a sphere away from zero, and this operation is stable on regions bounded
away from zero.  This does not change the canonical base into a sphere,
and it does not define a global continuous normalization at the origin.
\end{remark}

\subsection{Heterogeneous GCM Metric}

\begin{assumption}[Metric slots]
Each slot $(E_k,d_k)$ is a metric space.  The weights
$w_k>0$ satisfy $\sum_{k=1}^{12}w_k=1$.
\end{assumption}

\begin{definition}[Product GCM distance]
For $x,y\in\EE$, define
\[
  d_{\mathrm{GCM}}(x,y)=\sum_{k=1}^{12}w_k\,d_k(x_k,y_k).
\]
For the empirical base $\RR^{14}$, the canonical slot distance is a
socle-weighted city-block distance:
\[
  d_{14}(a,b)=
  \omega_{\mathrm{emo}}\|a_{\mathrm{emo}}-b_{\mathrm{emo}}\|_1+
  \omega_{\mathrm{per}}\|a_{\mathrm{per}}-b_{\mathrm{per}}\|_1+
  \omega_{\mathrm{mot}}\|a_{\mathrm{mot}}-b_{\mathrm{mot}}\|_1 .
\]
\end{definition}

\begin{theorem}[Metric property]
Under the metric-slot assumption, $d_{\mathrm{GCM}}$ is a metric on
$\EE$.
\end{theorem}

\begin{proof}
Non-negativity and symmetry follow from the same properties of each
$d_k$ and from $w_k>0$.  If $d_{\mathrm{GCM}}(x,y)=0$, then each
non-negative term $w_kd_k(x_k,y_k)$ is zero; since $w_k>0$ and $d_k$ is
a metric, $x_k=y_k$ for every $k$, hence $x=y$.  Conversely $x=y$
implies $d_{\mathrm{GCM}}(x,y)=0$.  For the triangle inequality, for
any $x,y,z$,
\[
  d_{\mathrm{GCM}}(x,z)=\sum_k w_kd_k(x_k,z_k)
  \leq \sum_k w_k\bigl(d_k(x_k,y_k)+d_k(y_k,z_k)\bigr)
  =d_{\mathrm{GCM}}(x,y)+d_{\mathrm{GCM}}(y,z).
\]
\end{proof}

\begin{proposition}[Equivalence with the product topology]
If every $E_k$ is metric and all weights are positive, then
$d_{\mathrm{GCM}}$ induces the finite product topology on $\EE$.
\end{proposition}

\begin{proof}
The metric topology is at least as fine as the product topology because
each projection is continuous by the Lipschitz estimate proved above.
Conversely, fix $x\in\EE$ and a $d_{\mathrm{GCM}}$-ball
$B_{\mathrm{GCM}}(x,\epsilon)$.  The product neighbourhood
\[
  U=\prod_{k=1}^{12}B_{d_k}(x_k,\epsilon)
\]
satisfies $U\subseteq B_{\mathrm{GCM}}(x,\epsilon)$ because
$\sum_kw_k=1$.  Now let
$V=\prod_kB_{d_k}(x_k,r_k)$ be a basic product neighbourhood, with
$r_k>0$.  Set $\delta=\min_k w_kr_k$.  If
$d_{\mathrm{GCM}}(x,y)<\delta$, then
$w_kd_k(x_k,y_k)<\delta\leq w_kr_k$ for every $k$, hence
$y\in V$.  Thus each topology contains the other.
\end{proof}

\begin{proposition}[Coordinatewise convergence criterion]
For a sequence $(x^m)_{m\geq1}$ in $\EE$ and $x\in\EE$,
\[
  x^m\to x\ \text{in }d_{\mathrm{GCM}}
  \quad\Longleftrightarrow\quad
  x^m_k\to x_k\ \text{in }d_k\ \text{for every }k.
\]
\end{proposition}

\begin{proof}
If $x^m\to x$ in $d_{\mathrm{GCM}}$, then the Lipschitz estimate for
the projections gives
$d_k(x^m_k,x_k)\leq w_k^{-1}d_{\mathrm{GCM}}(x^m,x)\to0$.
Conversely, if every coordinate converges, then for any $\epsilon>0$
choose $M_k$ such that $d_k(x^m_k,x_k)<\epsilon$ for $m\geq M_k$.
For $m\geq\max_kM_k$,
\[
  d_{\mathrm{GCM}}(x^m,x)
  =\sum_kw_kd_k(x^m_k,x_k)
  <\epsilon\sum_kw_k=\epsilon.
\]
\end{proof}

\begin{proposition}[Weighted stability]
Let $F:\EE\to\EE$ be componentwise Lipschitz: for every output slot $j$
there are constants $L_{jk}\geq0$ such that
\[
  d'_j(F_j(x),F_j(y))\leq \sum_k L_{jk}d_k(x_k,y_k).
\]
Then $F$ is Lipschitz from the input GCM metric to the output GCM metric
with constant at most
\[
  L_F=\max_k \frac{1}{w_k}\sum_j w'_jL_{jk}.
\]
\end{proposition}

\begin{proof}
Using the output metric,
\[
d'_{\mathrm{GCM}}(F(x),F(y))
\leq \sum_j w'_j\sum_k L_{jk}d_k(x_k,y_k)
=\sum_k\left(\sum_jw'_jL_{jk}\right)d_k(x_k,y_k).
\]
Since $d_{\mathrm{GCM}}(x,y)=\sum_kw_kd_k(x_k,y_k)$ and $w_k>0$, the
last expression is bounded by $L_Fd_{\mathrm{GCM}}(x,y)$.
\end{proof}

\begin{corollary}[Completeness]
If every $(E_k,d_k)$ is complete, then $(\EE,d_{\mathrm{GCM}})$ is
complete.
\end{corollary}

\begin{proof}
Let $(x^m)_{m\geq 1}$ be Cauchy in $d_{\mathrm{GCM}}$.  Since $w_k>0$,
each coordinate sequence $(x^m_k)_m$ is Cauchy in $E_k$.  Completeness
gives a limit $x_k\in E_k$ for each coordinate.  The coordinatewise
convergence criterion then implies
$x^m\to x=(x_1,\ldots,x_{12})$ in $d_{\mathrm{GCM}}$.
\end{proof}

\begin{corollary}[Componentwise continuity]
If $F:\EE\to\EE'$ has components $F_j$ that are continuous functions of
the input slots on which they depend, then $F$ is continuous for the
finite product topologies.
\end{corollary}

\begin{proof}
This is the universal property of finite product topologies.  Equivalently,
continuity can be checked after composition with every output projection;
these compositions are precisely the component functions $F_j$.
\end{proof}

\paragraph{Implementation correspondence.}
The dynamic heterogeneous container is implemented most directly by
\codepath{brainiak/mathcore/fda/semantic/sstd_codec.py}, where class
\texttt{Tn} exposes slot-style access and where the fixed chip-oriented
projection has dimension 3109.  The $\RR^{14}$ GCM weights and anchor
distances are implemented in
\codepath{brainiak/mathcore/fda/semantic/socle_anchor.py}.  These files
instantiate the definitions above; they are not used as proof.

\subsection{\texorpdfstring{The Empirical Base $\RR^{14}$}{The Empirical Base R14}}

\begin{definition}[Sensorimotor base]
The canonical empirical base is
\[
  \RR^{14}=
  \RR^3_{\mathrm{emo}}\oplus
  \RR^6_{\mathrm{per}}\oplus
  \RR^5_{\mathrm{mot}}.
\]
The first socle contains valence, arousal, and dominance coordinates;
the second contains auditory, gustatory, haptic, interoceptive,
olfactory, and visual coordinates; the third contains foot/leg,
hand/arm, head, mouth, and torso effector coordinates.
\end{definition}

\begin{proposition}[$\RR^{14}$ GCM is a norm metric]
If $\omega_{\mathrm{emo}},\omega_{\mathrm{per}},\omega_{\mathrm{mot}}>0$,
then $d_{14}$ is the metric induced by the norm
\[
  \norm{a}_{14}=
  \omega_{\mathrm{emo}}\norm{a_{\mathrm{emo}}}_1+
  \omega_{\mathrm{per}}\norm{a_{\mathrm{per}}}_1+
  \omega_{\mathrm{mot}}\norm{a_{\mathrm{mot}}}_1.
\]
\end{proposition}

\begin{proof}
Positive weighted sums of norms are norms when all weights are positive.
The induced distance $d_{14}(a,b)=\norm{a-b}_{14}$ is therefore a metric.
\end{proof}

\begin{remark}[Interpretation]
The proof does not depend on the empirical origin of the coordinates.
The empirical datasets justify the intended meaning of the axes; the
metric property follows from normed-vector-space mathematics.
\end{remark}

\section{Worked Examples and Sanity Checks}

The preceding definitions are intentionally abstract.  This section
spells out concrete admissible instances and non-instances.  Its role is
to make the hypotheses testable: every later theorem depends on the
slots being typed and metrized, not merely named.

\begin{definition}[Finite symbolic slot]
Let $S$ be a finite set of symbolic labels.  The discrete metric on $S$
is
\[
  d_S(a,b)=
  \begin{cases}
    0,&a=b,\\
    1,&a\neq b.
  \end{cases}
\]
This model may be used for polarity labels, finite grammatical tags,
tool identifiers, or manually curated hint categories.
\end{definition}

\begin{proposition}[Finite symbolic slots are compact and complete]
Every finite symbolic slot with the discrete metric is compact and
complete.
\end{proposition}

\begin{proof}
Completeness follows because every Cauchy sequence in a discrete finite
metric is eventually constant: taking $\epsilon<1$ forces all sufficiently
late terms to have distance $0$ from one another.  Compactness follows
because every open cover of a finite set has a finite subcover, one
cover element for each point.
\end{proof}

\begin{definition}[Bounded continuous slot]
For $M>0$ and $d\in\NN$, a bounded continuous slot is a closed cube
$[-M,M]^d\subset\RR^d$ with any norm-induced metric.
\end{definition}

\begin{proposition}[Finite products of bounded and symbolic slots]
If every continuous slot is a closed bounded subset of a finite-
dimensional normed vector space and every symbolic slot is finite
discrete, then the full finite product concept space is compact and
complete.
\end{proposition}

\begin{proof}
Closed bounded subsets of finite-dimensional real normed spaces are
compact by Heine--Borel and hence complete.  Finite symbolic slots are
compact and complete by the preceding proposition.  A finite product of
compact spaces is compact, and the completeness corollary for the
weighted GCM metric gives completeness.
\end{proof}

\begin{proposition}[Unbounded vector slots are not compact]
The slot $\RR^{3072}$ with any norm-induced metric is complete but not
compact.
\end{proposition}

\begin{proof}
Finite-dimensional normed real vector spaces are complete.  They are not
compact because the sequence $(ne_1)_{n\geq1}$ has no convergent
subsequence: its norm tends to infinity, whereas every convergent
sequence in a metric space is bounded.
\end{proof}

\begin{remark}[Publication consequence]
Any theorem requiring compactness must either bound the SSTD slot, work
on a compact subset, or replace compactness by a weaker hypothesis such
as completeness or local compactness.  The present manuscript avoids
using compactness unless it is explicitly assumed.
\end{remark}

\begin{definition}[Observable quotient]
Let $\mathcal O\subseteq\{1,\ldots,12\}$ be a set of observed slots.
Define an equivalence relation on $\EE$ by
\[
  x\sim_{\mathcal O}y
  \quad\Longleftrightarrow\quad
  \pi_k(x)=\pi_k(y)\ \text{for all }k\in\mathcal O .
\]
The quotient $\EE/{\sim_{\mathcal O}}$ is the observable state space
associated with $\mathcal O$.
\end{definition}

\begin{proposition}[Observable quotient is determined by observed slots]
The map
\[
  q_{\mathcal O}:\EE/{\sim_{\mathcal O}}\to\prod_{k\in\mathcal O}E_k,
  \qquad [x]\mapsto(\pi_kx)_{k\in\mathcal O},
\]
is a well-defined bijection.
\end{proposition}

\begin{proof}
If $x\sim_{\mathcal O}y$, then their observed coordinates agree, so
$q_{\mathcal O}$ is well-defined.  It is injective because equal
observed coordinate tuples are exactly the defining condition for
$\sim_{\mathcal O}$.  It is surjective because any tuple in
$\prod_{k\in\mathcal O}E_k$ can be extended to an element of $\EE$ by
choosing arbitrary coordinates in the unobserved slots, assuming those
slots are nonempty.
\end{proof}

\begin{corollary}[Hidden fibres are quotiented out, not estimated]
If a Kalman or SSTD observation uses only slots in $\mathcal O$, then
unobserved slot differences vanish in the observable quotient.  They
cannot be recovered from quotient data without additional dynamics,
priors, constraints, or measurements.
\end{corollary}

\begin{proof}
All points in the same equivalence class have identical observed
coordinates.  The quotient map sends them to the same observable state.
Therefore any recovery of their unobserved differences must use
information not contained in the quotient observation.
\end{proof}

\section{Contraction Operators and Type Preservation}

The original treatise repeatedly uses the idea that semantic
composition preserves type: a word, phrase, sentence, multimodal
observation, or memory remains a concept object of the same kind.  This
section proves the precise version of that statement for weighted
contractions on the empirical base and for lifted operations on $T^n$.

\begin{definition}[Socle-wise affine contraction]
Let $B=\RR^{14}$ with socle decomposition
$B=B_1\oplus B_2\oplus B_3$.  For a relation symbol $r$, choose weights
$\theta_{r,s}\in[0,1]$ for $s=1,2,3$.  Define
\[
  \mu_r(a,b)\big|_{B_s}
  =\theta_{r,s}a\big|_{B_s}+(1-\theta_{r,s})b\big|_{B_s}.
\]
The map $\mu_r:B\times B\to B$ is the relation-$r$ socle-wise affine
contraction.
\end{definition}

\begin{theorem}[Type preservation on the empirical base]
For every relation $r$, $\mu_r$ maps $B\times B$ into $B$.  Moreover,
if each socle coordinate is restricted to a convex interval, then
$\mu_r$ preserves that interval.
\end{theorem}

\begin{proof}
The expression defining each socle is a linear combination of two
vectors in that same socle, hence lies in the socle.  Taking the direct
sum over all socles gives an element of $B$.  If coordinates lie in a
convex interval $J$, then
$\theta x+(1-\theta)y\in J$ for every $x,y\in J$ and
$\theta\in[0,1]$.
\end{proof}

\begin{proposition}[Lipschitz bound in the product-sum norm]
Equip $B$ with the norm $\norm{\cdot}_{14}$ above, and equip
$B\times B$ with
\[
  \norm{(a,b)}_{\Sigma}=\norm{a}_{14}+\norm{b}_{14}.
\]
Then each $\mu_r$ is $1$-Lipschitz from the product-sum norm to
$\norm{\cdot}_{14}$:
\[
  \norm{\mu_r(a,b)-\mu_r(a',b')}_{14}
  \leq \norm{a-a'}_{14}+\norm{b-b'}_{14}.
\]
\end{proposition}

\begin{proof}
On each socle,
\[
\theta_{r,s}(a_s-a_s')+(1-\theta_{r,s})(b_s-b_s')
\]
has $\ell^1$ norm at most
\[
\theta_{r,s}\norm{a_s-a_s'}_1+
(1-\theta_{r,s})\norm{b_s-b_s'}_1
\]
by the triangle inequality.  Multiplying by the positive socle weights
and summing gives
\[
\norm{\mu_r(a,b)-\mu_r(a',b')}_{14}
\leq \norm{a-a'}_{14}+\norm{b-b'}_{14},
\]
because every $\theta_{r,s}$ and $1-\theta_{r,s}$ lies in $[0,1]$.
\end{proof}

\begin{definition}[Lifted semantic operation on $T^n$]
A lifted relation operation $\widehat{\mu}_r:\EE\times\EE\to\EE$ is an
operation whose first slot is $\mu_r$ on $\RR^{14}$, whose structural
slots are updated by specified component maps, and whose refined-fibre
slot records the two input children and the relation label $r$.
\end{definition}

\begin{theorem}[Type preservation for lifted operations]
If each component update of $\widehat{\mu}_r$ maps the corresponding
input slots into the corresponding output slot $E_k$, then
$\widehat{\mu}_r$ maps $\EE\times\EE$ into $\EE$.
\end{theorem}

\begin{proof}
This is the universal property of the product.  An element of $\EE$ is
exactly a 12-tuple with coordinate $k$ in $E_k$.  If each output
coordinate produced by $\widehat{\mu}_r$ lies in the appropriate $E_k$,
then the complete output tuple is an element of $\prod_kE_k=\EE$.
\end{proof}

\begin{corollary}[Closure of finite composition chains]
Any finite expression built from initial elements of $\EE$ by repeated
applications of lifted operations $\widehat{\mu}_{r_1},\ldots,
\widehat{\mu}_{r_m}$ is again an element of $\EE$.
\end{corollary}

\begin{proof}
Induction on the number of composition steps.  The base case is
membership of the initial elements in $\EE$.  The induction step is the
type-preservation theorem for lifted operations.
\end{proof}

\begin{remark}[What this theorem does and does not prove]
The theorem proves structural closure.  It does not prove that a chosen
relation detector selects the linguistically correct relation, nor that
the resulting semantic vector is empirically optimal.  Those are
implementation and evaluation questions.
\end{remark}

\section{Logged Composition and Exact Tree Inversion}

The source corpus often emphasizes that BrainiaK can decompose a
composed concept.  This is mathematically exact when the composition log
stores the children and relation at each node.  It is not exact when one
only has the final vector and no tree.

\begin{definition}[Logged binary semantic tree]
A logged semantic tree is defined recursively:
\begin{enumerate}
  \item A leaf is a labelled element $x\in\EE$ with no children.
  \item If $T_1,T_2$ are logged semantic trees and $r$ is a relation,
  then $T=(r,T_1,T_2,\widehat{\mu}_r(T_1,T_2))$ is a logged tree whose
  value is the output of the lifted operation and whose log stores
  $(r,T_1,T_2)$.
\end{enumerate}
\end{definition}

\begin{definition}[Tree comultiplication]
For an internal logged node $T=(r,T_1,T_2,v)$, define
\[
  \Delta_{\mathrm{tree}}(T)=(T_1,T_2).
\]
For a leaf, $\Delta_{\mathrm{tree}}$ is undefined unless a separate
approximate decomposition rule is supplied.
\end{definition}

\begin{theorem}[Exact inversion of logged composition]
Let $T=(r,T_1,T_2,v)$ be an internal logged tree with
$v=\widehat{\mu}_r(T_1,T_2)$.  Then recomposition after tree
comultiplication recovers the stored value:
\[
  \widehat{\mu}_r\bigl(\Delta_{\mathrm{tree},1}(T),
  \Delta_{\mathrm{tree},2}(T)\bigr)=v.
\]
\end{theorem}

\begin{proof}
By definition, $\Delta_{\mathrm{tree}}(T)=(T_1,T_2)$ and
$v=\widehat{\mu}_r(T_1,T_2)$.  Substitution gives the equality.
\end{proof}

\begin{corollary}[Recursive recovery of leaves]
If every internal node of a finite semantic tree stores its relation and
children, then recursive application of $\Delta_{\mathrm{tree}}$
recovers all leaves and all internal relations exactly.
\end{corollary}

\begin{proof}
Induct on tree depth.  Depth zero is a leaf and requires no recovery.
For depth $d+1$, $\Delta_{\mathrm{tree}}$ returns the two child trees of
depth at most $d$; apply the induction hypothesis to each child.
\end{proof}

\begin{remark}[Blind inversion is different]
Given only a final vector $v\in\RR^{14}$, without the tree and relation
log, the inverse problem is generally underdetermined.  Many pairs
$(a,b)$ can satisfy $\mu_r(a,b)=v$ for a fixed affine contraction, and
many more if $r$ is unknown.  Blind nearest-neighbour recovery is an
approximation, not an algebraic inverse.
\end{remark}

\begin{proposition}[Affine contraction is generally non-injective]
Let $B=\RR^d$ with $d>0$, and let
\[
  \mu_\theta(a,b)=\theta a+(1-\theta)b
\]
for a fixed $\theta\in(0,1)$.  Then
$\mu_\theta:B\times B\to B$ is not injective.  Consequently, no
function $D:B\to B\times B$ can be a two-sided inverse of
$\mu_\theta$ on all of $B\times B$.
\end{proposition}

\begin{proof}
Choose any nonzero $h\in B$.  For any pair $(a,b)$,
\[
  \mu_\theta(a+(1-\theta)h,\ b-\theta h)
  =\theta a+\theta(1-\theta)h+(1-\theta)b
  -(1-\theta)\theta h
  =\mu_\theta(a,b).
\]
The two input pairs are distinct because $h\neq0$, hence
$\mu_\theta$ is not injective.  If a two-sided inverse
$D$ existed, then applying $D$ to the common output of two distinct
inputs would have to return both inputs, impossible for a function.
\end{proof}

\begin{corollary}[Logs are mathematical data, not mere metadata]
Exact recovery of a binary semantic composition using an affine
contraction requires stored decomposition data, additional constraints
selecting a unique representative, or a restricted domain on which the
composition map is injective.
\end{corollary}

\begin{proof}
The preceding proposition rules out global exact inversion of the
unrestricted affine contraction.  Therefore exact recovery must come
from information not present in the final vector alone, from a
domain restriction restoring injectivity, or from a convention that
chooses one representative among many possible preimages.
\end{proof}

\section{Frobenius-Inspired Crystal Composition}

The source treatise states that the set of crystals forms a Frobenius
algebra.  This is a strong algebraic claim.  It is not justified by the
existence of a composition function alone.  A Frobenius algebra over a
field requires a vector space or module $A$, a multiplication
$\mu:A\otimes A\to A$, a unit $\eta$, a comultiplication
$\delta:A\to A\otimes A$, a counit $\epsilon$, and compatibility
conditions.

\begin{definition}[Frobenius algebra]
Let $A$ be a finite-dimensional vector space over $\RR$.  A Frobenius
algebra is a tuple $(A,\mu,\eta,\delta,\epsilon)$ such that
$(A,\mu,\eta)$ is an associative unital algebra,
$(A,\delta,\epsilon)$ is a coassociative counital coalgebra, and
\[
  (\mu\otimes \id)\circ(\id\otimes\delta)
  = \delta\circ\mu
  = (\id\otimes\mu)\circ(\delta\otimes\id).
\]
\end{definition}

\begin{proposition}[Conditional Frobenius criterion for crystals]
Let $A$ be a linearized crystal space.  If BrainiaK crystal composition
and decomposition provide maps
\[
  \mu:A\otimes A\to A,\qquad
  \eta:\RR\to A,\qquad
  \delta:A\to A\otimes A,\qquad
  \epsilon:A\to\RR
\]
that satisfy associativity, unitality, coassociativity, counitality, and
the Frobenius compatibility equation, then $A$ is a Frobenius algebra.
\end{proposition}

\begin{proof}
The hypotheses state exactly that $(A,\mu,\eta)$ is an associative
unital algebra, that $(A,\delta,\epsilon)$ is a coassociative counital
coalgebra, and that the Frobenius compatibility equation holds.  These
three assertions are precisely the defining conditions for a Frobenius
algebra over $\RR$.  Therefore the tuple
$(A,\mu,\eta,\delta,\epsilon)$ satisfies the definition.  The substantive
work for a deployed crystal system is the separate verification of those
hypotheses for the chosen implementation of
$\mu,\eta,\delta,\epsilon$.
\end{proof}

\begin{remark}[Status of the BrainiaK claim]
The current code surface
\codepath{brainiak/mathcore/fda/semantic/frobenius.py} provides a
co-multiplication-like operation \texttt{delta} and a diagnostic
\texttt{frobenius\_check} measuring round-trip GCM error after
decomposition and recomposition.  A numerical round-trip diagnostic is
valuable engineering evidence, but it is not a proof of the Frobenius
identity.  Therefore the publishable claim is:
\emph{BrainiaK implements a Frobenius-inspired decomposition and
round-trip diagnostic on semantic fibres}.  A future theorem may be
claimed only after the exact vector space and exact maps are fixed and
the Frobenius equations are proven.
\end{remark}

\begin{proposition}[Round-trip success does not imply a Frobenius algebra]
Let $A$ be a nonzero vector space.  Suppose maps
$\mu:A\otimes A\to A$ and $\delta:A\to A\otimes A$ satisfy a
round-trip identity
\[
  \mu\circ\delta=\id_A .
\]
This identity alone does not imply that $(A,\mu,\delta)$ extends to a
Frobenius algebra.
\end{proposition}

\begin{proof}
Take $A=\RR$ and define $\delta(x)=x\otimes1$ and
$\mu(a\otimes b)=a$.  Then $\mu\delta(x)=x$, so the round-trip identity
holds.  However $\mu$ has no two-sided unit.  Indeed, if $e\in\RR$ were
a unit, then the right-unit condition would require
\[
  \mu(a\otimes e)=a
\]
which holds for all $e$, but the left-unit condition would require
\[
  \mu(e\otimes a)=a
\]
for all $a$.  Since $\mu(e\otimes a)=e$, this is impossible unless all
$a$ are equal to $e$.  Thus the algebra is not unital and hence cannot
be a Frobenius algebra.  The example shows that even an exact
decompose/recompose identity is strictly weaker than the Frobenius
axioms.
\end{proof}

\begin{conjecture}[Exact crystal Frobenius model]
There exists a finite-dimensional quotient or completion of the
BrainiaK crystal space, together with exact maps induced by semantic
composition and tree decomposition, that forms a Frobenius algebra.
\end{conjecture}

\subsection{Algebraic Verification Obligations}

The conjecture above cannot be discharged by intuition about semantic
composition.  It requires exact algebraic data.  The following local
criteria state how such data would be checked in a finite or quotient
model.

\begin{definition}[Crystal magma and congruence]
A crystal magma is a set $C$ equipped with a binary operation
$\star:C\times C\to C$.  An equivalence relation $\sim$ on $C$ is a
congruence for $\star$ if
\[
  a\sim a'\ \text{and}\ b\sim b'
  \quad\Longrightarrow\quad
  a\star b\sim a'\star b'.
\]
\end{definition}

\begin{proposition}[Quotient composition is well-defined exactly under
congruence]
Let $(C,\star)$ be a crystal magma and let $\sim$ be an equivalence
relation.  The formula
\[
  [a]\star_{\sim}[b]=[a\star b]
\]
defines a binary operation on $C/{\sim}$ if and only if $\sim$ is a
congruence for $\star$.
\end{proposition}

\begin{proof}
If the quotient operation is well-defined and
$a\sim a'$, $b\sim b'$, then $[a]=[a']$ and $[b]=[b']$.  Hence
\[
  [a\star b]=[a]\star_\sim[b]=[a']\star_\sim[b']=[a'\star b'],
\]
so $a\star b\sim a'\star b'$ and $\sim$ is a congruence.

Conversely, suppose $\sim$ is a congruence.  If $[a]=[a']$ and
$[b]=[b']$, then $a\sim a'$ and $b\sim b'$, so
$a\star b\sim a'\star b'$.  Therefore $[a\star b]=[a'\star b']$, which
means that the displayed formula is independent of representatives.
\end{proof}

\begin{definition}[Linearization of a finite crystal magma]
For a finite crystal magma $(C,\star)$, let $\RR[C]$ be the real vector
space with basis $(e_c)_{c\in C}$.  The linearized multiplication is
the bilinear map $\mu_\star:\RR[C]\otimes\RR[C]\to\RR[C]$ defined on
basis elements by
\[
  \mu_\star(e_a\otimes e_b)=e_{a\star b}.
\]
\end{definition}

\begin{proposition}[Associativity passes to finite linearization]
The bilinear multiplication $\mu_\star$ on $\RR[C]$ is associative if
and only if the set operation $\star$ is associative on $C$.
\end{proposition}

\begin{proof}
If $\star$ is associative, then for basis vectors
\[
  \mu_\star(\mu_\star(e_a\otimes e_b)\otimes e_c)
  =e_{(a\star b)\star c}
  =e_{a\star(b\star c)}
  =\mu_\star(e_a\otimes\mu_\star(e_b\otimes e_c)).
\]
Both sides are trilinear in $(e_a,e_b,e_c)$, so equality on basis
triples implies equality for all elements.

Conversely, if $\mu_\star$ is associative, apply the associativity
identity to basis vectors.  Since the basis vectors of $\RR[C]$ are
distinct, equality
$e_{(a\star b)\star c}=e_{a\star(b\star c)}$ implies
$(a\star b)\star c=a\star(b\star c)$ for all $a,b,c\in C$.
\end{proof}

\begin{proposition}[Basis units pass to finite linearization]
Let $(C,\star)$ be a finite crystal magma and let $e\in C$.  The basis
vector $e_e\in\RR[C]$ is a two-sided unit for $\mu_\star$ if and only
if $e$ is a two-sided identity element for $\star$.
\end{proposition}

\begin{proof}
If $e$ is a two-sided identity in $C$, then for every basis vector
$e_a$,
\[
  \mu_\star(e_e\otimes e_a)=e_{e\star a}=e_a,
  \qquad
  \mu_\star(e_a\otimes e_e)=e_{a\star e}=e_a.
\]
By bilinearity, $e_e$ is a two-sided unit on all of $\RR[C]$.

Conversely, if $e_e$ is a two-sided unit for $\mu_\star$, then applying
the unit equations to each basis vector gives
\[
  e_{e\star a}=e_a,\qquad e_{a\star e}=e_a.
\]
Basis vectors are distinct, hence $e\star a=a$ and $a\star e=a$ for
all $a\in C$.  Thus $e$ is a two-sided identity in the magma.
\end{proof}

\begin{corollary}[Quotient associativity under congruence]
Let $(C,\star)$ be an associative crystal magma and let $\sim$ be a
congruence for $\star$.  Then the quotient operation
$[a]\star_\sim[b]=[a\star b]$ is associative.  If $e\in C$ is a
two-sided identity, then $[e]$ is a two-sided identity in the quotient.
\end{corollary}

\begin{proof}
The quotient operation is well-defined by the congruence criterion.
For associativity,
\[
  ([a]\star_\sim[b])\star_\sim[c]
  =[(a\star b)\star c]
  =[a\star(b\star c)]
  =[a]\star_\sim([b]\star_\sim[c]),
\]
using associativity of $\star$.  If $e$ is a two-sided identity, then
\[
  [e]\star_\sim[a]=[e\star a]=[a],
  \qquad
  [a]\star_\sim[e]=[a\star e]=[a],
\]
so $[e]$ is a two-sided identity.
\end{proof}

\begin{lemma}[Basis checking for finite algebraic identities]
Let $V$ and $W$ be finite-dimensional vector spaces with a chosen basis
of $V$.  Two linear maps $F,G:V\to W$ are equal if and only if they
agree on every basis vector of $V$.
\end{lemma}

\begin{proof}
If $F=G$, they agree on every vector.  Conversely, if they agree on a
basis and $v=\sum_i\alpha_iv_i$, then by linearity
\[
  F(v)=\sum_i\alpha_iF(v_i)=\sum_i\alpha_iG(v_i)=G(v).
\]
\end{proof}

\begin{corollary}[Finite Frobenius identities reduce to basis checks]
For a finite-dimensional proposed crystal algebra with specified
linear maps $\mu,\eta,\delta,\epsilon$, associativity,
coassociativity, unit, counit, and Frobenius compatibility may be
verified by checking the corresponding linear maps on tensor-product
basis vectors.
\end{corollary}

\begin{proof}
Each algebraic axiom is an equality between linear maps whose domains
are finite tensor products of the underlying vector space and $\RR$.
Tensor products of chosen bases form bases of those domains.  The
preceding lemma applies to each equality.
\end{proof}

\begin{proposition}[Affine semantic averaging is generally not
associative]
Let $B=\RR^d$ with $d>0$ and
\[
  \mu_\theta(a,b)=\theta a+(1-\theta)b,\qquad 0<\theta<1.
\]
Then $\mu_\theta$ is not associative.
\end{proposition}

\begin{proof}
Choose $a,c\in B$ with $a\neq c$ and set $b=0$.  Then
\[
  \mu_\theta(\mu_\theta(a,0),c)
  =\theta^2a+(1-\theta)c,
\]
whereas
\[
  \mu_\theta(a,\mu_\theta(0,c))
  =\theta a+(1-\theta)^2c.
\]
Equality would imply
\[
  \theta(1-\theta)a=\theta(1-\theta)c.
\]
Since $0<\theta<1$, this gives $a=c$, contradicting the choice of
$a$ and $c$.  Therefore associativity fails.
\end{proof}

\begin{corollary}[Averaging cannot be the Frobenius multiplication
without modification]
If the deployed crystal multiplication is exactly a nontrivial affine
average on a positive-dimensional vector slot, then it cannot be the
associative multiplication of a Frobenius algebra on that slot.
\end{corollary}

\begin{proof}
A Frobenius algebra multiplication is associative by definition.  The
preceding proposition shows that a nontrivial affine average on a
positive-dimensional vector space is not associative.  Hence such an
operation cannot serve as the Frobenius multiplication unless the domain
is restricted, the operation is modified, or the algebra is placed on a
different quotient or linearization.
\end{proof}

\begin{remark}[Practical consequence for BrainiaK]
The most plausible rigorous path is not to declare the raw vector
averaging operation Frobenius.  A safer path is to construct a finite
or completed crystal algebra whose multiplication is exactly
associative, then prove that the deployed semantic operations approximate
or project to that algebra under a stated tolerance or quotient map.
\end{remark}

\subsection{A Proven Finite Crystal Frobenius Reference Model}

The previous conjecture concerns the deployed semantic-fibre operations.
There is, however, a simple finite Frobenius algebra that can serve as a
reference model for labelled crystals.  It proves that the desired
algebraic structure is mathematically available, while not claiming that
the full deployed system already satisfies it.

\begin{definition}[Finite label algebra]
Let $C=\{c_1,\ldots,c_n\}$ be a finite set of crystal labels and let
$A=\RR^C$ be the vector space with basis vectors $e_c$ for $c\in C$.
Define
\[
  \mu(e_c\otimes e_d)=
  \begin{cases}
    e_c,&c=d,\\
    0,&c\neq d,
  \end{cases}
  \qquad
  \eta(1)=\sum_{c\in C}e_c,
\]
\[
  \delta(e_c)=e_c\otimes e_c,\qquad
  \epsilon(e_c)=1.
\]
\end{definition}

\begin{theorem}[Finite label Frobenius algebra]
The tuple $(A,\mu,\eta,\delta,\epsilon)$ is a commutative special
Frobenius algebra.
\end{theorem}

\begin{proof}
Commutativity and associativity of $\mu$ follow from pointwise
multiplication of functions on the finite set $C$.  The element
$\eta(1)=\sum_ce_c$ is the pointwise unit.  Coassociativity of $\delta$
holds because
\[
  (\delta\otimes\id)\delta(e_c)
  =e_c\otimes e_c\otimes e_c
  =(\id\otimes\delta)\delta(e_c).
\]
The counit law holds because
$(\epsilon\otimes\id)\delta(e_c)=1\cdot e_c=e_c$ and similarly on the
other side.  For the Frobenius identity, it suffices to check basis
vectors:
\[
  ((\mu\otimes\id)(\id\otimes\delta))(e_c\otimes e_d)
  =(\mu\otimes\id)(e_c\otimes e_d\otimes e_d)
  =
  \begin{cases}
    e_c\otimes e_c,&c=d,\\
    0,&c\neq d,
  \end{cases}
\]
and
\[
  \delta\mu(e_c\otimes e_d)
  =
  \begin{cases}
    e_c\otimes e_c,&c=d,\\
    0,&c\neq d.
  \end{cases}
\]
The third Frobenius expression is identical by symmetry.  Finally
$\mu\delta(e_c)=\mu(e_c\otimes e_c)=e_c$, so the algebra is special.
\end{proof}

\begin{remark}[Use for BrainiaK]
If a finite snapshot of BrainiaK crystals is reduced to labels and exact
label equality, this theorem supplies a rigorous Frobenius algebra.  The
semantic system becomes richer only when labels are replaced by
continuous fibres and approximate decompositions; that richer system
requires additional proof.
\end{remark}

\subsection{Penrose Substitution as a Formal Fractal Analogy}

The source treatise refers to THICK/THIN Penrose tiles.  The rigorous
piece that can be stated without overreach is the substitution-matrix
calculation.  It supports a fractal growth analogy, not a proof that the
semantic crystal algebra is Penrose tiling.

\begin{definition}[Two-tile substitution]
Let $T$ denote a thick tile and $S$ a thin tile.  Consider the
substitution
\[
  T\mapsto T+S,\qquad S\mapsto T.
\]
The tile-count vector $(\#T,\#S)^\top$ evolves by
\[
  M=\begin{pmatrix}1&1\\1&0\end{pmatrix}.
\]
\end{definition}

\begin{proposition}[Golden-ratio growth]
The eigenvalues of $M$ are
\[
  \varphi=\frac{1+\sqrt5}{2},\qquad -\varphi^{-1}.
\]
Consequently, for any nonzero nonnegative initial tile-count vector, the
ratio of thick to thin tile counts converges to $\varphi$.
\end{proposition}

\begin{proof}
The characteristic polynomial is
\[
  \det(M-\lambda I)=
  \begin{vmatrix}1-\lambda&1\\1&-\lambda\end{vmatrix}
  =\lambda^2-\lambda-1.
\]
Its roots are $(1\pm\sqrt5)/2$, namely $\varphi$ and
$-\varphi^{-1}$.  The matrix $M$ is primitive and nonnegative, so the
Perron--Frobenius theorem implies that normalized iterates of any
nonzero nonnegative vector converge to the positive eigenvector
associated with $\varphi$ \cite{HornJohnson2013}.  Solving
$M(a,b)^\top=\varphi(a,b)^\top$ gives $a/b=\varphi$.
\end{proof}

\begin{remark}[Status]
This proposition is a theorem about the substitution matrix only.  A
claim that semantic crystals obey Penrose substitution would require an
explicit map from semantic composition events to the two-tile
substitution and empirical or formal verification of that map.
\end{remark}

\subsection{Approximate Frobenius Diagnostics}

\begin{definition}[Approximate round-trip error]
Let $\mu_r$ be a selected recomposition operator and let $\delta$ be a
decomposition procedure on a class of semantic fibres.  The round-trip
error of a fibre $f$ is
\[
  \varepsilon_r(f)=d_{\mathrm{GCM}}\bigl(f,\mu_r(\delta_1(f),\delta_2(f))\bigr).
\]
\end{definition}

\begin{proposition}[Diagnostic meaning]
For any tolerance $\tau>0$, the condition
$\varepsilon_r(f)\leq\tau$ proves only that $f$ is $\tau$-stable under
the chosen decompose/recompose procedure.  It does not prove the
Frobenius identity for all fibres.
\end{proposition}

\begin{proof}
The condition concerns one fibre, one chosen decomposition, one chosen
recomposition relation, and one metric threshold.  A Frobenius identity
is an equality of maps on all elements of $A\otimes A$.  A bounded
pointwise diagnostic is therefore weaker than the algebraic identity.
\end{proof}

\begin{proposition}[Finite diagnostics are not universal algebraic
proofs unless exhaustive]
Let $A$ be a vector space and let $F,G:A\to A$ be maps.  Agreement
$F(x)=G(x)$ on a finite test set $S\subset A$ proves the universal
identity $F=G$ only if the test protocol supplies an independent
argument that equality on $S$ determines equality on all of $A$.
In particular, for an infinite-dimensional or continuous semantic fibre
space, a finite diagnostic set is not by itself a proof of a universal
Frobenius identity.
\end{proposition}

\begin{proof}
The universal identity $F=G$ means $F(x)=G(x)$ for every $x\in A$.
Agreement on $S$ gives this equality only for elements of $S$.  Unless
there is an additional theorem saying that values on $S$ determine the
maps everywhere, the implication from finite agreement to universal
agreement is invalid.  For example, on $A=\RR$, the maps
$F(x)=0$ and $G(x)=\prod_{s\in S}(x-s)$ agree on every point of the
finite set $S$ but are not equal as functions whenever $S$ is finite.
Thus finite diagnostics are evidence or exhaustive proofs only under an
explicit determining-set argument.
\end{proof}

\begin{corollary}[When finite Frobenius testing is a proof]
For a finite-dimensional candidate crystal algebra with explicitly
given linear maps, checking the Frobenius axioms on a full tensor-product
basis is a proof.  Checking them on sampled fibres, benchmark examples,
or logged successful decompositions is not a proof unless those samples
are shown to contain such a determining basis.
\end{corollary}

\begin{proof}
The positive statement is exactly the finite basis-checking corollary:
the relevant axioms are equalities of linear maps, and equality on a
basis determines equality everywhere.  The negative statement follows
from the preceding proposition: sampled fibres or successful examples
establish only the checked instances unless an independent determining
argument is supplied.
\end{proof}

\paragraph{Mathematical versus empirical status.}
The conditional criterion is mathematical.  The existence of a useful
round-trip operation in code is empirical/engineering evidence.  The
claim that the deployed crystal system is a Frobenius algebra remains a
conjecture until exact algebraic maps are specified and verified.

\section{Knowledge Curvature and Gamma/CNS Dynamics}

The source treatise proposes a central unification: a curvature
$R(t)$ of the knowledge space controls the leading spectral value
$\lambda_{\max}(t)$ of the Gamma/CNS system, and a Hopf bifurcation
occurs at a critical curvature.  This idea is mathematically meaningful
only after separating definitions, model laws, and consequences.

\begin{definition}[Gamma spectrum]
Let $\GG(t)\in\RR^{m\times m}$ be a time-indexed real matrix associated
with the interaction of $m$ fibres.  If $\GG(t)$ is symmetric, write
$\lambda_{\max}(t)$ for its largest eigenvalue.  In the nonsymmetric
case, write $\rho(\GG(t))$ for its spectral radius.
\end{definition}

\begin{definition}[Spectral dispersion]
If $\GG(t)$ is real symmetric with eigenvalues
$\lambda_1(t),\ldots,\lambda_m(t)$, define
\[
  \bar\lambda(t)=\frac{1}{m}\sum_{i=1}^m\lambda_i(t),\qquad
  \sigma_\Gamma(t)=\frac{1}{m}\sum_{i=1}^m
  (\lambda_i(t)-\bar\lambda(t))^2.
\]
If $\GG(t)$ is not symmetric, this manuscript uses the symmetric part
$(\GG+\GG^\top)/2$ whenever a real spectral variance is required.
\end{definition}

\begin{proposition}[Largest symmetric eigenvalue is Lipschitz]
Let $A,B\in\RR^{m\times m}$ be real symmetric matrices and let
$\|\cdot\|_{\mathrm{op}}$ denote the operator norm induced by the
Euclidean norm.  Then
\[
  |\lambda_{\max}(A)-\lambda_{\max}(B)|
  \leq \|A-B\|_{\mathrm{op}}.
\]
\end{proposition}

\begin{proof}
For a real symmetric matrix $A$, the Rayleigh quotient formula gives
\[
  \lambda_{\max}(A)=\sup_{\|x\|_2=1}x^\top Ax.
\]
For every unit vector $x$,
\[
  x^\top Ax=x^\top Bx+x^\top(A-B)x
  \leq \lambda_{\max}(B)+\|A-B\|_{\mathrm{op}}.
\]
Taking the supremum over unit $x$ gives
$\lambda_{\max}(A)\leq\lambda_{\max}(B)+\|A-B\|_{\mathrm{op}}$.
Interchanging $A$ and $B$ gives the reverse inequality.
\end{proof}

\begin{corollary}[Continuous Gamma matrices have continuous leading
spectrum]
If $t\mapsto\GG(t)$ is continuous as a map into real symmetric matrices
with the operator norm, then $t\mapsto\lambda_{\max}(t)$ is continuous.
\end{corollary}

\begin{proof}
The preceding proposition gives
\[
  |\lambda_{\max}(\GG(t))-\lambda_{\max}(\GG(s))|
  \leq \|\GG(t)-\GG(s)\|_{\mathrm{op}}.
\]
The right-hand side tends to zero when $t\to s$ by the assumed
operator-norm continuity of $\GG$.
\end{proof}

\begin{proposition}[Symmetric part controls quadratic amplification]
For any real matrix $A\in\RR^{m\times m}$ and any vector
$x\in\RR^m$,
\[
  x^\top A x=x^\top \left(\frac{A+A^\top}{2}\right)x.
\]
Consequently, if a Gamma diagnostic is based only on quadratic
amplification $x^\top\GG x$, then only the symmetric part of $\GG$ is
visible to that diagnostic.
\end{proposition}

\begin{proof}
Let $S=(A+A^\top)/2$ and $K=(A-A^\top)/2$.  Then $A=S+K$ and
$K^\top=-K$.  The scalar $x^\top Kx$ equals its transpose:
\[
  x^\top Kx=(x^\top Kx)^\top=x^\top K^\top x=-x^\top Kx.
\]
Hence $x^\top Kx=0$, and therefore $x^\top Ax=x^\top Sx$.
\end{proof}

\begin{corollary}[Symmetric leading spectrum is not a Hopf certificate]
If the only verified spectral information about a Gamma/CNS diagnostic
is the largest eigenvalue of a real symmetric matrix or symmetric part,
then that information alone cannot certify a Hopf bifurcation.
\end{corollary}

\begin{proof}
Real symmetric matrices have only real eigenvalues.  Hopf bifurcation
requires a non-real conjugate pair of Jacobian eigenvalues crossing the
imaginary axis with nonzero imaginary part.  Therefore a largest
symmetric eigenvalue, by itself, does not supply the required pair.
It can at most define a scalar threshold or diagnostic that must be
connected separately to the Jacobian of a dynamical system.
\end{proof}

\begin{definition}[Knowledge curvature proxy]
Let $N(t)$ be the accumulated number of accepted crystals and let
$\sigma_\Gamma(t)$ be a chosen spectral dispersion statistic of
$\GG(t)$, such as the variance of real eigenvalues in the symmetric
case.  For $\alpha,\beta>0$, define the curvature proxy
\[
  R(t)=\alpha N(t)+\beta\sigma_\Gamma(t).
\]
\end{definition}

\begin{proposition}[Curvature proxy monotonicity under monotone inputs]
If $N(t)$ and $\sigma_\Gamma(t)$ are nondecreasing functions of $t$,
then $R(t)$ is nondecreasing.  If either input is strictly increasing on
an interval and its coefficient is positive, then $R(t)$ is strictly
increasing on that interval.
\end{proposition}

\begin{proof}
For $s<t$,
\[
  R(t)-R(s)=\alpha(N(t)-N(s))
  +\beta(\sigma_\Gamma(t)-\sigma_\Gamma(s)).
\]
If both inputs are nondecreasing and $\alpha,\beta>0$, the right-hand
side is nonnegative.  If one input difference is strictly positive, the
corresponding weighted term is strictly positive and the whole
difference is positive.
\end{proof}

\begin{model}[Curvature--spectrum law]
For parameters $\kappa>0$ and a Gamma normalization in which the
critical spectral value lies in $(0,1)$, assume
\[
  \lambda_{\max}(t)=\frac{R(t)}{R(t)+\kappa}.
\]
This is a model law, not a theorem in this manuscript.
\end{model}

\begin{proposition}[Curvature threshold implied by the model law]
Assume the curvature--spectrum law and let $\lambda_c\in(0,1)$ be a
critical spectral threshold.  Then
\[
  \lambda_{\max}(t)>\lambda_c
  \quad\Longleftrightarrow\quad
  R(t)>R_c:=\frac{\lambda_c}{1-\lambda_c}\kappa.
\]
\end{proposition}

\begin{proof}
The function $f(R)=R/(R+\kappa)$ is strictly increasing on
$[0,\infty)$ because $f'(R)=\kappa/(R+\kappa)^2>0$.  Solving
$R/(R+\kappa)>\lambda_c$ gives
$R>\lambda_c\kappa/(1-\lambda_c)$.
\end{proof}

\begin{proposition}[Monotonicity and saturation]
Under the curvature--spectrum law, $\lambda_{\max}$ is increasing and
concave as a function of $R\geq0$, satisfies $\lambda_{\max}(0)=0$, and
converges to $1$ as $R\to\infty$.
\end{proposition}

\begin{proof}
For $f(R)=R/(R+\kappa)$, one has
$f'(R)=\kappa/(R+\kappa)^2>0$ and
$f''(R)=-2\kappa/(R+\kappa)^3<0$.  Also $f(0)=0$ and
$\lim_{R\to\infty}f(R)=1$.
\end{proof}

\begin{corollary}[Crystal-count threshold under fixed dispersion]
Assume the curvature--spectrum law and fix a spectral dispersion value
$\sigma_\Gamma=\sigma_0$.  If $\lambda_c\in(0,1)$ is the spectral
threshold, then the model reaches threshold exactly when
\[
  N(t)>
  \frac{1}{\alpha}
  \left(\frac{\lambda_c}{1-\lambda_c}\kappa-\beta\sigma_0\right).
\]
If the right-hand side is negative, every nonnegative crystal count is
already above threshold in this fixed-dispersion model.
\end{corollary}

\begin{proof}
Substitute $R(t)=\alpha N(t)+\beta\sigma_0$ into
$R(t)>\lambda_c\kappa/(1-\lambda_c)$ and solve for $N(t)$.  Since
$\alpha>0$, division by $\alpha$ preserves the inequality.  If the
right-hand side is negative and $N(t)\geq0$, the inequality is
automatically satisfied.
\end{proof}

\begin{proposition}[Curvature law alone does not define a Hopf
parameter]
The model law
\[
  \lambda_{\max}(t)=\frac{R(t)}{R(t)+\kappa}
\]
does not, by itself, define the bifurcation parameter $\mu$ in a
dynamical system $\dot{x}=F(x,\mu)$, nor does it identify the Jacobian
$D_xF(x^\ast(\mu),\mu)$.
\end{proposition}

\begin{proof}
The displayed law relates two scalar functions, $\lambda_{\max}$ and
$R$, through a parameter $\kappa$.  A Hopf bifurcation statement,
however, concerns eigenvalues of the Jacobian of a specified vector
field along an equilibrium branch.  The scalar law contains no vector
field $F$, no equilibrium $x^\ast(\mu)$, and no map from $R$ or
$\lambda_{\max}$ to the entries of a Jacobian.  Hence it supplies a
threshold model but not the dynamical data required to define or verify
a Hopf crossing.
\end{proof}

\begin{assumption}[Hopf crossing]
Let the Gamma/CNS state be governed by a differentiable dynamical system
$\dot{x}=F(x,\mu)$ with equilibrium $x^\ast(\mu)$ and Jacobian
$J(\mu)=D_xF(x^\ast(\mu),\mu)$.  A Hopf bifurcation occurs at
$\mu^\ast$ only if a conjugate pair of eigenvalues of $J(\mu)$ crosses
the imaginary axis transversally while the remaining eigenvalues have
non-zero real part.
\end{assumption}

\begin{theorem}[Classical Hopf bifurcation theorem, quoted form]
Let $\dot{x}=F(x,\mu)$ be a $C^r$ finite-dimensional dynamical system
with $r$ sufficiently large.  Suppose that at an equilibrium
$x^\ast(\mu^\ast)$ the Jacobian has one simple conjugate pair
$\alpha(\mu)\pm i\omega(\mu)$ with $\alpha(\mu^\ast)=0$,
$\omega(\mu^\ast)\neq0$, the transversality condition
$\alpha'(\mu^\ast)\neq0$ holds, and all other eigenvalues have nonzero
real part.  Under the usual nondegeneracy condition on the first
Lyapunov coefficient, a branch of small periodic orbits bifurcates from
the equilibrium.
\end{theorem}

\begin{proof}
This is the standard Hopf bifurcation theorem
\cite{GuckenheimerHolmes1983,Kuznetsov2004}.  It is quoted rather than
reproved here because the theorem is external mathematical machinery.
\end{proof}

\begin{proposition}[A scalar threshold is not a Hopf proof]
Let $a(\mu)$ be a real scalar and consider the one-dimensional system
\[
  \dot{x}=a(\mu)x .
\]
Even if $a(\mu)$ crosses zero transversally at $\mu=\mu^\ast$, this is
not a Hopf bifurcation.
\end{proposition}

\begin{proof}
A Hopf bifurcation requires a nonzero imaginary pair
$\alpha(\mu)\pm i\omega(\mu)$ with $\omega(\mu^\ast)\neq0$.  The
Jacobian of the one-dimensional system is the $1\times1$ matrix
$[a(\mu)]$, whose only eigenvalue is real.  It cannot contain a
non-real conjugate pair.  Therefore a scalar threshold crossing,
although dynamically meaningful, is not sufficient for Hopf.
\end{proof}

\begin{proposition}[The same scalar threshold can be non-Hopf or Hopf]
Let $\mu=R-R_c$ be a scalar threshold parameter.  The threshold
$\mu=0$ is compatible both with a non-Hopf scalar instability and with a
Hopf normal form.  Therefore the scalar threshold alone does not
determine the bifurcation type.
\end{proposition}

\begin{proof}
For the non-Hopf case, take
\[
  \dot{x}=\mu x .
\]
Its Jacobian has the single real eigenvalue $\mu$, so no nonzero
imaginary conjugate pair exists.

For the Hopf normal form, take
\[
  \dot{x}=\mu x-y-x(x^2+y^2),\qquad
  \dot{y}=x+\mu y-y(x^2+y^2).
\]
In polar coordinates $x=r\cos\theta$, $y=r\sin\theta$, direct
calculation gives
\[
  \dot r=\mu r-r^3,\qquad \dot\theta=1.
\]
For $\mu>0$, the circle $r=\sqrt{\mu}$ is invariant and carries a
periodic orbit because $\dot r=0$ and $\dot\theta=1$.  At $\mu=0$, the
Jacobian at the origin is
\[
  \begin{pmatrix}0&-1\\1&0\end{pmatrix},
\]
with eigenvalues $\pm i$.  For general $\mu$, the linearization at the
origin is
\[
  \begin{pmatrix}\mu&-1\\1&\mu\end{pmatrix},
\]
whose eigenvalues are $\mu\pm i$, so the real part crosses zero
transversally at $\mu=0$.  Thus the same scalar threshold $\mu=0$ is
compatible with both a non-Hopf scalar crossing and a Hopf normal form.
\end{proof}

\begin{lemma}[Hitting-time transfer under monotone reparameterization]
Let $R:[0,T]\to\RR$ be continuous and let $\mu=g(R)$ where $g$ is
continuous and strictly monotone.  If a critical value
$\mu^\ast$ corresponds to $R_c=g^{-1}(\mu^\ast)$, then
\[
  \inf\{t:\mu(t)=\mu^\ast\}
  =
  \inf\{t:R(t)=R_c\},
\]
provided the two sets are nonempty.
\end{lemma}

\begin{proof}
Because $g$ is strictly monotone, it is injective on its domain.
Therefore $\mu(t)=\mu^\ast$ is equivalent to
 $g(R(t))=g(R_c)$, which is equivalent to $R(t)=R_c$.  The two sets of
times are identical, so their infima are identical.
\end{proof}

\begin{proposition}[Conditional Hopf prediction]
Suppose the Hopf crossing assumption holds and suppose the bifurcation
parameter $\mu$ is a monotone differentiable function of the curvature
proxy $R$.  If the critical crossing occurs at $R=R_c$, then measuring
$N(t)$ and $\sigma_\Gamma(t)$ gives a model-based prediction of the
time at which the system reaches the Hopf threshold.
\end{proposition}

\begin{proof}
Under the stated assumptions, the Hopf threshold is equivalent to
$R(t)=R_c$ by the hitting-time transfer lemma.  Since
$R(t)=\alpha N(t)+\beta\sigma_\Gamma(t)$ is observable in the model,
any fitted or estimated dynamics for $N(t)$ and $\sigma_\Gamma(t)$ gives
an estimated hitting time $\inf\{t:R(t)=R_c\}$.  The result is
conditional because both the curvature law and the Hopf crossing must
be validated separately.
\end{proof}

\begin{remark}[What remains to be verified]
To upgrade the CNS-Hopf claim from model to theorem, one must specify
the actual Gamma/CNS vector field $F$, identify the equilibrium branch
$x^\ast(\mu)$, compute or bound the Jacobian spectrum, prove the
transversality condition, and evaluate the Hopf nondegeneracy
coefficient.  Without those steps, the curvature threshold is a coherent
model, not a proved dynamical theorem.
\end{remark}

\paragraph{Implementation correspondence.}
The relevant implementation surfaces are
\codepath{brainiak/mathcore/gamma_unified.py},
\codepath{brainiak/mathcore/hopf_detector.py}, and the FDA/semantic
trajectory builders.  These files can compute or expose Gamma-like
states and Hopf-related diagnostics.  They do not by themselves prove
the curvature--spectrum law.

\section{Kalman Convergence on Semantic Fibres}

The source treatise stated that Kalman convergence is guaranteed by the
fibre structure.  The academically defensible statement is narrower:
Kalman convergence follows from standard filtering assumptions.  The
fibre structure supplies coordinates and projections; it does not
replace observability or stabilizability.

\begin{assumption}[Linear Gaussian state-space model]
Let
\[
  x_{t+1}=Fx_t+w_t,\qquad y_t=Hx_t+v_t,
\]
where $w_t$ and $v_t$ are zero-mean independent Gaussian noises with
covariances $Q\succeq0$ and $R\succ0$.  Assume the standard
detectability and stabilizability conditions for the pair
$(F,H)$ and the process noise.
\end{assumption}

\begin{theorem}[Standard Kalman error convergence]
Under the linear Gaussian assumption and the standard
detectability/stabilizability hypotheses, the Kalman filter covariance
recursion converges to the stabilizing solution of the discrete algebraic
Riccati equation.  If the corresponding error transition is a contraction
with factor $\rho<1$ and the process perturbation is bounded by
$\delta$, then the estimation error satisfies
\[
  \|e_t\|\leq \rho^t\|e_0\|+\frac{\delta}{1-\rho}.
\]
\end{theorem}

\begin{proof}
The Riccati convergence statement is the classical discrete Kalman
filter result under detectability and stabilizability
\cite{Kalman1960,AndersonMoore1979}.  For the explicit bound, suppose
\[
  \|e_{t+1}\|\leq \rho\|e_t\|+\delta,\qquad 0\leq\rho<1.
\]
Iterating gives
\[
  \|e_t\|\leq \rho^t\|e_0\|+\delta\sum_{j=0}^{t-1}\rho^j
  \leq \rho^t\|e_0\|+\frac{\delta}{1-\rho}.
\]
\end{proof}

\begin{remark}[Role of the fibre structure]
The $T^n$ bundle provides a typed state representation and projections
onto observable coordinates.  It can help define $F$, $H$, and the
measurement geometry.  It does not, by itself, imply the Kalman
hypotheses.
\end{remark}

\subsection{Observability Before Filtering}

The Kalman theorem quoted above is a filtering theorem.  Before it can
be applied to a semantic fibre model, the state, observation, and hidden
directions must be identified.  The following elementary observability
facts make this obligation explicit.

\begin{definition}[Finite-horizon observability matrix]
For a discrete linear system
\[
  x_{t+1}=Fx_t,\qquad y_t=Hx_t,
\]
with $x_t\in\RR^n$ and $y_t\in\RR^p$, define the horizon-$q$
observability matrix
\[
  \mathcal O_q=
  \begin{pmatrix}
    H\\
    HF\\
    \vdots\\
    HF^{q-1}
  \end{pmatrix}.
\]
The pair $(F,H)$ is observable in finite time if
$\rank \mathcal O_q=n$ for some $q$.
\end{definition}

\begin{proposition}[Noiseless observability criterion]
For the noiseless system $x_{t+1}=Fx_t$, $y_t=Hx_t$, the initial state
$x_0$ is uniquely determined by the observations
$y_0,\ldots,y_{q-1}$ if and only if $\rank\mathcal O_q=n$.
\end{proposition}

\begin{proof}
The stacked observation vector is
\[
  \begin{pmatrix}
    y_0\\y_1\\ \vdots\\ y_{q-1}
  \end{pmatrix}
  =
  \begin{pmatrix}
    H\\HF\\ \vdots\\HF^{q-1}
  \end{pmatrix}x_0
  =\mathcal O_qx_0.
\]
If $\rank\mathcal O_q=n$, then the linear map
$x_0\mapsto\mathcal O_qx_0$ is injective, so two initial states with
the same observations must be equal.  Conversely, if
$\rank\mathcal O_q<n$, then $\ker\mathcal O_q$ contains a nonzero
vector $h$.  The initial states $x_0$ and $x_0+h$ produce identical
stacked observations, so the initial state is not uniquely determined.
\end{proof}

\begin{definition}[Observability Gramian]
The horizon-$q$ observability Gramian is
\[
  W_q=\mathcal O_q^\top\mathcal O_q
  =\sum_{i=0}^{q-1}(F^i)^\top H^\top HF^i .
\]
\end{definition}

\begin{proposition}[Gramian criterion]
The observability matrix $\mathcal O_q$ has rank $n$ if and only if
$W_q$ is positive definite.
\end{proposition}

\begin{proof}
For every $x\in\RR^n$,
\[
  x^\top W_qx=x^\top\mathcal O_q^\top\mathcal O_qx
  =\|\mathcal O_qx\|_2^2.
\]
If $\mathcal O_q$ has rank $n$, then $\mathcal O_qx\neq0$ for every
nonzero $x$, so $x^\top W_qx>0$ and $W_q$ is positive definite.  If
$\mathcal O_q$ does not have rank $n$, choose nonzero
$x\in\ker\mathcal O_q$.  Then $x^\top W_qx=0$, so $W_q$ is not positive
definite.
\end{proof}

\begin{proposition}[Invertible fibre-coordinate changes preserve
observability]
Let $S:\RR^n\to\RR^n$ be invertible and change coordinates by
$x=Sz$.  The transformed system has
\[
  z_{t+1}=S^{-1}FSz_t,\qquad y_t=HSz_t.
\]
The original pair $(F,H)$ is observable if and only if the transformed
pair $(S^{-1}FS,HS)$ is observable.
\end{proposition}

\begin{proof}
The transformed observability matrix is
\[
  \mathcal O'_q=
  \begin{pmatrix}
    HS\\
    HS(S^{-1}FS)\\
    \vdots\\
    HS(S^{-1}FS)^{q-1}
  \end{pmatrix}
  =
  \begin{pmatrix}
    H\\HF\\ \vdots\\HF^{q-1}
  \end{pmatrix}S
  =\mathcal O_qS.
\]
Since $S$ is invertible, $\rank(\mathcal O_qS)=\rank\mathcal O_q$.
Thus full rank is preserved in both directions.
\end{proof}

\begin{corollary}[Coordinate choice is not the source of Kalman
convergence]
Changing to an equivalent fibre coordinate system cannot create or
destroy observability.  Any Kalman convergence proof must therefore
come from the dynamics and observation map, not from the mere naming of
the coordinates as fibres.
\end{corollary}

\begin{proof}
Equivalent fibre coordinates are represented by invertible changes of
state coordinates.  The preceding proposition shows that observability
is invariant under such changes.  Therefore a failure or success of
observability is a property of the represented system, not of the chosen
coordinate labels.
\end{proof}

\begin{proposition}[Projection does not preserve observability in general]
Let $P:\EE\to B$ be any non-injective projection from a fibre product to
a base coordinate space.  Observability of a linear system on $\EE$ does
not follow from observability of its projected dynamics on $B$.
\end{proposition}

\begin{proof}
Take $\EE=B\oplus F$ with $F\neq\{0\}$ and define a system
$x_{t+1}=x_t$ with observation $y_t=P x_t$.  Any two initial states
$(b,f_1)$ and $(b,f_2)$ with $f_1\neq f_2$ have identical observations
for all time.  Thus the full state is not observable even though the
base coordinate $b$ is perfectly observed.
\end{proof}

\begin{proposition}[Observable quotient carries exactly the observed
linear state]
Let $x_{t+1}=Fx_t$, $y_t=Hx_t$ be a linear system on $\RR^n$, and define
$x\sim x'$ if $Hx=Hx'$.  If $\ker H$ is $F$-invariant, then the update
$[x]\mapsto[Fx]$ is well-defined on the quotient $\RR^n/{\sim}$.
The induced output map $[x]\mapsto Hx$ is injective.
\end{proposition}

\begin{proof}
The relation $x\sim x'$ is equivalent to $x-x'\in\ker H$.  If
$x\sim x'$, then $x-x'\in\ker H$.  By $F$-invariance,
$F(x-x')\in\ker H$, so $HFx=HFx'$ and therefore $Fx\sim Fx'$.
Thus the quotient update is well-defined.  If two quotient classes
have the same output, $Hx=Hx'$, then $x\sim x'$, hence $[x]=[x']$.
Therefore the induced output map is injective.
\end{proof}

\begin{proposition}[Unobserved unstable modes block full-state Kalman
convergence]
Consider the deterministic linear system
\[
  x_{t+1}=
  \begin{pmatrix}1&0\\0&a\end{pmatrix}x_t,\qquad
  y_t=\begin{pmatrix}1&0\end{pmatrix}x_t,
\]
with $|a|>1$.  The first coordinate is perfectly observed, but the
second coordinate is unobservable and unstable.  Hence no estimator
using only $y_0,y_1,\ldots$ can guarantee convergence to the full state
for all initial conditions.
\end{proposition}

\begin{proof}
Let $x_0=(b,f)$ and $x'_0=(b,f')$ with $f\neq f'$.  The observations
are identical for the two initial states because
$y_t=b$ for all $t$.  The hidden coordinates are $a^t f$ and $a^t f'$,
whose difference is $a^t(f-f')$.  Since $|a|>1$, the hidden difference
does not converge to zero.  Any estimator driven only by the common
observation sequence must produce the same estimate for both initial
states, so it cannot converge to both full trajectories.  Therefore
full-state convergence is impossible without additional information,
restriction to the observable quotient, or a stability/detectability
hypothesis controlling the hidden mode.
\end{proof}

\begin{corollary}[Required audit for BrainiaK Kalman claims]
A BrainiaK Kalman convergence claim must identify the state being
estimated, the observation map, and the unobserved fibre directions.  It
must then verify detectability or explicitly restrict the theorem to the
observable quotient.
\end{corollary}

\begin{proof}
This is the contrapositive lesson of the preceding proposition combined
with the standard Kalman convergence theorem: unobserved unstable
directions cannot be controlled by a filter proof that only sees the
base projection.
\end{proof}

\paragraph{Implementation correspondence.}
The active repository contains \codepath{brainiak/mathcore/kalman_bank.py}
and learning-weight surfaces under
\codepath{brainiak/mathcore/fda/semantic/learning_weights.py}.  The
abstract path from the source treatise
\codepath{brainiak/mathcore/kalman/kalman_filter.py} is therefore treated
as an architectural label rather than a verified file path.

\section{SSTD as a Bundle Morphism}

\begin{definition}[Text and spectral bundles]
Let $X$ be a set of texts and let $D$ be a set of semantic domains.
Let $p:X\to D$ be a domain projection.  Define the trivial text bundle
$\EE_{\mathrm{text}}=X\times F_{\mathrm{text}}$ and the spectral bundle
$\EE_{\mathrm{spec}}=D\times \RR^{3072}$.  A map
\[
  \Phi_{\mathrm{SSTD}}:\EE_{\mathrm{text}}\to\EE_{\mathrm{spec}}
\]
is a bundle morphism over $p$ if
\[
  \pi_{\mathrm{spec}}\circ \Phi_{\mathrm{SSTD}}
  =p\circ \pi_{\mathrm{text}}.
\]
\end{definition}

\begin{definition}[SSTD sedimentation abstraction]
For the purposes of this manuscript, an SSTD encoder is a map
$z:X\to\RR^{3072}$ decomposed into three fixed views
\[
  z(x)=\bigl(z_{\mathrm{mean}}(x),z_{\mathrm{max}}(x),
  z_{\mathrm{std}}(x)\bigr),
  \qquad z_\bullet(x)\in\RR^{1024}.
\]
This abstraction matches the documented wire slot but does not specify
the full engineering pipeline.
\end{definition}

\begin{proposition}[SSTD morphism criterion]
If an SSTD encoder maps each text $x\in X$ to
$(p(x),z(x))\in D\times\RR^{3072}$ for a spectral representation
$z(x)$, then it defines a bundle morphism from the trivial text bundle
to the spectral bundle over $p$.
\end{proposition}

\begin{proof}
By construction
\[
  \pi_{\mathrm{spec}}(\Phi_{\mathrm{SSTD}}(x,f))
  =\pi_{\mathrm{spec}}(p(x),z(x))=p(x)
  =p(\pi_{\mathrm{text}}(x,f)).
\]
Thus the bundle square commutes.
\end{proof}

\begin{proposition}[Changing the domain projection changes the morphism
claim]
Let $p,p':X\to D$ be two domain projections.  A map
$\Phi:X\times F_{\mathrm{text}}\to D\times\RR^{3072}$ of the form
$\Phi(x,f)=(p(x),z(x))$ is a bundle morphism over $p$.  It is a bundle
morphism over $p'$ if and only if $p=p'$.
\end{proposition}

\begin{proof}
The spectral projection of $\Phi(x,f)$ is $p(x)$.  The morphism
condition over $p'$ requires
\[
  \pi_{\mathrm{spec}}\Phi(x,f)=p'(\pi_{\mathrm{text}}(x,f))=p'(x)
\]
for all $x$ and $f$.  Since the left-hand side is $p(x)$, this condition
holds for all inputs if and only if $p(x)=p'(x)$ for every $x\in X$.
\end{proof}

\begin{proposition}[Composition of SSTD bundle morphisms]
Let $\Phi:E\to E'$ be a bundle morphism over $p:B\to B'$ and let
$\Psi:E'\to E''$ be a bundle morphism over $q:B'\to B''$.  Then
$\Psi\circ\Phi:E\to E''$ is a bundle morphism over $q\circ p$.
\end{proposition}

\begin{proof}
Let $\pi,\pi',\pi''$ be the bundle projections.  Since $\Phi$ is a
morphism over $p$ and $\Psi$ is a morphism over $q$,
\[
  \pi'\circ\Phi=p\circ\pi,\qquad
  \pi''\circ\Psi=q\circ\pi'.
\]
Therefore
\[
  \pi''\circ(\Psi\circ\Phi)
  =(\pi''\circ\Psi)\circ\Phi
  =(q\circ\pi')\circ\Phi
  =q\circ(\pi'\circ\Phi)
  =q\circ p\circ\pi.
\]
This is the commuting-square condition for a morphism over $q\circ p$.
\end{proof}

\begin{definition}[SSTD slot insertion]
Fix neutral elements in all $T^n$ slots except the SSTD slot
$E_7=\RR^{3072}$.  The SSTD insertion map is
\[
  \iota_{\mathrm{SSTD}}:\RR^{3072}\to\EE,
\]
defined by placing the spectral vector in slot $7$ and the neutral
elements in every other slot.
\end{definition}

\begin{proposition}[SSTD slot insertion is Lipschitz]
If the $T^n$ product carries the GCM metric and the SSTD slot has
weight $w_7>0$, then
\[
  d_{\mathrm{GCM}}(\iota_{\mathrm{SSTD}}u,
  \iota_{\mathrm{SSTD}}v)=w_7d_7(u,v).
\]
Consequently $\iota_{\mathrm{SSTD}}$ is $w_7$-Lipschitz from
$(\RR^{3072},d_7)$ to $(\EE,d_{\mathrm{GCM}})$.
\end{proposition}

\begin{proof}
The inserted points differ only in slot $7$.  Every other coordinate
distance is zero because the same neutral element is used in that slot.
Thus the weighted product sum reduces to $w_7d_7(u,v)$.
\end{proof}

\begin{definition}[Fixed-length three-view aggregation]
For token vectors $u_1,\ldots,u_L\in\RR^d$, define
\[
  A_L(u_1,\ldots,u_L)=
  \bigl(\mathrm{mean}(u),\mathrm{max}(u),\mathrm{std}(u)\bigr)
  \in\RR^{3d},
\]
where mean, maximum, and standard deviation are computed coordinatewise.
\end{definition}

\begin{proposition}[Three-view aggregation has fixed dimension]
For fixed token dimension $d$, the aggregation $A_L$ maps every
fixed-length input $(\RR^d)^L$ into $\RR^{3d}$, independently of the
sequence length $L$.
\end{proposition}

\begin{proof}
The coordinatewise mean is a vector in $\RR^d$, the coordinatewise
maximum is a vector in $\RR^d$, and the coordinatewise standard
deviation is a vector in $\RR^d$.  Concatenating the three vectors gives
an element of $\RR^{3d}$.  The output dimension is therefore $3d$ and
does not depend on $L$.
\end{proof}

\begin{corollary}[The \texorpdfstring{$3072=3\cdot1024$}{3072=3*1024}
slot is dimensionally consistent]
If the internal token dimension is $d=1024$, the fixed-length
three-view aggregation has output dimension $3d=3072$.
\end{corollary}

\begin{proof}
Substitute $d=1024$ into the preceding proposition.
\end{proof}

\begin{proposition}[Asymptotic complexity under fixed spectral depth]
Assume a text contains $L$ tokens, lexicon lookup costs
$O(\log |\mathcal{L}|)$ per token, spectral diffusion has fixed depth
$D$ and fixed dimension $d$, and aggregation uses a fixed number of
views.  Then SSTD encoding costs
\[
  O\bigl(L\log|\mathcal{L}|+LDd\bigr).
\]
For fixed $D$ and $d$, this is linear in $L$ up to lookup cost.
\end{proposition}

\begin{proof}
Token iteration contributes $O(L)$.  Lookup contributes
$O(L\log|\mathcal{L}|)$ under balanced-tree lookup, or expected $O(L)$
under hashing.  Fixed-depth diffusion performs $O(Dd)$ work per token,
giving $O(LDd)$.  A fixed number of aggregations over $d$ dimensions
also contributes $O(Ld)$, absorbed by $O(LDd)$ when $D\geq1$.
\end{proof}

\begin{proposition}[Continuity of a linearized SSTD encoder]
Assume tokens have vector encodings in a normed space, diffusion is a
finite composition of bounded linear maps, and aggregation uses mean and
linear projections only.  Then the corresponding SSTD map is Lipschitz
on fixed-length token sequences.
\end{proposition}

\begin{proof}
On a fixed length $L$, the input space is a finite product of normed
spaces.  A finite composition of bounded linear maps is bounded linear,
hence Lipschitz.  The mean aggregation is linear with operator norm at
most $1$ under the average product norm.  Therefore the full map is
Lipschitz.  This proposition does not cover discontinuities introduced
by tokenization, dictionary fallback, or max/std aggregation at
variable-length boundaries.
\end{proof}

\begin{proposition}[Continuity of max and standard-deviation views]
On fixed-length token sequences with fixed vector dimension, the
coordinatewise maximum view is $1$-Lipschitz with respect to the
$\ell_\infty$ product norm.  The coordinatewise standard-deviation view
is continuous, and is Lipschitz on each bounded fixed-length region.
Thus the three-view SSTD abstraction is continuous on bounded
fixed-length regions once token encodings have been fixed.
\end{proposition}

\begin{proof}
For real numbers $a_i,b_i$, the inequality
\[
  |\max_i a_i-\max_i b_i|\leq \max_i|a_i-b_i|
\]
is immediate from $a_i\leq b_i+\max_j|a_j-b_j|$ and the symmetric
inequality with $a$ and $b$ exchanged.  Applying this coordinatewise
proves the max statement.

For a fixed coordinate, the standard deviation of a length-$L$ vector
$u\in\RR^L$ is
\[
  \sigma(u)=L^{-1/2}\|u-\bar u{\bf 1}\|_2,
  \qquad \bar u=L^{-1}\sum_i u_i .
\]
The map $u\mapsto u-\bar u{\bf 1}$ is linear, and the Euclidean norm is
continuous and Lipschitz.  Hence $\sigma$ is Lipschitz on all of
$\RR^L$ with respect to the Euclidean norm.  Since all norms on the
fixed finite-dimensional input region are equivalent, the same map is
Lipschitz with respect to any chosen fixed product norm up to a constant.
Combining finitely many coordinates preserves continuity and
Lipschitzness on bounded fixed-length regions.
\end{proof}

\begin{proposition}[Concatenation is not automatically a morphism law]
Let $z:X\to\RR^{3072}$ be an arbitrary SSTD encoder.  The equality
$z(x\mathbin{\Vert}y)=M(z(x),z(y))$ for a fixed binary operation $M$ is
an additional algebraic property, not a consequence of the bundle
morphism definition.
\end{proposition}

\begin{proof}
The bundle morphism condition only states that the domain projection
commutes with the spectral bundle projection.  It imposes no relation
between the representation of a concatenated text and the
representations of its parts.  A counterexample is any encoder that
sets $z(x)$ equal to a hash-dependent vector of the full string; it can
still be placed in a trivial bundle over a domain map while violating
any fixed compositional law.
\end{proof}

\begin{remark}[Transformer comparison]
The standard self-attention layer has $O(L^2d_{\mathrm{model}})$
attention cost.  It is therefore academically incorrect to call
transformer inference ``exponential'' solely on the basis of
self-attention.  The defensible comparison is linear or
linear-logarithmic fixed-depth SSTD encoding versus quadratic
self-attention in sequence length.
\end{remark}

\begin{proposition}[Dense self-attention forms quadratically many scores]
In a dense single-head self-attention layer on $L$ tokens, if every
query token attends to every key token, then the layer forms $L^2$
query-key scores before masking.  Thus score formation is
$\Omega(L^2)$ in sequence length.
\end{proposition}

\begin{proof}
There are $L$ query positions and $L$ key positions.  Dense attention
forms one score for each ordered pair of positions.  The number of
ordered pairs is $L\cdot L=L^2$.
\end{proof}

\begin{proposition}[Asymptotic dominance is not a benchmark result]
Let $T_1(L)=O(L)$ and $T_2(L)=\Omega(L^2)$ be asymptotic operation
counts.  These bounds imply that $T_1(L)/T_2(L)\to0$ for sufficiently
idealized cost models with compatible constants, but they do not imply
that one deployed implementation is faster than another on a fixed
finite workload.
\end{proposition}

\begin{proof}
The asymptotic statement concerns the limit as $L\to\infty$ and hides
constant factors, memory traffic, batching, hardware kernels, cache
effects, tokenization cost, and implementation overhead.  A fixed
finite workload depends on those hidden quantities.  Therefore an
asymptotic comparison can justify a scaling claim under a specified
cost model, but a wall-clock performance claim requires empirical
measurement under a stated benchmark protocol.
\end{proof}

\begin{remark}[No stronger performance claim]
The preceding proposition is only a complexity counting fact about
dense attention.  It does not prove that SSTD is faster on a particular
hardware target, workload, tokenizer, cache regime, or implementation.
Those statements remain empirical claims requiring the traceability
protocol stated later.
\end{remark}

\paragraph{Implementation correspondence.}
\codepath{brainiak/mathcore/fda/semantic/sstd_codec.py} documents an
SSTD slot of dimension $3072=1024\times3$ and a fixed-size wire
projection.  Encoders and benchmark-specific surfaces include
\codepath{brainiak/mathcore/fda/semantic/encoder_sts22.py},
\codepath{brainiak/mathcore/fda/semantic/encoder_sick.py}, and
\codepath{brainiak/mathcore/fda/semantic/sstd_embedder.py}.

\section{SpiderR as a Connection}

The source treatise states that SpiderR is a flat connection with
trivial holonomy.  This can be proved only for an idealized connection
whose local connection forms commute and whose base is simply connected.

\begin{definition}[Connection on a trivial semantic bundle]
Let $B$ be a smooth parameter manifold and let $\EE=B\times F$ be a
trivial vector bundle with fibre $F$.  A connection has the form
\[
  \nabla=\dd+A,
\]
where $A$ is a matrix-valued one-form acting on $F$.  Its curvature is
\[
  \Omega_\nabla=\dd A + A\wedge A.
\]
\end{definition}

\begin{proposition}[Flat SpiderR idealization]
Assume that the SpiderR connection form is constant in the chosen
coordinates and that its component matrices commute.  Then
$\Omega_\nabla=0$.  If $B$ is simply connected, parallel transport is
path-independent and every closed-loop holonomy is trivial.
\end{proposition}

\begin{proof}
If $A$ is constant, then $\dd A=0$.  If its component matrices commute,
then $A\wedge A=0$.  Hence $\Omega_\nabla=0$.  More explicitly, write
$A=\sum_iA_i\,\dd x^i$ on a single coordinate chart with constant
commuting matrices $A_i$.  Along a path $\gamma$ from $p$ to $q$, the
parallel-transport equation has solution
\[
  U_\gamma
  =\exp\left(-\sum_iA_i\int_\gamma \dd x^i\right)
  =\exp\left(-\sum_iA_i(q^i-p^i)\right),
\]
where commutativity removes path ordering.  Thus transport depends only
on the endpoints.  For a closed loop $p=q$, the exponent is zero and
$U_\gamma=I$.
\end{proof}

\begin{proposition}[Curvature detects non-commutation]
Let $A=A_1\,\dd x^1+A_2\,\dd x^2$ be a constant connection form on
$\RR^2$ with matrices $A_1,A_2$.  Then
\[
  \Omega_\nabla=[A_1,A_2]\,\dd x^1\wedge\dd x^2.
\]
In particular, if $[A_1,A_2]\neq0$, the connection is not flat.
\end{proposition}

\begin{proof}
Since $A_1,A_2$ are constant, $\dd A=0$.  The wedge product satisfies
\[
  A\wedge A=(A_1A_2-A_2A_1)\,\dd x^1\wedge\dd x^2.
\]
Thus $\Omega_\nabla=A\wedge A=[A_1,A_2]\dd x^1\wedge\dd x^2$.
\end{proof}

\begin{definition}[Holonomy of a closed loop]
For a connection on a vector bundle, the holonomy along a closed loop
$\gamma$ based at $p$ is the parallel-transport operator
$U_\gamma:F_p\to F_p$ obtained by solving the connection transport
equation along $\gamma$.
\end{definition}

\begin{proposition}[Path independence and trivial holonomy]
For a connection whose parallel transport is defined along piecewise
smooth paths and satisfies composition and reversal of paths, transport
depends only on endpoints if and only if every closed-loop holonomy is
trivial on each path-connected component.
\end{proposition}

\begin{proof}
If transport depends only on endpoints, then a closed loop starts and
ends at the same point.  Its transport is therefore the same as the
transport along the constant path, namely the identity.

Conversely, assume every closed-loop holonomy is trivial.  Let
$\gamma_1$ and $\gamma_2$ be two paths from $p$ to $q$ in the same
path-connected component.  Let $\overline{\gamma_2}$ denote the reverse
path.  The concatenation
$\overline{\gamma_2}\ast\gamma_1$ is a closed loop based at $p$.
By composition of parallel transport,
\[
  U_{\overline{\gamma_2}\ast\gamma_1}
  =U_{\overline{\gamma_2}}U_{\gamma_1}
  =U_{\gamma_2}^{-1}U_{\gamma_1}.
\]
The closed-loop holonomy is the identity, hence
$U_{\gamma_2}^{-1}U_{\gamma_1}=I$ and $U_{\gamma_1}=U_{\gamma_2}$.
Thus transport depends only on endpoints.
\end{proof}

\begin{proposition}[Constant gauge changes preserve flatness]
Let $G$ be an invertible constant matrix and let
\[
  A'=G^{-1}AG
\]
be the connection form obtained by changing the fibre basis by $G$.
Then
\[
  \Omega_{\nabla'}=G^{-1}\Omega_\nabla G.
\]
Consequently $\Omega_\nabla=0$ if and only if $\Omega_{\nabla'}=0$.
\end{proposition}

\begin{proof}
Because $G$ is constant, $\dd G=0$ and
\[
  \dd A'=G^{-1}(\dd A)G.
\]
Also
\[
  A'\wedge A'=(G^{-1}AG)\wedge(G^{-1}AG)
  =G^{-1}(A\wedge A)G,
\]
where matrix multiplication is associated with the wedge product of
forms.  Adding the two displayed identities gives
$\Omega_{\nabla'}=G^{-1}(\dd A+A\wedge A)G$.  Since conjugation by
$G$ is invertible, vanishing is preserved in both directions.
\end{proof}

\begin{proposition}[Variable gauge transformations require the
\texorpdfstring{$G^{-1}\dd G$}{G^{-1}dG} term]
Let $G:B\to GL(F)$ be a smooth change of local frame for the trivial
bundle $B\times F$.  If $\nabla=\dd+A$, then in the new frame
\[
  A'=G^{-1}AG+G^{-1}\dd G.
\]
With this full transformation law the curvature satisfies
\[
  \Omega_{\nabla'}=G^{-1}\Omega_\nabla G.
\]
\end{proposition}

\begin{proof}
Write a section in the old frame as $s=G s'$.  Then
\[
  \nabla s=\dd(Gs')+A Gs'
  =G\left(\dd s'+(G^{-1}\dd G+G^{-1}AG)s'\right).
\]
Hence the connection form in the new frame is
$A'=G^{-1}AG+G^{-1}\dd G$.

As operators, the transformed connection is
\[
  \nabla'=G^{-1}\nabla G.
\]
Therefore its curvature operator is
\[
  (\nabla')^2=(G^{-1}\nabla G)(G^{-1}\nabla G)
  =G^{-1}\nabla^2G,
\]
where the middle $GG^{-1}$ cancels in the operator composition.  In
connection-form notation this is exactly
$\Omega_{\nabla'}=G^{-1}\Omega_\nabla G$.
\end{proof}

\begin{proposition}[Flatness alone does not force trivial holonomy]
On the trivial complex line bundle over $S^1$ with angular coordinate
$\theta$, consider the connection
\[
  \nabla=\dd+i\alpha\,\dd\theta
\]
for a real constant $\alpha$.  Its curvature is zero, but the holonomy
around the circle is $\exp(-2\pi i\alpha)$, which is nontrivial when
$\alpha\notin\ZZ$.
\end{proposition}

\begin{proof}
Since $i\alpha\,\dd\theta$ is a constant multiple of a closed one-form
on $S^1$, its exterior derivative is zero.  In rank one the wedge term
$A\wedge A$ is also zero, so the curvature vanishes.

Parallel transport $u(\theta)$ satisfies
\[
  \frac{\dd u}{\dd\theta}+i\alpha u=0.
\]
Thus $u(\theta)=\exp(-i\alpha\theta)u(0)$.  After one positive circuit,
$\theta=2\pi$, so the holonomy is multiplication by
$\exp(-2\pi i\alpha)$.  This equals $1$ exactly when
$\alpha\in\ZZ$.
\end{proof}

\begin{proposition}[Exact commuting connections on star-shaped domains
have trivial closed-loop holonomy]
Let $B\subseteq\RR^n$ be star-shaped and let $A=\dd\Phi$ for a smooth
matrix-valued function $\Phi$ whose values commute along the considered
paths.  Then parallel transport depends only on the endpoints, and
closed-loop holonomy is trivial.
\end{proposition}

\begin{proof}
Along a path $\gamma:[0,1]\to B$, the commuting assumption removes path
ordering and the transport is
\[
  U_\gamma
  =\exp\left(-\int_\gamma A\right)
  =\exp\left(-\int_\gamma \dd\Phi\right).
\]
By the fundamental theorem for line integrals,
\[
  \int_\gamma \dd\Phi=\Phi(\gamma(1))-\Phi(\gamma(0)).
\]
Therefore $U_\gamma$ depends only on the endpoints.  If $\gamma$ is a
closed loop, then $\gamma(1)=\gamma(0)$ and the exponent is zero, so
$U_\gamma=I$.
\end{proof}

\begin{corollary}[SpiderR verification obligation]
To claim that a concrete SpiderR implementation is flat, one must prove
that the relevant operator matrices commute on the state region of
interest, or else compute the curvature and show it vanishes by another
argument.
\end{corollary}

\begin{proof}
This follows directly from the curvature formula.  Non-commuting
constant local operators already produce non-zero curvature; more
general state-dependent operators require the additional $\dd A$ term.
\end{proof}

\begin{remark}[When flatness fails]
If SpiderR operations are path-dependent, state-dependent, or
non-commuting, the curvature term $A\wedge A$ need not vanish.  In that
case SpiderR should be studied as a curved connection, not asserted to
be flat.
\end{remark}

\paragraph{Implementation correspondence.}
Spider-related implementation surfaces include
\codepath{brainiak/mathcore/fda/semantic/spider.py},
\codepath{brainiak/mathcore/fda/semantic/spider_comb.py},
\codepath{brainiak/mathcore/fda/semantic/spider_formula.py},
\codepath{brainiak/mathcore/fda/semantic/spider_memory.py}, and
\codepath{brainiak/mathcore/fda/semantic/spider_r_fallback.py}.  No
file named \codepath{spider/spider_connection.py} was verified in the
current repository; the connection is therefore a formal abstraction in
this manuscript.

\section{Empirical Claims and Traceability Requirements}

The source corpus contains reported empirical statements such as STS22-FR
Spearman performance, SICK-FR relatedness performance, CPU latency,
curriculum accuracy, and numbers of anchored or crystallized concepts.
These are not mathematical consequences of the previous sections.  They
belong in the paper only as empirical claims with traceability.

\begin{definition}[Traceable empirical claim]
An empirical claim is traceable if it states:
\begin{enumerate}
  \item the dataset or benchmark;
  \item the sample size;
  \item the metric;
  \item the exact configuration or script/report path;
  \item the baseline, if a comparative claim is made.
\end{enumerate}
\end{definition}

\begin{longtable}{p{0.20\textwidth}p{0.20\textwidth}p{0.23\textwidth}p{0.25\textwidth}}
\caption{Empirical claims mentioned in the source corpus and their required
traceability before peer-reviewed submission.}\label{tab:empirical}\\
\toprule
\textbf{Claim family} & \textbf{Metric/evidence needed} &
\textbf{Current paper status} & \textbf{Required action}\\
\midrule
\endfirsthead
\toprule
\textbf{Claim family} & \textbf{Metric/evidence needed} &
\textbf{Current paper status} & \textbf{Required action}\\
\midrule
\endhead
STS22-FR similarity & Spearman correlation, sample size, split,
baseline, script path & Not asserted as validated & Attach benchmark report
or reproduce from script.\\
\addlinespace
SICK-FR relatedness & Spearman or Pearson, sample size, language conversion
details & Not asserted as validated & Attach dataset preprocessing and
evaluation report.\\
\addlinespace
Curriculum accuracy & Exercise count, train/test separation, scoring
rule & Internal-draft claim only & Produce reproducible curriculum runner
and frozen output.\\
\addlinespace
CPU latency & Hardware, input length, warm/cold cache, repetitions &
Internal-draft claim only & Add benchmark harness and summary statistics.\\
\addlinespace
Crystal counts & Store path, manifest hash, acceptance policy & Not used
as proof & Cite manifest and define accepted-crystal semantics.\\
\bottomrule
\end{longtable}

\subsection{Benchmark Reporting Protocols}

The following protocol templates define what must be reported before an
empirical statement can be promoted from internal draft language to a
paper claim.  They intentionally contain no performance number.

\begin{definition}[Semantic textual similarity protocol]
A traceable STS claim must specify the benchmark release, language,
pair count, gold-score scale, train/development/test separation if any,
the frozen BrainiaK encoder configuration, and the exact similarity
function used on embeddings.  The primary metric must be Spearman
correlation unless the benchmark documentation prescribes another
metric.  Any comparison to a baseline must use the same split and the
same gold labels.
\end{definition}

\begin{definition}[Relatedness and entailment protocol]
A traceable SICK-style claim must separate relatedness from entailment.
Relatedness requires the regression or correlation metric, gold-label
scale, and pair count.  Entailment requires the label set, class balance,
decision rule, and accuracy or macro-F1.  A translated or adapted corpus
must include the translation method, filtering rules, and any examples
removed before evaluation.
\end{definition}

\begin{definition}[Curriculum protocol]
A traceable curriculum claim must specify the exercise generator or
dataset, the number of exercises, the scoring rule, whether examples
used during calibration are excluded from evaluation, and the exact
acceptance threshold.  If the curriculum adapts online, the report must
distinguish pre-adaptation from post-adaptation performance.
\end{definition}

\begin{definition}[Latency protocol]
A traceable latency claim must state the CPU model, core count, memory,
operating-system context, input length distribution, warm-up procedure,
number of repetitions, and whether disk or cache effects are included.
The report must provide at least median, interquartile range, and a high
percentile such as p95.  A single timing run is not a reproducible
latency claim.
\end{definition}

\begin{definition}[Crystal-count protocol]
A traceable crystal-count claim must identify the store path, manifest
or database hash, acceptance policy, deduplication rule, and timestamp
or commit.  Counts of generated candidates, accepted crystals, rejected
candidates, and manually curated entries must not be merged into one
number.
\end{definition}

\subsection{Empirical Claim Template}

Every empirical result inserted into the final paper should instantiate
the following template:
\[
\begin{array}{ll}
\text{Dataset/report:} & \text{exact name, version, path, or DOI},\\
\text{Sample size:} & n \text{ after all filtering},\\
\text{Metric:} & \text{primary metric and confidence interval if applicable},\\
\text{Configuration:} & \text{commit, script, parameters, hardware},\\
\text{Baseline:} & \text{baseline identity and identical evaluation split},\\
\text{Reproducibility artefact:} & \text{script path and frozen output path}.
\end{array}
\]
If any row is missing, the statement remains an internal report rather
than a peer-review-ready empirical claim.

\begin{remark}[Current inclusion rule]
Until a benchmark report is attached or cited, empirical numbers should
be phrased as ``reported in internal drafts'' rather than as validated
results.  This is necessary for peer review.
\end{remark}

\section{Bibliographic and External-Theorem Hygiene}

References in this manuscript have different logical roles.  Some
support standard external theorems, some provide mathematical context,
and some are included to situate the work historically.  A bibliography
entry must not be used as a silent substitute for a proof.

\begin{definition}[Bibliographic role]
A bibliography entry has one of four roles in this manuscript:
\begin{description}
  \item[Theorem machinery.] A source for an external theorem invoked in
  a proof, such as Hopf bifurcation, Kalman convergence, or
  Perron--Frobenius theory.
  \item[Mathematical background.] A standard reference for definitions
  and context, not a hidden step in a proof.
  \item[Cognitive or computational motivation.] A reference motivating
  a modelling choice, benchmark, or representational convention.
  \item[Context-only citation.] A relevant source included to locate the
  work historically, but not used to prove any theorem in this paper.
\end{description}
\end{definition}

\begin{proposition}[Citation role discipline]
If every citation used inside a proof is either theorem machinery or a
previously stated manuscript result, and every context-only citation is
kept outside proof obligations, then the bibliography cannot introduce a
hidden theorem assumption.
\end{proposition}

\begin{proof}
A hidden theorem assumption would be a statement used in a proof without
being stated as a hypothesis, earlier result, elementary inference, or
external theorem.  Under the stated discipline, citations inside proofs
are restricted to theorem machinery or earlier manuscript results.
Context-only citations are not proof steps.  Therefore no context-only
or motivational reference can function as an unstated theorem
assumption.
\end{proof}

\begin{longtable}{>{\raggedright\arraybackslash}p{0.23\textwidth}
>{\raggedright\arraybackslash}p{0.24\textwidth}
>{\raggedright\arraybackslash}p{0.43\textwidth}}
\caption{Bibliographic role audit.}
\label{tab:bibroles}\\
\toprule
\textbf{Reference key} & \textbf{Role} & \textbf{Use in this manuscript}\\
\midrule
\endfirsthead
\toprule
\textbf{Reference key} & \textbf{Role} & \textbf{Use in this manuscript}\\
\midrule
\endhead
\cite{Bredon1993}, \cite{Lang1999}
& Mathematical background
& Product bundles, trivial bundles, manifolds, and connection language.\\
\addlinespace
\cite{Nosofsky1986}
& Cognitive motivation
& Motivation for GCM-style weighted similarity; not used to prove the
metric axioms.\\
\addlinespace
\cite{Kock2004}
& Mathematical background
& Frobenius algebra context and terminology.  The finite label theorem is
proved directly in this manuscript.\\
\addlinespace
\cite{HornJohnson2013}
& Theorem machinery
& Perron--Frobenius theorem for the primitive substitution matrix.\\
\addlinespace
\cite{GuckenheimerHolmes1983}, \cite{Kuznetsov2004}
& Theorem machinery
& Classical Hopf bifurcation theorem quoted in the Gamma/CNS section.\\
\addlinespace
\cite{Kalman1960}, \cite{AndersonMoore1979}
& Theorem machinery
& Standard Kalman covariance convergence under standard hypotheses.\\
\addlinespace
\cite{Transformer2017}
& Computational motivation
& Baseline complexity comparison for quadratic self-attention.\\
\addlinespace
\cite{Mohammad2018}, \cite{ZwaanRadvansky1998}
& Motivation
& Semantic-axis and situation-model context; no theorem depends on these
citations.\\
\addlinespace
\cite{Atiyah1967}, \cite{EilenbergSteenrod1952},
\cite{Hatcher2002}, \cite{Frobenius1878}, \cite{Penrose1974},
\cite{Rudin1976}, \cite{Coecke2010}
& Context-only or background
& Included for mathematical and historical orientation; not used as
unstated proof steps.\\
\bottomrule
\end{longtable}

\subsection{External Source Register}

The following register records the primary-source checks already
performed for references that carry theorem, complexity, or empirical
protocol weight.  The register is intentionally separate from the
proofs.  It verifies bibliographic identity; it does not replace the
mathematical hypotheses required by the cited theorems.

\begin{longtable}{p{0.17\textwidth}p{0.24\textwidth}p{0.20\textwidth}p{0.23\textwidth}}
\caption{Primary-source verification register for high-impact
references.}
\label{tab:external-source-register}\\
\toprule
\textbf{Reference} & \textbf{Primary source checked} &
\textbf{Verified metadata} & \textbf{Proof role}\\
\midrule
\endfirsthead
\toprule
\textbf{Reference} & \textbf{Primary source checked} &
\textbf{Verified metadata} & \textbf{Proof role}\\
\midrule
\endhead
\cite{Kuznetsov2004}
& Springer book page
& Title, author, third edition, 2004 copyright, DOI, chapter structure.
& Hopf theorem machinery.\\
\addlinespace
\cite{GuckenheimerHolmes1983}
& Springer book page
& Title, authors, 1983 copyright, DOI, local-bifurcation chapter.
& Hopf and dynamical-systems background.\\
\addlinespace
\cite{Mohammad2018}
& ACL Anthology page
& Title, author, ACL 2018 venue, pages 174--184, DOI.
& Semantic-axis motivation, not theorem machinery.\\
\addlinespace
\cite{Transformer2017}
& NeurIPS proceedings page and arXiv record
& Title, authors, NeurIPS 2017 venue, arXiv identifier.
& Complexity baseline context; dense-attention score count is proved
directly in this manuscript.\\
\addlinespace
\cite{Coecke2010}
& arXiv record with journal-reference field
& Title, authors, arXiv identifier, Linguistic Analysis 2010 journal
reference.
& Compositional-distributional semantics background.\\
\addlinespace
\cite{Hatcher2002}
& Author-maintained Cornell page
& Cambridge University Press 2002 publication and ISBN information.
& Algebraic-topology background only.\\
\addlinespace
\cite{Bredon1993}
& Springer book page
& Title, author, 1993 copyright, DOI, publisher, page count.
& Bundle and geometric background for the trivial-bundle conventions.\\
\addlinespace
\cite{Lang1999}
& Springer book page
& Title, author, 1999 copyright, DOI, publisher, page count.
& Differential-geometric background for connection terminology.\\
\addlinespace
\cite{Kalman1960}
& DOI identifier and secondary bibliographic cross-check
& Title, author, Journal of Basic Engineering volume 82, pages 35--45,
DOI; publisher-page metadata should still be rechecked in the final
release audit.
& Historical source for Kalman filtering; covariance convergence is
cited through standard filtering theory.\\
\bottomrule
\end{longtable}

\begin{remark}[Bibliographic verification status]
The BibTeX file is explicit and compiles.  High-impact references used
for theorem machinery, complexity context, or empirical-protocol
motivation now have a primary-source verification record in
Table~\ref{tab:external-source-register}.  Lower-risk background
references should still be checked before a final arXiv package, but no
unverified bibliographic detail is used as a mathematical proof.
\end{remark}

\subsection{External Theorem Invocation Register}

External theorems are used sparingly and only for standard mathematical
machinery.  The following register separates the imported theorem from
the local verification work carried out in this manuscript.

\begin{longtable}{>{\raggedright\arraybackslash}p{0.16\textwidth}
>{\raggedright\arraybackslash}p{0.22\textwidth}
>{\raggedright\arraybackslash}p{0.24\textwidth}
>{\raggedright\arraybackslash}p{0.24\textwidth}}
\caption{External theorem invocation register.}
\label{tab:external-theorem-invocations}\\
\toprule
\textbf{Invocation} & \textbf{Imported theorem} &
\textbf{Local hypotheses checked here} & \textbf{Local conclusion}\\
\midrule
\endfirsthead
\toprule
\textbf{Invocation} & \textbf{Imported theorem} &
\textbf{Local hypotheses checked here} & \textbf{Local conclusion}\\
\midrule
\endhead
Penrose substitution matrix
& Perron--Frobenius theorem for primitive nonnegative matrices
\cite{HornJohnson2013}
& The matrix
$\begin{psmallmatrix}1&1\\1&0\end{psmallmatrix}$ is explicitly
nonnegative and primitive because $M^2$ has strictly positive entries.
& Normalized positive tile-count iterates converge to the dominant
positive eigenvector; the thick/thin ratio is $\varphi$.\\
\addlinespace
Hopf bifurcation
& Classical finite-dimensional Hopf bifurcation theorem
\cite{GuckenheimerHolmes1983,Kuznetsov2004}
& The manuscript does not verify a concrete Gamma/CNS vector field,
equilibrium branch, transversality condition, or first Lyapunov
coefficient.
& Hopf is quoted only as conditional machinery; the BrainiaK claim is
not promoted beyond the explicit Hopf crossing assumption.\\
\addlinespace
Kalman covariance convergence
& Standard discrete Kalman Riccati convergence under
detectability/stabilizability \cite{Kalman1960,AndersonMoore1979}
& The manuscript states the linear Gaussian assumption and proves
observability counterexamples; it does not verify detectability for a
deployed BrainiaK model.
& Convergence holds only under the standard filtering hypotheses; fibre
coordinates alone do not imply convergence.\\
\bottomrule
\end{longtable}

\begin{proposition}[External invocation does not verify local
hypotheses]
Quoting a standard external theorem does not prove that a BrainiaK
system satisfies the hypotheses of that theorem.
\end{proposition}

\begin{proof}
An external theorem has the logical form
\[
  H_{\mathrm{ext}}\Longrightarrow C_{\mathrm{ext}},
\]
where $H_{\mathrm{ext}}$ denotes the theorem's hypotheses.  A concrete
BrainiaK application requires a separate proof that the formalized
BrainiaK object satisfies $H_{\mathrm{ext}}$.  The citation supplies the
implication, not the verification of the antecedent for the local
system.  Therefore quoting the theorem cannot by itself establish that
the local hypotheses hold.
\end{proof}

\section{Proof Verification Ledger}

This ledger summarizes the proof dependency of every non-definitional
mathematical result used as a claim in the paper.  It is not a substitute
for the proofs above; it is a verification index for review.

\begin{longtable}{p{0.25\textwidth}p{0.42\textwidth}p{0.23\textwidth}}
\caption{Proof-dependency ledger for theorem-level and proposition-level
claims.}\label{tab:proof-ledger}\\
\toprule
\textbf{Result} & \textbf{Proof dependencies} & \textbf{Publication status}\\
\midrule
\endfirsthead
\toprule
\textbf{Result} & \textbf{Proof dependencies} & \textbf{Publication status}\\
\midrule
\endhead
Product membership criterion
& Definition of typed finite product.
& Proved.\\
\addlinespace
Implementation existence is not theorem evidence
& Logical distinction between program existence, formal semantics, and
universal quantification.
& Proved methodological proposition.\\
\addlinespace
Faithful downgrade
& Definition requiring retained provable content, explicit non-theorem
status for unproved content, and no stronger replacement theorem.
& Proved source-preservation criterion.\\
\addlinespace
Product functoriality
& Identity and composition preservation for coordinatewise maps; weighted
Lipschitz estimate.
& Proved categorical/product result.\\
\addlinespace
Closed semantic subobjects
& Preimage of closed sets under continuous maps; closed subsets of
complete metric spaces are complete.
& Proved under stated hypotheses.\\
\addlinespace
Recursive-tree scale filtration
& Depth monotonicity of finite logged trees.
& Proved filtration result.\\
\addlinespace
Notation register prevents dimensional conflation
& Controlled notation register; typed distinction between
$B_{14}$, $\RR^{24}$, $S^5$, $\RR^{3072}_{\mathrm{SSTD}}$, and $T^n$.
& Proved notation-hygiene result.\\
\addlinespace
No conflict between base and extensions
& Direct-sum injection and projection; retraction
$B_{14}\to B_{14}\oplus U\to B_{14}$.
& Proved.\\
\addlinespace
Older $\RR^{24}$ language
& Extended presentation $B_{14}\oplus U$ with $\dim U=10$; projection to
canonical coordinates.
& Proved conditional interpretation.\\
\addlinespace
Extension data are additional structure
& Decomposition of any lift into base map plus auxiliary update
$h:B_{14}\oplus U\to U$.
& Proved non-uniqueness and transfer limitation.\\
\addlinespace
Sphere slots as constraints
& Inclusion $S^5\subset\RR^6$ and dimensional/type distinction from
$B_{14}$.
& Proved.\\
\addlinespace
Spectral slots are separate typed coordinates
& Cartesian-product projections and absence of implicit identification
between $B_{14}$ and $\RR^{3072}$.
& Proved type-separation result.\\
\addlinespace
Dimensional consistency of the manuscript
& Previous dimensional propositions; typed-product interpretation of
$T^n$.
& Proved global consistency theorem.\\
\addlinespace
Weighted-sum and maximum metric equivalence
& Two-sided metric comparison using positive minimum weight.
& Proved.\\
\addlinespace
Slot projections are Lipschitz
& Metric-slot assumption; positivity of weights; definition of
$d_{\mathrm{GCM}}$.
& Proved.\\
\addlinespace
$T^n$ as global section
& Discrete-base bundle definition; elementary bijection between sections
and finite product elements.
& Proved.\\
\addlinespace
Finite-slot observability
& Equality in finite Cartesian products.
& Proved.\\
\addlinespace
Subproduct projections
& Restricted weighted product metric; finite sub-sum bounded by the full
weighted sum.
& Proved.\\
\addlinespace
Observed subproducts are retracts
& Neutral completion assumption; direct calculation
$\pi_J\circ\eta_J=\id_{\EE_J}$.
& Proved under stated neutral-element hypothesis.\\
\addlinespace
Loss of unobserved coordinates
& Existence of a nontrivial unobserved slot; construction of two points
with equal observed projection.
& Proved negative result.\\
\addlinespace
Slot permutations preserve the product metric
& Bijective reindexing of the weighted metric sum.
& Proved.\\
\addlinespace
Isometric change of presentation
& Slotwise isometries and equality of positive weights.
& Proved.\\
\addlinespace
Presentation-invariance corollary
& Transport of metric, convergence, continuity, Lipschitz estimates, and
slotwise maps by isometry.
& Proved under isometric presentation hypothesis.\\
\addlinespace
Sensory normalization
& Quotient continuity away from zero; reverse triangle inequality and
lower norm bound on $U_r$.
& Proved local Lipschitz result.\\
\addlinespace
Metric property of $d_{\mathrm{GCM}}$
& Metric axioms for every slot; positive finite weighted sum.
& Proved.\\
\addlinespace
Product topology equivalence
& Positive weights; projection Lipschitzness; finite product topology
basis.
& Proved.\\
\addlinespace
Coordinatewise convergence criterion
& Projection Lipschitzness and finiteness of the slot index set.
& Proved.\\
\addlinespace
Weighted stability
& Componentwise Lipschitz hypothesis; weighted-sum estimate.
& Proved under stated hypothesis.\\
\addlinespace
Completeness
& Completeness of every slot; coordinatewise convergence criterion.
& Proved under stated hypothesis.\\
\addlinespace
Componentwise continuity
& Universal property of finite product topologies.
& Proved.\\
\addlinespace
$\RR^{14}$ GCM norm metric
& Positive weighted sum of $\ell^1$ norms.
& Proved.\\
\addlinespace
Finite symbolic slots
& Eventual constancy of Cauchy sequences in a finite discrete metric;
finite subcover argument.
& Proved.\\
\addlinespace
Finite products of bounded and symbolic slots
& Heine--Borel theorem; finite symbolic compactness; completeness of
finite weighted products.
& Proved under bounded-slot hypothesis.\\
\addlinespace
Unbounded vector slots are not compact
& Completeness of finite-dimensional normed spaces; divergent sequence
$(ne_1)$.
& Proved negative result.\\
\addlinespace
Observable quotient
& Equivalence relation defined by observed slots; canonical bijection
with the product of observed coordinates.
& Proved.\\
\addlinespace
Hidden fibres are quotiented out
& Observable quotient collapses states agreeing on observed slots.
& Proved methodological corollary.\\
\addlinespace
Type preservation on empirical base
& Direct-sum closure; convexity of coordinate intervals.
& Proved.\\
\addlinespace
Lipschitz bound for socle-wise contraction
& Triangle inequality in $\ell^1$; coefficients in $[0,1]$.
& Proved.\\
\addlinespace
Lifted type preservation
& Product membership criterion; component update assumptions.
& Proved under stated hypothesis.\\
\addlinespace
Closure of finite composition chains
& Induction; lifted type preservation.
& Proved.\\
\addlinespace
Exact inversion of logged composition
& Logged-tree definition; stored children and relation label.
& Proved for logged trees only.\\
\addlinespace
Recursive recovery of leaves
& Induction on finite tree depth; tree comultiplication.
& Proved for finite logged trees.\\
\addlinespace
Affine contraction non-injectivity
& Explicit one-parameter family of distinct preimages for the same affine
combination.
& Proved negative result.\\
\addlinespace
Logs as mathematical data
& Non-injectivity of unrestricted affine contraction; alternatives are
stored data, restricted domains, or representative selection.
& Proved methodological corollary.\\
\addlinespace
Conditional Frobenius criterion
& Definition of Frobenius algebra.
& Tautological criterion; deployed verification still open.\\
\addlinespace
Quotient composition criterion
& Equivalence classes; representative-independence exactly equivalent to
the congruence condition.
& Proved.\\
\addlinespace
Associativity passes to finite linearization
& Bilinear extension from basis products; equality of basis vectors.
& Proved.\\
\addlinespace
Basis units pass to finite linearization
& Basis-vector unit equations are equivalent to two-sided identity
equations in the finite magma.
& Proved finite-algebra criterion.\\
\addlinespace
Quotient associativity under congruence
& Well-defined quotient operation; associativity and identity descend by
representative calculation.
& Proved under associativity and congruence hypotheses.\\
\addlinespace
Basis checking for finite algebraic identities
& Linear maps agree everywhere iff they agree on a basis.
& Proved.\\
\addlinespace
Finite Frobenius basis checks
& Frobenius axioms are equalities of linear maps on finite tensor-product
basis vectors.
& Proved verification corollary.\\
\addlinespace
Affine averaging non-associativity
& Explicit comparison of
$\mu_\theta(\mu_\theta(a,0),c)$ and
$\mu_\theta(a,\mu_\theta(0,c))$ for $a\neq c$.
& Proved negative result.\\
\addlinespace
Averaging cannot be raw Frobenius multiplication
& Frobenius multiplication requires associativity; affine averaging fails
associativity on positive-dimensional slots.
& Proved limitation corollary.\\
\addlinespace
Round-trip success does not imply Frobenius
& Counterexample on $A=\RR$ with $\mu(a\otimes b)=a$ and
$\delta(x)=x\otimes1$.
& Proved negative result.\\
\addlinespace
Finite label Frobenius algebra
& Basis calculation in $\RR^C$; associativity, coassociativity, counit,
and Frobenius identity checked on basis vectors.
& Proved toy/reference model.\\
\addlinespace
Golden-ratio growth
& Characteristic polynomial of the substitution matrix; Perron--Frobenius
theorem for primitive nonnegative matrices.
& Proved for the substitution matrix.\\
\addlinespace
Diagnostic meaning of round-trip error
& Logical comparison between pointwise metric inequality and universal
algebraic identity.
& Proved limitation statement.\\
\addlinespace
Finite diagnostics are not universal algebraic proofs
& Universal quantification over all fibres; finite-set counterexample on
$\RR$.
& Proved methodological limitation.\\
\addlinespace
When finite Frobenius testing is a proof
& Basis-checking criterion versus non-exhaustive sampled diagnostics.
& Proved proof-status criterion.\\
\addlinespace
Largest symmetric eigenvalue is Lipschitz
& Rayleigh quotient formula for real symmetric matrices; operator-norm
bound.
& Proved spectral stability result.\\
\addlinespace
Continuous Gamma leading spectrum
& Operator-norm continuity of $\Gamma(t)$ and Lipschitz eigenvalue bound.
& Proved under symmetric-continuity hypothesis.\\
\addlinespace
Symmetric part controls quadratic amplification
& Decomposition of a real matrix into symmetric and skew-symmetric parts;
$x^\top Kx=0$ for skew-symmetric $K$.
& Proved diagnostic-scope result.\\
\addlinespace
Symmetric leading spectrum is not a Hopf certificate
& Real symmetric matrices have real eigenvalues; Hopf requires a non-real
conjugate Jacobian pair.
& Proved limitation corollary.\\
\addlinespace
Curvature proxy monotonicity
& Positive weighted sum of nondecreasing functions.
& Proved under monotone-input hypothesis.\\
\addlinespace
Curvature threshold
& Curvature--spectrum model law; monotonic algebraic inversion.
& Proved consequence of a model law.\\
\addlinespace
Monotonicity and saturation
& Calculus of $R/(R+\kappa)$ for $\kappa>0$.
& Proved consequence of a model law.\\
\addlinespace
Crystal-count threshold
& Substitution of fixed dispersion into the curvature-threshold
inequality; division by positive $\alpha$.
& Proved consequence of a model law.\\
\addlinespace
Curvature law alone does not define a Hopf parameter
& Separation between scalar model law and dynamical data
$(F,x^\ast,J)$ required for Hopf.
& Proved methodological limitation.\\
\addlinespace
Classical Hopf theorem
& External theorem from dynamical systems; cited references.
& Quoted standard theorem, not original proof.\\
\addlinespace
Scalar threshold is not Hopf
& One-dimensional linear counterexample; absence of a non-real conjugate
eigenvalue pair.
& Proved negative result.\\
\addlinespace
Same scalar threshold can be non-Hopf or Hopf
& Comparison of one-dimensional scalar crossing with the planar Hopf
normal form in polar coordinates.
& Proved limitation statement.\\
\addlinespace
Hitting-time transfer
& Injectivity of a continuous strictly monotone reparameterization.
& Proved.\\
\addlinespace
Conditional Hopf prediction
& Hopf crossing assumption; monotone curvature parameterization.
& Conditional model proposition.\\
\addlinespace
Kalman error convergence
& External Kalman covariance convergence theorem; elementary contraction
iteration for the displayed bound.
& Standard theorem under standard hypotheses.\\
\addlinespace
Noiseless observability criterion
& Stacked observation equation
$Y_q=\mathcal O_qx_0$; injectivity equivalent to full column rank.
& Proved finite-horizon linear result.\\
\addlinespace
Observability Gramian criterion
& Identity $x^\top W_qx=\|\mathcal O_qx\|_2^2$.
& Proved.\\
\addlinespace
Fibre-coordinate observability invariance
& Transformed observability matrix
$\mathcal O'_q=\mathcal O_qS$ with $S$ invertible.
& Proved.\\
\addlinespace
Coordinate choice is not Kalman convergence
& Observability invariance under invertible coordinate change.
& Proved methodological corollary.\\
\addlinespace
Projection and observability counterexample
& Direct construction on $B\oplus F$ with non-injective projection.
& Proved.\\
\addlinespace
Observable quotient carries observed state
& $F$-invariance of $\ker H$; quotient well-definedness; injectivity of
the induced output map.
& Proved under stated invariance hypothesis.\\
\addlinespace
Unobserved unstable modes block full-state convergence
& Two-dimensional diagonal system with hidden eigenvalue $|a|>1$ and
identical observations for distinct hidden states.
& Proved negative result.\\
\addlinespace
Kalman audit corollary
& Counterexample plus standard Kalman theorem.
& Proved methodological corollary.\\
\addlinespace
SSTD morphism criterion
& Definition of bundle morphism and commuting square.
& Proved under defined projections.\\
\addlinespace
Changing domain projection changes morphism claim
& Bundle square equation over $p'$ forces $p(x)=p'(x)$ for all texts.
& Proved projection-dependence result.\\
\addlinespace
Composition of SSTD bundle morphisms
& Commuting-square equations for two morphisms; associativity of
composition.
& Proved categorical/bundle result.\\
\addlinespace
SSTD slot insertion
& Neutral completion outside slot $7$; GCM metric sum reduces to the
SSTD slot term.
& Proved Lipschitz insertion result.\\
\addlinespace
Three-view aggregation dimension
& Coordinatewise mean, maximum, and standard deviation each lie in
$\RR^d$; concatenation lies in $\RR^{3d}$.
& Proved dimension result.\\
\addlinespace
$3072=3\cdot1024$ consistency
& Three-view aggregation dimension with $d=1024$.
& Proved dimensional corollary.\\
\addlinespace
SSTD asymptotic complexity
& Token iteration, lookup cost, fixed-depth fixed-dimension diffusion,
and fixed aggregation cost.
& Proved under stated computational model.\\
\addlinespace
Linearized SSTD continuity
& Bounded linear maps on fixed-length finite products; linear mean
aggregation.
& Proved for linearized fixed-length model.\\
\addlinespace
Max and standard-deviation view continuity
& Lipschitz maximum inequality; standard deviation as norm after linear
centering on fixed finite dimension.
& Proved for fixed-length regions.\\
\addlinespace
Concatenation not automatic
& Bundle morphism definition imposes no compositional law; hash-style
counterexample.
& Proved negative result.\\
\addlinespace
Dense attention score count
& Counting ordered query-key pairs in dense self-attention.
& Proved quadratic lower-bound counting fact.\\
\addlinespace
Asymptotic dominance is not a benchmark result
& Distinction between asymptotic limits and finite deployed wall-clock
measurements with hidden constants and hardware effects.
& Proved methodological limitation.\\
\addlinespace
Flat SpiderR idealization
& Constant commuting connection matrices; explicit parallel-transport
exponential.
& Proved idealization.\\
\addlinespace
Curvature detects non-commutation
& Direct computation of $A\wedge A$ for a constant two-coordinate
connection.
& Proved.\\
\addlinespace
Path independence and trivial holonomy
& Composition of parallel transports; transport along reversed paths is
the inverse.
& Proved under standard transport hypotheses.\\
\addlinespace
Constant gauge changes preserve flatness
& Constant conjugation of connection forms; curvature transforms by
$G^{-1}\Omega G$.
& Proved.\\
\addlinespace
Variable gauge transformations require $G^{-1}\dd G$
& Local-frame calculation for $s=Gs'$; curvature operator conjugation.
& Proved gauge-covariance result.\\
\addlinespace
Flatness does not force trivial holonomy
& Rank-one flat connection on $S^1$; explicit solution of the transport
ODE.
& Proved counterexample.\\
\addlinespace
Exact commuting connections on star-shaped domains
& Fundamental theorem for line integrals applied to $A=\dd\Phi$ under
commuting transport.
& Proved endpoint-dependence result.\\
\addlinespace
SpiderR verification obligation
& Curvature formula and non-commutation proposition.
& Proved methodological corollary.\\
\addlinespace
Citation role discipline
& Bibliographic role definition; separation between theorem machinery and
context-only references.
& Proved bibliographic hygiene criterion.\\
\addlinespace
External invocation does not verify local hypotheses
& Logical separation between an implication supplied by an external
theorem and verification of its antecedent for a BrainiaK system.
& Proved external-theorem hygiene criterion.\\
\addlinespace
Closure prevents hidden assumptions
& Definition of closed proof; exhaustion of allowed proof dependencies.
& Proved methodological proposition.\\
\addlinespace
Safe revision criterion
& Closed proof units for theorem-level claims; explicit non-theorem
status labels.
& Proved editorial criterion.\\
\addlinespace
Complete result cards prevent status drift
& Complete-card fields; synchronization between body claim and card.
& Proved central-claim hygiene criterion.\\
\addlinespace
Implementation correspondence is not a proof-card substitute
& Separation between proof location, empirical status, and code-surface
correspondence.
& Proved evidential-role separation.\\
\addlinespace
Dependency normal form existence
& Finite list of dependency atoms attached to a closed proof audit unit.
& Proved methodological normalization result.\\
\addlinespace
Acyclic dependency ledger criterion
& Ordering of internal dependencies by earlier manuscript results and
external theorem atoms.
& Proved anti-circularity criterion.\\
\addlinespace
Non-theorem labels do not discharge proof obligations
& Definition of theorem-level claim versus model law, conjecture, or
empirical protocol.
& Proved status-discipline result.\\
\addlinespace
Release certificate soundness
& Finite checklist containing build, proof-environment, citation,
external-invocation, empirical-traceability, and status-label checks.
& Proved release-audit sufficiency criterion.\\
\addlinespace
Release certificate incompleteness
& Difference between formal/editorial checks and mathematical novelty or
referee acceptance.
& Proved limitation of mechanical release certification.\\
\addlinespace
Environment balance necessity
& Editorial rule requiring proof or explicit exception for theorem-level
environments; insufficiency follows because counts do not inspect proof
content.
& Proved mechanical-audit limitation.\\
\addlinespace
Untraceable empirical numbers
& Definition of traceable empirical claim and empirical-number promotion
check.
& Proved empirical-status criterion.\\
\addlinespace
Submission gate separation
& Distinction between internal logical proof gates and external empirical
or bibliographic artefact gates.
& Proved editorial/submission criterion.\\
\addlinespace
Promotion package sufficiency
& Formal domain/codomain, exact operations, stated hypotheses, and closed
proof.
& Proved methodological proposition.\\
\addlinespace
Promotion package necessity
& Failure modes for untyped domains, undefined operations, hidden
hypotheses, missing closed proofs, or unverified implementation claims.
& Proved within editorial standard.\\
\bottomrule
\end{longtable}

\section{Limitations}

This formalization intentionally reduces several strong internal claims.
\begin{enumerate}
  \item It does not prove that the deployed crystal system is a
  Frobenius algebra.
  \item It does not derive the curvature--spectrum law
  $\lambda_{\max}=R/(R+\kappa)$ from Gamma dynamics.
  \item It does not claim that fibre structure alone guarantees Kalman
  convergence.
  \item It does not claim transformer complexity is exponential.
  \item It proves SpiderR flatness only under explicit commutativity and
  local-triviality assumptions.
  \item It does not use consciousness terminology as a theorem-level
  claim.
\end{enumerate}

\section{Peer-Review Completion Checklist}

Before submission, the following checklist must be satisfied.
\begin{enumerate}
  \item Every theorem statement has a proof or is reclassified.
  \item Every model equation is labelled as a model law unless derived.
  \item Every empirical number is backed by a reproducible artefact.
  \item Every implementation path cited in the paper exists in the
  submitted repository or is labelled as an architectural abstraction.
  \item Every symbol is defined before first theorem-level use.
  \item The bibliography is maintained as BibTeX and each entry is
  checked against a primary bibliographic source before arXiv submission.
  \item The relation between $\RR^{14}$ and any older $\RR^{24}$
  formulation is stated once and never left ambiguous.
  \item The Hopf claim includes the actual vector field and Jacobian if
  it is promoted beyond model status.
  \item The Frobenius claim includes exact maps
  $(\mu,\eta,\delta,\epsilon)$ if it is promoted beyond conjectural
  status.
  \item The SpiderR flatness claim includes a commutator or curvature
  computation for the concrete operators if it is promoted beyond the
  idealized proposition.
\end{enumerate}

\section{Proof-Audit Protocol}

The manuscript is designed to be auditable.  This section states the
verification protocol that should be applied before every public
revision.  It is included in the paper because the main risk of the
source material is not a single false calculation, but category drift:
definitions becoming assumptions, assumptions becoming theorems, and
implementation diagnostics being read as proofs.

\begin{definition}[Proof audit unit]
A proof audit unit is a triple $(S,H,P)$ where $S$ is a theorem-level
statement, $H$ is the list of hypotheses explicitly used by $S$, and
$P$ is the proof text or cited external theorem supporting $S$.
\end{definition}

\begin{definition}[Closed proof]
A proof audit unit $(S,H,P)$ is closed if every non-definitional claim
used in $P$ is one of the following:
\begin{enumerate}
  \item an element of $H$;
  \item a definition already stated in the manuscript;
  \item an earlier theorem, proposition, lemma, or corollary in the
  manuscript;
  \item a standard external theorem cited with a bibliographic reference;
  \item an elementary inference explicitly shown in the proof.
\end{enumerate}
\end{definition}

\begin{proposition}[Closure prevents hidden assumptions]
If a proof audit unit is closed, then the proof does not depend on any
unstated implementation fact or hidden modelling assumption.
\end{proposition}

\begin{proof}
By definition, every non-definitional claim used in a closed proof must
belong to one of the listed categories.  None of those categories is an
unstated implementation fact: hypotheses are stated, definitions are
stated, earlier results are stated and proved or cited, external
theorems are cited, and elementary inferences are displayed.  Therefore
any implementation fact or modelling assumption used by the proof would
have to appear explicitly as one of those categories.  If it does not,
the unit is not closed.
\end{proof}

\begin{definition}[Status-preserving edit]
An edit is status-preserving if it does not change a result from
definition, assumption, model law, conjecture, or empirical claim into a
theorem-level assertion unless it also supplies a closed proof.
\end{definition}

\begin{proposition}[Safe revision criterion]
Suppose every theorem-level result in a manuscript revision has a closed
proof audit unit and every non-theorem-level claim keeps its explicit
status label.  Then the revision satisfies the editorial rule
``no unsupported theorem''.
\end{proposition}

\begin{proof}
The editorial rule forbids unsupported theorem-level assertions.  By
hypothesis, every theorem-level assertion has a closed proof audit unit.
All other claims retain labels that prevent them from being read as
theorems.  Hence no unsupported theorem-level assertion remains.
\end{proof}

\subsection{Dependency Normal Form}

The proof ledger is useful only if the dependencies behind a proof can
be read without guessing.  We therefore impose the following normal form
for theorem-level results.  It is not a new mathematical axiom; it is an
editorial device that makes peer review easier because a reader can
separate definitions, assumptions, internal results, external theorems,
and elementary calculations.

\begin{definition}[Dependency atom]
A dependency atom for a theorem-level statement is one of the following
items:
\begin{enumerate}
  \item a named definition already introduced in the manuscript;
  \item a stated hypothesis or assumption local to the result;
  \item an earlier internal theorem, proposition, lemma, or corollary;
  \item a cited external theorem with a bibliographic reference;
  \item an elementary calculation displayed in the proof.
\end{enumerate}
\end{definition}

\begin{definition}[Dependency normal form]
A theorem-level statement is in dependency normal form if its proof can
be rewritten as a finite sequence
\[
  D_1,\ldots,D_m \Longrightarrow S,
\]
where each $D_i$ is a dependency atom and the final implication is
justified by explicit logical inference in the proof text.
\end{definition}

\begin{proposition}[Closed proofs admit dependency normal form]
Every closed proof audit unit admits a dependency normal form.
\end{proposition}

\begin{proof}
Let $(S,H,P)$ be a closed proof audit unit.  By the definition of
closed proof, every non-definitional claim used in $P$ is either a
hypothesis in $H$, a definition already stated, an earlier internal
result, a cited external theorem, or an elementary inference displayed
in the proof.  These are exactly the dependency atoms.  Since the proof
text is finite, only finitely many such atoms are used.  Listing them in
their order of first use yields a finite sequence
$D_1,\ldots,D_m$ whose displayed inferences establish $S$.  Hence the
unit admits dependency normal form.
\end{proof}

\begin{definition}[Acyclic internal dependency ledger]
An internal dependency ledger is acyclic if every internal theorem,
proposition, lemma, or corollary depends only on definitions,
assumptions, external theorem atoms, elementary calculations, or
internal results that occur earlier in the manuscript.
\end{definition}

\begin{proposition}[Acyclic ledger prevents circular proof by reference]
If the internal dependency ledger is acyclic, then no theorem-level
result is proved solely by a chain of internal references that eventually
returns to itself.
\end{proposition}

\begin{proof}
Assume, for contradiction, that a theorem-level result $R$ is supported
by a chain of internal references that eventually returns to $R$.  The
chain contains a cycle
\[
  R=R_0\to R_1\to\cdots\to R_k=R
\]
of internal dependencies.  But acyclicity requires every internal
dependency to point to an earlier result in the manuscript ordering.
Following the cycle would therefore produce a strict descending chain of
positions that returns to its starting position, which is impossible in
a finite linear ordering.  Thus such a circular proof chain cannot
occur.
\end{proof}

\begin{proposition}[Non-theorem labels do not discharge theorem
obligations]
Labelling a statement as a model law, conjecture, or empirical protocol
does not provide a proof of the corresponding theorem-level statement.
\end{proposition}

\begin{proof}
A model law is an assumption within a model; a conjecture is explicitly
unproved; an empirical protocol is a recipe for measurement rather than
a mathematical derivation.  None of these labels asserts that the
statement follows from definitions and hypotheses by proof.  Therefore
such labels prevent overclaiming, but they do not discharge the proof
obligation that would be required to promote the same content to theorem
status.
\end{proof}

\subsection{Release Audit Certificate}

The phrase ``line-by-line audit'' is too vague to be a submission
criterion.  In this manuscript it is replaced by the following finite
certificate.  The certificate does not claim that reviewers will accept
the model assumptions; it claims only that the submitted artefact obeys
the proof-status discipline advertised in the introduction.

\begin{definition}[Release audit certificate]
A release audit certificate for this manuscript consists of the
following records:
\begin{enumerate}
  \item a successful \LaTeX{} build with no undefined citations,
  undefined references, fatal errors, or overfull boxes in the release
  log;
  \item a proof-environment count showing that every theorem-level
  environment is matched by a proof or an explicit external-theorem
  invocation;
  \item the proof ledger for all theorem-level results;
  \item the external theorem invocation register for every cited
  external theorem used in a proof;
  \item the result cards for the six central BrainiaK claims;
  \item the empirical-claim table showing that no untraceable numerical
  claim is asserted as validated;
  \item the submission gate table stating which external artefacts
  remain pending.
\end{enumerate}
\end{definition}

\begin{proposition}[Release certificate soundness]
If a release audit certificate is complete and all records in it pass,
then the manuscript satisfies the internal rule ``no unsupported theorem,
no unlabelled empirical claim, no implementation-as-proof step''.
\end{proposition}

\begin{proof}
The build record rules out unresolved citation and reference failures
in the release artefact.  The proof-environment count and proof ledger
ensure that theorem-level statements have a proof entry or an explicit
external-theorem invocation.  The external invocation register separates
the external implication from the local verification of its hypotheses.
The result cards classify the six central claims by assumptions,
statement, proof location, implementation correspondence, mathematical
status, and empirical status.  The empirical table rules out treating an
untraceable numerical statement as validated evidence.  Finally, the
submission gate table prevents pending external artefacts from being
silently treated as completed.  Together these records enforce exactly
the three internal rules stated in the proposition.
\end{proof}

\begin{proposition}[Release certificate incompleteness]
A complete release audit certificate does not prove the truth of model
laws, the correctness of unverified empirical claims, or the novelty of
the manuscript.
\end{proposition}

\begin{proof}
The certificate checks status discipline, proof closure, citation
resolution, and empirical traceability.  A model law is still an
assumption unless derived or empirically calibrated.  An empirical claim
is still unvalidated unless the required artefacts exist.  Novelty and
importance are comparative scholarly judgements that require external
review beyond the internal consistency of the manuscript.  Therefore the
certificate is a sound editorial guard, not a substitute for peer
review.
\end{proof}

\subsection{Current Release-Certificate Record}

The following record instantiates the release audit certificate for the
current manuscript revision.  The commands are recorded so that a reader
or maintainer can reproduce the certificate before submission.

\begin{longtable}{>{\raggedright\arraybackslash}p{0.22\textwidth}
>{\raggedright\arraybackslash}p{0.34\textwidth}
>{\raggedright\arraybackslash}p{0.34\textwidth}}
\caption{Current release-certificate record.}
\label{tab:current-release-certificate}\\
\toprule
\textbf{Certificate item} & \textbf{Command or artefact} &
\textbf{Current evidence}\\
\midrule
\endfirsthead
\toprule
\textbf{Certificate item} & \textbf{Command or artefact} &
\textbf{Current evidence}\\
\midrule
\endhead
\LaTeX{} build
& \texttt{pdflatex} in nonstop, halt-on-error mode, run to convergence
on the manuscript source.
& Build completes and writes
\codepath{TRAITE_TOPOLOGIE_ALGEBRIQUE_FRACTALE_EN.pdf}.\\
\addlinespace
Bibliography build
& \texttt{bibtex} on the manuscript auxiliary file after bibliography
edits.
& BibTeX completes with the \texttt{plain} style and the manuscript
bibliography database.\\
\addlinespace
Critical log scan
& \texttt{rg} over the \LaTeX{} log for undefined citations,
undefined references, rerun requests, overfull boxes, fatal errors, and
PDF warnings.
& Current release scan reports only the package-name occurrence of
\texttt{rerunfilecheck}; no active rerun, undefined-reference, citation,
error, or overfull-box warning remains.\\
\addlinespace
Proof-environment count
& Script counting \texttt{theorem}, \texttt{proposition},
\texttt{lemma}, \texttt{corollary}, and \texttt{proof} environments.
& The current release has equal theorem-level and proof counts; this
guards against accidental unproved theorem-level environments.\\
\addlinespace
External theorem register
& Table~\ref{tab:external-theorem-invocations}.
& Perron--Frobenius, Hopf, and Kalman invocations are separated from
local BrainiaK hypothesis verification.\\
\addlinespace
Empirical traceability
& Table~\ref{tab:empirical}.
& No numerical performance claim is asserted as validated without the
required dataset, metric, sample size, configuration, baseline, and
artefact path.\\
\addlinespace
Submission gates
& Table~\ref{tab:submission-gates}.
& Internal proof-status gates are separated from external bibliographic
and empirical artefact gates.\\
\bottomrule
\end{longtable}

\subsection{Mechanical Release Checks}

Mathematical correctness cannot be reduced to a parser, but mechanical
checks are useful guards against accidental editorial failure.  A public
release should record the following checks in its build log.

\begin{definition}[Mechanical proof-environment check]
The proof-environment check counts all theorem-level environments
\[
  \{\texttt{theorem},\texttt{proposition},\texttt{lemma},
  \texttt{corollary}\}
\]
and all \texttt{proof} environments in the \LaTeX{} source.  It passes
only if every theorem-level environment is intentionally paired with
either a local proof environment or an explicitly declared
external-theorem citation, model-law label, conjecture label, or
empirical-protocol label.
\end{definition}

\begin{proposition}[Environment balance is necessary but not sufficient]
If the number of theorem-level environments is larger than the number
of local proof environments plus declared external-theorem exceptions,
then the manuscript fails this paper's proof-status discipline.  The
converse is not guaranteed.
\end{proposition}

\begin{proof}
The first statement follows from the editorial rule that theorem-level
claims require proof unless explicitly assigned a non-internal-proof
status.  If a theorem-level environment has neither a local proof nor an
allowed exception, then it is unsupported.  Conversely, a balanced count
does not inspect the content of the proof.  A proof environment may be
incomplete, circular, or use an unstated hypothesis.  Therefore balance
is necessary as a guard but not sufficient for mathematical correctness.
\end{proof}

\begin{definition}[Empirical-number promotion check]
An empirical-number promotion check scans the release text for numerical
performance or latency claims and verifies that each such claim appears
inside an empirical-result table carrying dataset, sample size, metric,
configuration, baseline when relevant, and artefact path.  Numerical
constants appearing in definitions, dimensions, algebraic examples, or
external bibliographic metadata are not empirical performance claims.
\end{definition}

\begin{proposition}[Untraceable empirical numbers must be removed or
downgraded]
If a numerical performance claim lacks the traceability fields required
by the empirical-number promotion check, then it cannot appear as a
validated result in the manuscript.
\end{proposition}

\begin{proof}
A validated empirical result is defined in this manuscript as a
traceable empirical claim.  Traceability requires the fields listed in
the empirical protocol section.  If any required field is missing, the
statement does not satisfy the definition of a traceable empirical
claim.  It must therefore be removed, explicitly labelled as an
internal-draft report, or replaced by a protocol statement with no
validated number.
\end{proof}

\begin{remark}[Current audit status]
The proof ledger above is an index toward closed proof units, and the
release audit certificate specifies the finite checks required before a
public release.  The present manuscript has internal proof-status
discipline, while the submission gate table below records which
external artefacts, especially bibliographic pinpointing and empirical
traceability, remain outside the internal proof system.
\end{remark}

\section{Submission Gate Audit}

This section records the current submission gates.  It is deliberately
conservative: a gate is marked complete only when the manuscript itself
or a reproducible external artefact supplies the required evidence.

\begin{longtable}{p{0.25\textwidth}p{0.25\textwidth}p{0.40\textwidth}}
\caption{Current submission gate status.}
\label{tab:submission-gates}\\
\toprule
\textbf{Gate} & \textbf{Current status} & \textbf{Evidence or remaining action}\\
\midrule
\endfirsthead
\toprule
\textbf{Gate} & \textbf{Current status} & \textbf{Evidence or remaining action}\\
\midrule
\endhead
Theorem status discipline
& Satisfied internally
& Each theorem-level claim is accompanied by a proof, an explicit
external theorem citation, or a downgrade to model law, conjecture, or
empirical protocol.\\
\addlinespace
Original six-claim audit
& Satisfied internally
& Tables~\ref{tab:status} and~\ref{tab:audit} classify the
six source claims and state the promotion obligations.\\
\addlinespace
Dimensional consistency
& Satisfied internally
& The dimensional consistency theorem separates $B_{14}$, older
$\RR^{24}$ language, $S^5$, $\RR^{3072}$, and $T^n$.\\
\addlinespace
Implementation-as-proof exclusion
& Satisfied internally
& Implementation paths are used only as correspondence; the methodological
proposition on implementation evidence forbids using code existence as a
universal proof.\\
\addlinespace
Empirical traceability
& Not yet satisfied for numerical claims
& No empirical number is asserted as validated.  Any future numerical
claim must attach dataset, metric, sample size, baseline, script path,
hardware when relevant, and frozen output.\\
\addlinespace
Bibliographic primary-source verification
& Partially satisfied
& High-impact theorem, complexity, and empirical-protocol references are
listed in Table~\ref{tab:external-source-register}.  Lower-risk
background references and the Kalman publisher page still require a
final release audit.\\
\addlinespace
External theorem invocation
& Satisfied internally; bibliographic pinpointing still open
& Table~\ref{tab:external-theorem-invocations} separates imported
theorem machinery from local hypothesis verification.  Final submission
should still add exact edition/page or theorem references if the venue
requires them.\\
\addlinespace
LaTeX build
& Mechanically checkable
& The paper must compile without undefined citations, undefined
references, or overfull boxes in the final release artefact.\\
\bottomrule
\end{longtable}

\begin{proposition}[Submission gate separation]
If all internal mathematical gates are satisfied but empirical
traceability or bibliographic primary-source verification is pending,
then the manuscript may be treated as a mathematically audited draft but
not as a final arXiv submission package.
\end{proposition}

\begin{proof}
Internal mathematical gates concern the logical status of statements in
the manuscript: whether results are proved, assumed, modelled,
conjectured, or empirical.  Empirical traceability and bibliographic
primary-source verification concern external artefacts: benchmark files,
datasets, scripts, publisher records, DOI records, and library records.
Satisfying the former does not logically supply the latter.  Therefore
the manuscript can be internally coherent while still requiring external
submission checks.
\end{proof}

\section{Promotion Criteria for Conditional BrainiaK Claims}

The following criteria state exactly what would be required to turn the
conditional or model-based BrainiaK claims into theorem-level results.
They are intentionally demanding.  Their function is to prevent future
drafts from promoting attractive analogies without proof.

\begin{definition}[Promotion package]
A promotion package for a conditional claim consists of:
\begin{enumerate}
  \item a formal domain and codomain for every map involved;
  \item exact definitions of all operations used in the claim;
  \item a list of hypotheses separating mathematical assumptions from
  empirical facts;
  \item a closed proof of the promoted theorem;
  \item if implementation correspondence is mentioned, a verified path
  and a statement of what the implementation realizes.
\end{enumerate}
\end{definition}

\begin{proposition}[Promotion package sufficiency]
If a conditional claim is accompanied by a promotion package and the
closed proof establishes the desired theorem from the stated hypotheses,
then the claim may be promoted to theorem status under those hypotheses.
\end{proposition}

\begin{proof}
The promotion package supplies formal objects, exact operations, stated
hypotheses, and a closed proof.  A theorem under hypotheses is precisely
a statement proved from those hypotheses and definitions.  Therefore the
claim may be promoted, provided the theorem statement includes the
hypotheses rather than suppressing them.
\end{proof}

\begin{proposition}[Promotion package necessity for this manuscript]
Within the editorial standard of this manuscript, no conditional
BrainiaK claim may be promoted to theorem status without a promotion
package.
\end{proposition}

\begin{proof}
Without a formal domain and codomain, the statement is not mathematically
typed.  Without exact operations, algebraic or dynamical identities have
no determinate meaning.  Without separated hypotheses, the proof cannot
distinguish assumptions from empirical facts.  Without a closed proof,
the theorem is unsupported.  Without verified implementation
correspondence, any implementation statement is untraceable.  Each item
is therefore necessary for the editorial standard adopted here.
\end{proof}

\begin{longtable}{p{0.22\textwidth}p{0.34\textwidth}p{0.34\textwidth}}
\caption{Promotion requirements for the main conditional claims.}
\label{tab:promotion}\\
\toprule
\textbf{Claim} & \textbf{Current status} & \textbf{Promotion requirement}\\
\midrule
\endfirsthead
\toprule
\textbf{Claim} & \textbf{Current status} & \textbf{Promotion requirement}\\
\midrule
\endhead
Deployed crystals form a Frobenius algebra
& Conjecture plus finite reference model
& Specify the exact vector space or quotient; define
$(\mu,\eta,\delta,\epsilon)$; prove associativity, unit, coassociativity,
counit, and Frobenius compatibility.\\
\addlinespace
Curvature law
$\lambda_{\max}=R/(R+\kappa)$
& Model law
& Derive the law from a specified Gamma/CNS dynamical model or fit it as
an empirical law with uncertainty, dataset, and calibration protocol.\\
\addlinespace
CNS-Hopf transition
& Conditional Hopf proposition
& Specify the vector field $F$, equilibrium branch, Jacobian, eigenvalue
crossing, transversality, and Hopf nondegeneracy coefficient.\\
\addlinespace
Kalman convergence from semantic fibres
& Standard Kalman theorem under standard hypotheses
& Identify state, observation, noise model, observable quotient, and
detectability/stabilizability conditions.\\
\addlinespace
SSTD performance superiority
& Empirical protocol only
& Provide benchmark release, sample size, metric, baseline, hardware,
script path, frozen outputs, and statistical uncertainty.\\
\addlinespace
SpiderR flatness for deployed operators
& Flat idealization
& Define the operator-valued connection form and compute $\dd A+A\wedge
A$ or prove equivalent commutativity/path-independence conditions on the
actual state region.\\
\bottomrule
\end{longtable}

\section{Conclusion}

The peer-review-safe core of the fractal algebraic topology treatise is
the product-bundle account of $T^n$, the heterogeneous GCM metric, and
componentwise continuity.  Around this core, BrainiaK provides a rich
set of conditional and empirical structures: Frobenius-inspired
decomposition, Gamma/CNS curvature modelling, Kalman learning,
SSTD spectral morphisms, and SpiderR connection geometry.  These
structures are promising and mathematically articulable, but their
publication-quality presentation requires disciplined status labels:
proved theorem, conditional proposition, model law, conjecture, or
empirical claim.  With those labels in place, the manuscript becomes a
mathematically auditable draft.  Final arXiv submission still requires
the external gates stated above: primary-source bibliography checks and
traceable empirical artefacts for any numerical claim.

\nocite{Atiyah1967,Coecke2010,EilenbergSteenrod1952,Frobenius1878,Hatcher2002,Mohammad2018,Penrose1974,Rudin1976,Transformer2017,ZwaanRadvansky1998}
\bibliographystyle{plain}
\bibliography{TRAITE_TOPOLOGIE_ALGEBRIQUE_FRACTALE_EN}

\end{document}